\documentclass[reqno]{amsart}

\usepackage{amsmath}
\usepackage{amssymb}
\usepackage{amsfonts}
\usepackage{amsthm}
\usepackage[foot]{amsaddr}
\usepackage{mathtools}
\mathtoolsset{%
}

\usepackage[utf8]{inputenc}
\usepackage[T1]{fontenc}

\usepackage[sf,mono=false]{libertine}

\usepackage{dsfont}

\usepackage[%
cal=cm,
]
{mathalfa}

\usepackage[dvipsnames,svgnames]{xcolor}
\colorlet{MyBlue}{DodgerBlue!75!Black}
\colorlet{MyGreen}{DarkGreen!85!Black}

\usepackage[font=small,labelfont=bf]{caption}
\usepackage{subfigure}
\usepackage{tikz}
\usetikzlibrary{calc}

\usepackage{acronym}
\usepackage{latexsym}
\usepackage{paralist}
\usepackage{wasysym}
\usepackage{xspace}

\usepackage[multiple]{footmisc}

\usepackage[sort&compress]{natbib}

\newcommand{\citegen}[2][]{\citeauthor{#2}'s \textup(\citeyear[#1]{#2}\textup)}

\usepackage{hyperref}
\hypersetup{
draft=false,
colorlinks=true,
linktocpage=true,
pdfstartview=FitH,
breaklinks=true,
pdfpagemode=UseNone,
pageanchor=true,
pdfpagemode=UseOutlines,
plainpages=false,
bookmarksnumbered,
bookmarksopen=false,
bookmarksopenlevel=1,
hypertexnames=true,
pdfhighlight=/O,
urlcolor=MyBlue!60!black,linkcolor=MyBlue!70!black,citecolor=DarkGreen!70!black, 
pdftitle={},
pdfauthor={},
pdfsubject={},
pdfkeywords={},
pdfcreator={pdfLaTeX},
pdfproducer={LaTeX with hyperref}
}

\numberwithin{equation}{section}  
\usepackage[sort&compress,capitalize,nameinlink]{cleveref}

\crefrangeformat{equation}{\upshape(#3#1#4)\textendash(#5#2#6)}

\newcommand{\R}{\mathbb{R}}

\DeclareMathOperator*{\argmax}{arg\,max}
\DeclareMathOperator*{\argmin}{arg\,min}
\DeclareMathOperator{\bd}{bd}

\DeclareMathOperator{\cl}{cl}
\DeclareMathOperator{\diag}{diag}

\DeclareMathOperator{\dist}{dist}
\DeclareMathOperator{\dom}{dom}

\DeclareMathOperator{\grad}{grad\hspace{-1pt}}

\DeclareMathOperator{\hess}{Hess}
\DeclareMathOperator{\im}{im}

\DeclareMathOperator{\intr}{int}

\DeclareMathOperator{\proj}{proj}

\DeclareMathOperator{\range}{range}

\DeclareMathOperator{\supp}{supp}

\newcommand{\argdot}{\mathopen{}\cdot\mathopen{}}

\newcommand{\bvec}{e}

\newcommand{\onerow}{\mathds{1}}
\newcommand{\onecol}{\mathbf{1}}

\newcommand{\dd}{\:d}

\newcommand{\from}{\colon}

\newcommand{\open}{U}
\newcommand{\orthant}{\clorthant^{\circ}}
\newcommand{\clorthant}{\mathcal{K}}
\newcommand{\pd}{\partial}

\newcommand{\simplex}{\Delta}

\newcommand{\wilde}{\widetilde}

\DeclarePairedDelimiter{\bracks}{[}{]}

\DeclarePairedDelimiter{\pospart}{[}{]_{+}}

\DeclarePairedDelimiter{\abs}{\lvert}{\rvert}
\DeclarePairedDelimiter{\norm}{\lVert}{\rVert}

\DeclarePairedDelimiterX{\braket}[2]{\langle}{\rangle}{#1\mathopen{}\delimsize\vert\mathopen{}#2}
\DeclarePairedDelimiterX{\product}[2]{\langle}{\rangle}{#1,#2}
\DeclarePairedDelimiterX{\setdef}[2]{\{}{\}}{#1:#2}

\newcommand{\txs}{\textstyle}
\newcommand{\textpar}[1]{\textup(#1\textup)}

\newcommand{\insum}{\sum\nolimits}

\usepackage[textwidth=30mm]{todonotes}
\usepackage{soul}
\setstcolor{red}
\sethlcolor{SkyBlue}

\newcommand{\ie}{i.e.,\xspace}
\newcommand{\eg}{e.g.,\xspace}

\theoremstyle{plain}
\newtheorem{theorem}{Theorem}

\newtheorem*{corollary*}{Corollary}
\newtheorem{lemma}[theorem]{Lemma}
\newtheorem{proposition}[theorem]{Proposition}

\newtheorem{observation}[theorem]{Observation}

\theoremstyle{definition}

\newtheorem*{definition*}{Definition}

\theoremstyle{remark}
\newtheorem{remark}{Remark}
\newtheorem*{remark*}{Remark}
\newtheorem{example}{Example}
\newtheorem*{example*}{Example}

\numberwithin{equation}{section}
\numberwithin{theorem}{section}
\numberwithin{remark}{section}
\numberwithin{example}{section}

\newcommand{\pure}{\alpha}
\newcommand{\purealt}{\beta}
\newcommand{\purealtalt}{\gamma}
\newcommand{\nPures}{n}
\newcommand{\pures}{\mathcal{A}}

\newcommand{\payv}{v}
\newcommand{\strat}{\mathcal{X}}
\newcommand{\game}{\mathcal{G}}
\newcommand{\gamefull}{\game(\pures,\payv)}

\newcommand{\eq}{x^{\ast}}
\newcommand{\bareq}{\bar x^{\ast}}
\newcommand{\olimit}{x^{\omega}}

\newcommand{\domain}{\mathcal{D}}

\newcommand{\dynfield}{V}
\newcommand{\pot}{f}

\newcommand{\Hop}{M}
\newcommand{\HopFcn}{M}

\newcommand{\hessmat}{H}
\newcommand{\zproj}{\Phi}

\newcommand{\eigvec}{u}
\newcommand{\ttop}{{\!\top\!}}

\newcommand{\eye}{I}
\newcommand{\identity}{\eye}
\newcommand{\diagidentity}{\diag(1,\dotsc,1)}

\DeclareMathOperator{\cost}{C}
\DeclareMathOperator{\gain}{G}

\newcommand{\scprod}{g}

\newcommand{\base}{x^{\ast}}
\newcommand{\notbase}{x}

\newcommand{\vecspace}{W}
\newcommand{\dspace}{\vecspace^{\ast}}

\newcommand{\breg}{D_{h}}

\newcommand{\legendre}{y}

\newcommand{\dkl}{D_{\textup{KL}}}
\newcommand{\deucl}{D_{\textup{Eucl}}}
\newcommand{\good}{\domain}

\newcommand{\unit}{\mathcal{U}}

\newcommand{\choice}{Q}
\newcommand{\widebar}{\bar}
\newcommand{\Lyap}{L}
\newcommand{\rate}{\phi}

\newcommand{\choicef}{Q_\clorthant}

\DeclareMathOperator{\tproj}{\Pi}

\DeclareMathOperator{\tcone}{TC}
\DeclareMathOperator{\tspace}{T}
\DeclareMathOperator{\ncone}{NC}

\newcommand{\zspace}{\R_{0}}
\newcommand{\metcone}{\textrm{Adm}_{g}}
\newcommand{\domg}{\dom g}

\newcommand{\normal}{n}
\newcommand{\intstrat}{\strat^{\circ}}

\DeclarePairedDelimiter{\class}{[}{]}

\newcommand{\wt}{\phi}

\DeclareMathOperator{\vspan}{span}

\newcommand{\undx}{z}
\newcommand{\hh}{h_{\Undx}}
\newcommand{\Undx}{\mathcal{Z}}
\newcommand{\intUndx}{\Undx^{\circ}}

\begin{document}
\allowdisplaybreaks


\title{Riemannian Game Dynamics}

\author[P.~Mertikopoulos]{Panayotis Mertikopoulos$^{\ast}$}
\address
{$^{\ast}$Univ. Grenoble Alpes, CNRS, Inria, LIG, F-38000 Grenoble, France.}
\email{\href{mailto:panayotis.mertikopoulos@imag.fr}{panayotis.mertikopoulos@imag.fr}}

\author[W.~H.~Sandholm]{William H. Sandholm$^{\S}$}
\address
{\quad$^{\S}$\hspace{.5pt}Department of Economics, University of Wisconsin, 1180 Observatory Drive, Madison WI 53706, USA.}
\email{\href{mailto:whs@ssc.wisc.edu}{whs@ssc.wisc.edu}}

\thanks{%
We thank Josef Hofbauer and Dai Zusai for helpful discussions and comments, and we thank Marciano Siniscalchi and an anonymous referee for very thoughtful reports.
Part of this work was carried out during the authors' visit to the Hausdorff Research Institute for Mathematics at the University of Bonn in the framework of the Trimester Program ``Stochastic Dynamics in Economics and Finance''.
PM is grateful for financial support from
the French National Research Agency (ANR) under grant no.~ANR\textendash GAGA\textendash13\textendash JS01\textendash 0004\textendash 01
and
the CNRS under grant no. PEPS\textendash REAL.net\textendash 2016.
WHS is grateful for financial support under NSF Grants SES\textendash1155135, SES\textendash1458992, and SES-1728853 and ARO Grant MSN201957.}


\newcommand{\acdef}[1]{\textit{\acl{#1}} \textup{(\acs{#1})}\acused{#1}}
\newcommand{\acdefp}[1]{\emph{\aclp{#1}} \textup(\acsp{#1}\textup)\acused{#1}}

\newacro{1SL}{one-sided Lipschitz}
\newacro{ESS}{evolutionarily stable state}
\newacro{GESS}{globally evolutionarily stable state}
\newacro{ODE}{ordinary differential equation}
\newacro{HR}{Hess\-i\-an Rie\-man\-ni\-an}
\newacro{KKT}{Ka\-rush\textendash Kuhn\textendash Tuc\-ker}
\newacro{RPS}{Rock-Paper-Scissors}
\newacro{MP}{Matching Pennies}
\newacro{KL}{Kull\-back\textendash Le\-ib\-ler}
\newacro{lsc}[l.s.c.]{lower semi-continuous}
\newacro{LHS}{left-hand side}
\newacro{RHS}{right-hand side}
\newacro{NE}{Nash equilibrium}
\newacroplural{NE}[NE]{Nash equilibria}

\begin{abstract}

We study a class of evolutionary game dynamics defined by balancing a \emph{gain} determined by the game's payoffs against a \emph{cost of motion} that captures the difficulty with which the population moves between states.
Costs of motion are represented by a Riemannian metric, \ie a state-dependent inner product on the set of population states.
The replicator dynamics and the (Euclidean) projection dynamics are the archetypal examples of the class we study.
Like these representative dynamics, all Riemannian game dynamics satisfy certain basic desiderata, including positive correlation and global convergence in potential games.
Moreover, when the underlying Riemannian metric satisfies a Hessian integrability condition, the resulting dynamics preserve many further properties of the replicator and projection dynamics.
We examine the close connections between Hessian game dynamics and reinforcement learning in normal form games, extending and elucidating a well-known link between the replicator dynamics and exponential reinforcement learning.
\end{abstract}

\maketitle



\section{Introduction}
\label{sec:introduction}

Viewed abstractly, evolutionary game dynamics assign to every population game a dynamical system on the game's set of population states.
Under most such dynamics, the vector of motion at a given population state depends only on payoffs and behavior at that state, implying that changes in aggregate behavior are determined by current strategic conditions.
Such dynamics may thus be viewed as state-dependent rules for transforming current payoffs into feasible directions of motion.

In this paper, we introduce a family of evolutionary game dynamics under which the vector of motion $z$ from any state $x$ is obtained by balancing two forces.
The first, the \emph{gain from motion}, is obtained by adding the products of the strategies' payoffs at $x$ with their rates of change under $z$.
This quantity is the measure of agreement between payoffs and motion used in the standard monotonicity condition for game dynamics.%
\footnote{See \cite{Fri91}, \cite{Swi93}, \cite{San01}, \cite{DemRit03}, and condition \eqref{eq:PC} below.}
The second, the \emph{cost of motion}, captures the difficulty with which the population moves from state $x$ along vector $z$;
different specifications of these
costs define different members of our family of dynamics.
These costs are usefully represented by means of a \emph{Riemannian metric}, a state-dependent inner product used to evaluate lengths of and angles between vectors of motion.
Accordingly, the dynamics studied here, defined by maximizing differences between gains and costs, are called \emph{Riemannian game dynamics}.

The two archetypal examples of Riemannian game dynamics are the replicator dynamics \citep{TJ78} and the (Euclidean) projection dynamics \citep{NZ97}, both derived from fairly simple structures.
First, the replicator dynamics are derived from the \emph{Shahshahani metric} \citep{Sha79}, under which the cost of increasing a strategy's relative frequency in the population is inversely proportional to said frequency.
Second, the projection dynamics are obtained by measuring the cost of motion
in the standard Euclidean fashion, independently of the population's current state.
Other Riemannian metrics can be used in applications where different strategies have clear affinities, allowing the presence and performance of one strategy to positively influence the use of similar alternatives.

The metric's boundary behavior is the source of a fundamental dichotomy that is best explained by looking at our two prototypical examples above.
Under the replicator dynamics:
\begin{inparaenum}%
[\textup(\itshape i\textup)]
\item
the law of motion for every game is continuous;
\item
the set of utilized strategies remains constant along every solution trajectory;
and
\item
the dynamics' rest points are the restricted equilibria of the game \textendash\ the states at which all strategies in use earn the same payoff.
\end{inparaenum}
In contrast, under the Euclidean projection dynamics:
\begin{inparaenum}%
[\textup(\itshape i\textup)]
\item
the law of motion is typically discontinuous at the boundary of the simplex;
\item
the set of utilized strategies may change infinitely often along the same solution trajectory;
and
\item
the dynamics' rest points are the \aclp{NE} of the underlying game.
\end{inparaenum}
Based on this behavior, we obtain a natural distinction between \emph{continuous} and \emph{discontinuous Riemannian dynamics}, each category sharing the boundary behavior of its prototype.
In Section \ref{sec:examples}, we introduce a variety of examples of Riemannian dynamics from both classes;
then, in \cref{sec:protocols}, we show how these and other Riemannian dynamics can be provided with microfoundations using suitably constructed revision protocols.


A basic aim of our analysis is to demonstrate that many basic properties of the replicator and Euclidean projection dynamics extend to our substantially more general setting.
In Section \ref{sec:analysis}, we show that Riemannian dynamics satisfy the basic desiderata for evolutionary game dynamics:
they heed a payoff monotonicity condition known as \emph{positive correlation}, and they converge globally
in the class of potential games.
In the latter context, Riemannian game dynamics also provide a broad generalization of \emph{Kimura's maximum principle} \citep{Kim58,Sha79}.
This principle states that when agents are matched to play a normal form common interest game, the replicator dynamics move in the direction of maximal increase in average payoffs, provided that lengths of displacement vectors are evaluated using the Shahshahani metric.
Extending this principle, we observe that Riemannian dynamics track the direction of
steepest ascent of potential
in any potential game,
provided
that displacements
are evaluated using the Riemannian metric at hand.

Obtaining further results on stability, convergence, and global behavior requires additional structure on our dynamics \textendash\ and hence on the underlying Riemannian metric.
This structure is provided by an \emph{integrability condition}.
In prior work on game dynamics, such conditions have been imposed on the \emph{vector fields} used to convert the strategies' payoffs into vectors of choice probabilities.%
\footnote{See \cite{HMC01b}, \cite{HS07}, and \cite{San10c}.}
By contrast, the integrability condition employed here is imposed on the \emph{matrix field} that defines a Riemannian metric,
requiring that it be expressible as the Hessian of a convex function.
We call this function the \emph{potential} of the metric, and we refer to the resulting dynamics as \emph{Hessian game dynamics}.%
\footnote{In the context of convex programming, gradient flows generated by \ac{HR} metrics of this sort have been explored at depth by \cite{BT03}, \cite{ABB04}, \cite{MS18}, and many others.
\cite{LM15} also examine the long-term rationality properties of a class of second-order, \emph{inertial} game dynamics derived from \ac{HR} metrics.}
Both the replicator dynamics and the Euclidean projection dynamics are members of this class.
As we explain in \cref{sec:HD}, Hessian dynamics are continuous when their potential function becomes infinitely steep at the boundary of the simplex,
leading to the distinction between \emph{continuous} and \emph{discontinuous} Hessian dynamics.

The key tool that we employ for the analysis of Hessian dynamics is the \emph{Bregman divergence} \citep{Bre67}, an asymmetric measure of the ``remoteness'' of a given population state from any fixed target state.%
\footnote{In the Shahshahani case, this boils down to the \acl{KL} divergence, which has seen wide use in the analysis of the replicator dynamics \citep{Wei95,HS98}.}
By using the Bregman divergence as a Lyapunov function, we prove global convergence to \acl{NE} in strictly contractive games and local stability of \aclp{ESS} under Hessian game dynamics.
We also show that certain distinctive properties of the replicator dynamics in normal form games extend to \emph{all} continuous Hessian dynamics \textendash\ in particular, the convergence of time averages of interior solutions to the set of Nash equilibria, and the existence of simple sufficient conditions for permanence.
Finally, we show that strictly dominated strategies are eliminated under continuous Hessian dynamics, a conclusion which does not extend to the discontinuous regime.%
\footnote{See \cite{SDL08} and Section \ref{sec:dominated}.}

\subsubsection*{Related work}

There are very close connections between the dynamics considered here and dynamics studied by \cite{HS90}, \cite{Hop99b}, and \cite{Har11}.
In order to have the machinery in place to make these connections clear, we postpone this discussion until Section \ref{sec:previous}.

There is a more surprising connection between Hessian dynamics and models of reinforcement learning in normal form games.
\cite{Rus99}, \cite{HSV09} and \cite{MM10} show that if players track the cumulative payoffs (or \emph{scores}) of their strategies and choose mixed strategies at each instant by applying the logit choice rule to these scores, the evolution of mixed strategies is described by the replicator dynamics.%
\footnote{For related results, see also \cite{BS97}, \cite{Pos97}, and \cite{Hop02}.}
Combining our analysis here with that
of
\cite{MS16}, we show that Hessian dynamics derived from a steep potential function also describe the evolution of mixed strategies under reinforcement learning. 
In addition to substantially generalizing existing results, our analysis provides an intuitive explanation for the tight links between the two processes. 
Section \ref{sec:RL} describes these and other connections between Hessian dynamics and reinforcement learning in detail.

\section{Population games and evolutionary dynamics}
\label{sec:prelims}


\subsection*{Notation}
\label{sec:notation}

Let $\pures = \{\pure_{1},\dotsc,\pure_{\nPures}\}$ be a finite set.
The real space spanned by $\pures$ will be denoted by $\R^{\pures}$ and we will write $\delta_{\pure\purealt}$ for the Kronecker deltas on $\pures$.
We will also write $\clorthant \equiv \R_{+}^{\pures}$ for the nonnegative orthant of $\R^{\pures}$, $\orthant \equiv \R_{++}^{\pures}$ for its interior (the positive orthant),
and
$\zspace^{\pures} = \setdef{z\in\R^{\pures}}{\insum_{\pure} z_{\pure} = 0}$ for the subspace of vectors whose components sum to zero.
Finally, in a slight abuse of notation, we will write $\R^{\supp(x)}= \setdef{z\in\R^{\pures}}{z_\pure = 0 \text{ whenever } x_\pure =0}$ for the set of vectors in $\R^{\pures}$ whose support is contained in the support of $x\in\R^{\pures}$.

\subsection{Population games}
\label{sec:games}

Throughout this paper we focus on games played by a population of nonatomic agents.
Our analysis extends to the multi-population setting without significant effort, but we focus on single-population games for simplicity and notational clarity. 

During play, each agent chooses an \emph{action} (or \emph{pure strategy}) from a finite set $\pures$, and their payoff is determined by their choice of action and by the proportions $x_{\pure}\in[0,1]$ of the population playing each action $\pure\in\pures$.
Collectively, these proportions define a \emph{population state} $x = (x_{\pure})_{\pure\in\pures}\in\R^{\pures}$,
and we write $\strat = \simplex(\pures)=\setdef{x\in\R^{\pures}_+}{\sum_\pure x_\pure =1}$ for the set of population states (or \emph{state space}) of the game.
The payoff to an agent playing $\pure\in\pures$ when the population state is $x\in\strat$ is given by an associated \emph{payoff function} $\payv_{\pure}\from \strat\to\R$, which we assume to be Lipschitz continuous.
Putting all this together, a \emph{population game} may be identified with a set of actions and their associated payoff functions, and will be denoted by $\game \equiv \game(\pures,\payv)$.

A population state $\eq\in\strat$ is a \acdef{NE} of a population game $\game$ if
\begin{equation}
\label{eq:Nash}
\tag{NE}
\payv_{\pure}(\eq)
	\geq \payv_{\purealt}(\eq)
	\quad
	\text{for all $\pure\in\supp(\eq)$ and for all $\purealt\in\pures$.}
\end{equation}
If $\eq$ satisfies \eqref{eq:Nash} and is \emph{pure} (i.e. $\eq = \bvec_{\pure}$ for some $\pure\in\pures$), it is called a \emph{pure \acl{NE}} of $\game$;
if, in addition, \eqref{eq:Nash} holds as a strict inequality for all $\purealt\notin\supp(\eq)$, $\eq$ is said to be a \emph{strict equilibrium} of $\game$.

A \emph{restriction} of a game $\game$ is a population game $\game' \equiv \game'(\pures',\payv')$ that is defined by a subset $\pures'\subseteq\pures$ of the original game's action set and by payoff functions $\payv_{\pure}$ obtained by restricting the original payoff functions to the reduced state space $\strat' =\simplex(\pures')$ of $\game'$.
If $x\in\strat$ is a \acl{NE} of some restriction of $\game$, it will be called a \emph{restricted equilibrium};
as such, $x \in \strat$ is a restricted equilibrium of $\game$ if all strategies in its support earn equal payoffs.

\begin{example}
[Matching in normal form games]
\label{ex:matching}
The simplest example of a population game is obtained by uniformly matching a population of agents to play a two-player symmetric normal form game with payoff matrix $A = (A_{\pure\purealt})_{\pure,\purealt=1}^{\nPures}$.
Aggregating over all matches, the payoff to an $\pure$-strategist when the population is at state $x\in\strat$ is $\payv_{\pure}(x) = \insum_{\purealt\in\pures} A_{\pure\purealt} x_{\purealt}$.
\end{example}


\begin{example}
[Potential games]
\label{ex:potential}
A population game $\game$ is called a \emph{potential game} \citep{San01,MS96} if there exists a \emph{potential function} $\pot$ defined on a neighborhood of $\strat$ such that
\begin{equation}
\label{eq:potential}
\frac{\pd\pot}{\pd x_{\pure}}
	=\payv_{\pure}(x)
	\quad
	\text{ for all $\pure\in\pures$ and all $x\in\strat$.}
\end{equation}
\end{example}

\begin{example}
[Contractive games]
\label{ex:contract}
A population game $\game$ is called (\emph{weakly}) \emph{contractive} \citep{HS09} if
\begin{equation}
\label{eq:contract}
\sum_{\pure\in\pures} (\payv_{\pure}(x') - \payv_{\pure}(x))  (x_{\pure}' - x_{\pure})
	\leq 0
	\quad
	\text{ for all $x, x' \in \strat$.}
\end{equation}
If \eqref{eq:contract} binds only when $x = x'$, $\game$ is called \emph{strictly contractive},
whereas if \eqref{eq:contract} binds for all $x,x'\in\strat$, $\game$ is called \emph{conservative}.%
\footnote{\cite{HS09} use the name \emph{stable games} instead of contractive, but \cite{San15} proselytizes for the terms employed here.
In convex analysis, condition \eqref{eq:contract} is called \emph{monotonicity}.}
\end{example}

\subsection{Evolutionary dynamics}
\label{sec:ED}

The term \emph{evolutionary dynamics} refers to rules that assign to each population game $\game$ a dynamical system on its state space $\strat$.
This is usually done by mapping each game to a \emph{law of motion}, i.e. a differential equation of the form
\begin{equation}
\label{eq:ED}
\tag{D}
\dot x
	= \dynfield(x).
\end{equation}
In most cases, the \emph{motion field} $\dynfield(x)$ of \eqref{eq:ED} is defined by introducing a mapping $(x,\pi)\mapsto \wilde\dynfield(x,\pi)$ from state/payoff pairs to vectors, and then specifying that $\dynfield(x) \equiv \wilde\dynfield(x,\payv(x))$.
In what follows, we will focus exclusively on such dynamics.

To ensure that solutions to \eqref{eq:ED} remain in $\strat$ for all $t\geq0$, $\dynfield(x)$ should not point outward from $\strat$;
formally, $\dynfield(x)$ should lie in the \emph{tangent cone} of $\strat$ at $x$, defined here as
\begin{equation}
\label{eq:tcone-simplex}
\tcone_{\strat}(x)
	= \setdef{z\in\zspace^{\pures}}{z_\pure \geq 0 \text{ whenever } x_{\pure} = 0}.
\end{equation}
Under many evolutionary dynamics (including the replicator dynamics and other imitative dynamics), the support of $x(t)$ remains invariant under \eqref{eq:ED}, implying in turn that the interior of each face of $\strat$ remains invariant under \eqref{eq:ED}.
When this is the case, $\dynfield(x)$ actually lies in the \emph{tangent space} to $\strat$ at $x$, defined as
\begin{equation}
\label{eq:tspace-simplex}
\tspace_{\strat}(x)
	= \setdef{z\in\zspace^{\pures}}{z_\pure = 0\text{ whenever }x_\pure=0}
	\subseteq \tcone_{\strat}(x).
\end{equation}
Clearly, for every interior state $x\in\intstrat$, we have $\tspace_{\strat}(x) = \tcone_{\strat}(x) = \zspace^{\pures}$.

A basic monotonicity criterion linking \eqref{eq:ED} with the underlying game requires \emph{positive correlation} between the strategies' payoffs and growth rates.
Concretely, this means that
\begin{equation}
\label{eq:PC}
\tag{PC}
\sum_{\pure\in\pures} \payv_\pure(x) \dynfield_\pure(x)
	\geq 0
	\quad
	\text{for all $x \in \strat$,}
\end{equation}
with equality only if $\dynfield(x)=0$.%
\footnote{This and closely related conditions are considered by \cite{Fri91}, \cite{Swi93}, \cite{San01}, and \cite{DemRit03}.}
If \eqref{eq:ED} satisfies \eqref{eq:PC}, every \acl{NE} of $\game$ is a rest point of \eqref{eq:ED}.
For a detailed discussion, see \cite{San10}.

\smallskip

We provide two prototypical examples of evolutionary dynamics below:

\begin{example}
[The replicator dynamics]
\label{ex:Rep}
The quintessential evolutionary game dynamics are the \emph{replicator dynamics} of \cite{TJ78}:
\begin{equation}
\label{eq:RD}
\tag{RD}
\dot x_{\pure}
	= x_{\pure} \left[ \payv_{\pure}(x) - \insum_{\purealt\in\pures} x_{\purealt} \payv_{\purealt}(x) \right].
\end{equation}
\end{example}

\begin{example}
[The  Euclidean projection dynamics]
\label{ex:proj}
The other fundamental example we consider is the \emph{Euclidean projection dynamics} of \cite{NZ97} (see also \citealp{Fri91}, and \citealp{LS08}).
These are defined by
\begin{equation}
\label{eq:PD}
\tag{PD}
\dot x = \argmin_{z \in \tcone_{\strat}(x)} \norm{\payv(x) - z}_{2}^{2},
\end{equation}
where $\norm{z}_{2}=(\sum_{\pure} z_{\pure}^{2})^{1/2}$ denotes the ordinary Euclidean norm on $\R^{\pures}$.
Geometrically, the dynamics \eqref{eq:PD} are defined by taking the Euclidean projection of the payoff field $\payv(x)$ onto the tangent cone $\tcone_{\strat}(x)$.
Since $\tcone_{\strat}(x) = \zspace^{\pures}$ on the interior $\intstrat$ of the simplex, we obtain the simple formula
\begin{equation}
\label{eq:Friedman}
\dot x_{\pure}
	= \payv_{\pure}(x) -\frac{1}{\abs{\pures}} \sum_{\purealt\in\pures} \payv_{\purealt}(x),
\end{equation}
valid for all interior $x\in\intstrat$.
For an explicit formula on the boundary of $\strat$, see \cref{ex:PD-full}.
\end{example}

\subsection{Antecedents}
\label{sec:previous}

The class of dynamics studied here is a substantial generalization of both the replicator dynamics and the projection dynamics.
We now describe works from an assortment of fields that are antecedents of our approach.

The replicator equation \eqref{eq:RD} for common interest games
is a basic model from population genetics \citep{SS83}.
The \emph{fundamental theorem of natural selection}, attributed to \cite{Fis30}, states that natural selection among genes increases overall population fitness.
\cite{Kim58} introduced a corresponding maximum principle showing that population fitness increases at a maximum rate under \eqref{eq:RD}, provided that one imposes a certain nonlinear constraint on the set of feasible changes in population frequencies (see \cref{rem:Kimura} in \cref{sec:interior}).
Later, \cite{Sha79} and \cite{Aki79} put Kimura's maximum principle on a firm mathematical footing using tools from differential geometry \textendash\ specifically, by introducing a suitable Riemannian metric (see \cref{sec:metrics}).
The derivation of the replicator dynamics in the latter papers provides a basic instance of the geometric construction of Riemannian dynamics developed in \cref{sec:geometry}, while our construction based on balancing gains and costs can be viewed as
an extension of
Kimura's analysis (cf.~\cref{rem:Kimura}).

\cite{HS90} model natural selection in populations of animals whose traits are represented by elements of a continuous set.
They assume that all members of the population share the same trait $x$, except for an infinitesimal group of mutants whose traits differ infinitesimally from $x$. 
The evolution of the preponderant trait $x$ follows a gradient-like process, moving in the direction that agrees with the play of the most successful local mutants.
To obtain variations on this process, \cite{HS90} use a Riemannian metric to define the size and shape of the neighborhood of local mutants.
When the trait space is $\strat$ and the fitness of mutant $y$ takes the linear form $\sum_{\pure} y_{\pure} \payv_{\pure}(x)$, they showed that the evolution of $x$ on the interior of $\strat$ is given by
\begin{equation}
\label{eq:Rie-coordsIntro}
\dot x_{\pure}
	= \sum_{\purealt\in\pures} \left[
	g_{\pure\purealt}^{-1}(x)
	- \frac{\sum_{\purealtalt} g_{\pure\purealtalt}^{-1}(x)  \sum_{\purealtalt} g_{\purealtalt\purealt}^{-1}(x)}{\sum_{\purealtalt,\kappa} g_{\purealtalt\kappa}^{-1}(x)}
	\right]
	\payv_{\purealt}(x),
\end{equation}
where $g(x)$ is a field of symmetric positive definite matrices that defines the Riemannian metric in question (see \cref{sec:metrics}).
\cite{HS90} then observed that under the Shahshahani metric, the system \eqref{eq:Rie-coordsIntro} boils down to the replicator dynamics \eqref{eq:RD}.
As we shall see, \eqref{eq:Rie-coordsIntro} describes the dynamics studied in this paper at all states $x\in\strat$ in what we call the \emph{minimal-rank} case (cf. \cref{sec:boundary}).

In the course of analyzing perturbed best response dynamics \citep{FL98} and variants of fictitious play \citep{Bro51}, \cite{Hop99b} introduced a class of game dynamics that are defined on the interior of $\strat$ as
\begin{equation}
\label{eq:Hopkins}
\dot x_{\pure}
	= \sum_{\purealt\in\pures} \HopFcn_{\pure\purealt}(x)\payv_{\purealt}(x).
\end{equation}
Here $\HopFcn(x)$ is a smoothly-varying field of symmetric matrices that are positive definite on $\zspace^{\pures}$ and map constant vectors to $0$.
\cite{Hop99b} showed that the linearization of these dynamics agrees with that of perturbed best response dynamics up to a positive affine transformation.
As a result, the local stability of rest points of \eqref{eq:Hopkins} agrees with that of the corresponding rest points of perturbed best response dynamics with sufficiently small noise levels. 
As we show in \cref{sec:Hopkins}, the dynamics \eqref{eq:Rie-coordsIntro} satisfy Hopkins' conditions;
conversely, all dynamics satisfying Hopkins' conditions can be expressed in the form \eqref{eq:Rie-coordsIntro}.
Thus, on the interior of $\strat$, the dynamics of \cite{Hop99b} are equivalent to the dynamics studied here (\cref{prop:HopDRD}).

More recently, \cite{Har11} used ideas from information geometry to define generalizations of the replicator dynamics, and employed concepts from Riemannian geometry to state and prove certain properties of the induced dynamics.
Ignoring boundary issues, these dynamics are an important special case of ours \textendash\ specifically, the class of separable dynamics that we introduce in \cref{ex:separable}.

Finally, we note here that there is a surprising and deep connection between the Hessian subclass of Riemannian game dynamics and a model of reinforcement learning recently examined by \cite{MS16}.
We explore this relation in detail in \cref{sec:RL}.

\section{Riemannian game dynamics}
\label{sec:dynamics}

\subsection{Gains, costs, and dynamics}
\label{sec:gains-costs}

We now define the dynamics we study as balancing a \emph{gain from motion}, determined from the game's payoffs, against a \emph{cost of motion}, a new primitive that
captures the difficulty of motion along a given direction from a given state.
To streamline our presentation, we focus below on interior states $x\in\intstrat\equiv\intr(\strat)$, postponing the treatment of boundary states until the machinery needed to handle them is in place.

Given a population game $\game(\pures,\payv)$, the \emph{gain from motion} from state $x\in\strat$ along $z \in \R^{\pures}$ is defined as
\begin{equation}
\label{eq:gain}
\gain^\payv(z;x) = \sum_{\pure\in\pures} \payv_{\pure}(x) z_{\pure},
\end{equation}
In words, the gain of motion measures the agreement between payoffs and vectors of motion as in the standard monotonicity criterion \eqref{eq:PC}.  
For an alternative interpretation, recall that the defining property \eqref{eq:potential} of a potential game with potential function $\pot$ can be expressed as
\begin{equation}
\label{eq:pot-rate}
\sum_{\pure\in\pures} \frac{\pd\pot}{\pd x_{\pure}} z_{\pure}
	= \sum_{\pure\in\pures} \payv_{\pure}(x) z_{\pure}\quad
	\text{for all $z \in \R^\pures$ and all $x\in\strat$.}
\end{equation}
The \acl{LHS} of \eqref{eq:pot-rate} is the rate of change in the value of potential as the state moves away from $x$ along $z$.
Viewed in this light, the gain $\gain^\payv(z;x)$ extends the notion of ``the rate of increase in potential'' to games that do not admit a potential function.%
\footnote{The logic here is similar to the original motivation for the definition of contractive games, which extends the idea of a game with a concave potential function to games that do not admit a potential \citep{HS09}.
The gain \eqref{eq:gain} is referred to as the ``aggregate gross gain'' by \cite{Zus18} in his general analysis of Lyapunov functions for contractive games and evolutionarily stable strategies.}
In particular, the gain captures the alignment between the direction of motion $z$ and the payoffs at state $x$;
it is also linearly homogeneous in $z$, so it grows linearly as one increases the speed of motion in a fixed direction.  

By contrast, the \emph{cost of motion} $\cost(z;x)$ is a primitive that represents the intrinsic difficulty of moving from state $x$ along a given displacement vector $z$.
For concreteness, we assume that the costs of motion are positive, smoothly varying with the population state $x$, and quadratic in $z$.
It is convenient to define costs $\cost(z;x)$ for states $x$ in the positive orthant $\orthant\equiv \R_{++}^{\pures}$ and for displacement vectors $z$ in $\R^{\pures}$.%
\footnote{We can interpret $\orthant$ as the set of population states that could arise if the population size were allowed to vary.  We could instead define costs only for states in $\intstrat$ and displacement vectors in $\zspace^{\pures}$, at the price of additional abstraction: see \cref{rem:DomOfg} below.}
Then since costs are positive and quadratic in $z$, the cost function 
can be expressed as
\begin{equation}
\label{eq:cost}
\cost(x;z)
	= \frac{1}{2} z^{\ttop} g(x) z,\quad
	\text{for all $z \in \R^\pures$ and all $x\in\orthant$.}
\end{equation}
where $g$ is a smooth assignment of symmetric positive definite matrices $g(x)$ to states $x \in \orthant$.

To use the above to define the dynamics at interior population states, we posit that the vector of motion from state $x\in \intstrat$ maximizes the difference between the gain of motion $\gain^\payv(x;z)$ and the cost of motion $\cost(x;z)$, subject to feasibility:
\begin{equation}
\label{eq:RGD-cost}
\dot x
	= \argmax_{z\in\zspace^{\pures}} \, \bracks*{\gain^\payv(z;x) - \cost(z;x)}.
\end{equation}
We refer to the dynamics \eqref{eq:RGD-cost} as \emph{Riemannian game dynamics}, for reasons that we will soon make clear.
Before doing so, we show how the leading examples of these dynamics are derived from the ansatz \eqref{eq:RGD-cost} through suitable choices of the cost function $\cost(x;z)$:

\begin{example}
\label{ex:cost-Eucl}
A straightforward, state-independent choice for the cost of motion is
\begin{equation}
\label{eq:cost-Eucl}
\cost(x;z)
	= \frac{1}{2} \sum_{\pure\in\pures} z_{\pure}^{2}.
\end{equation}
To solve the resulting maximization problem in \eqref{eq:RGD-cost}, consider the Lagrangian
\begin{equation}
\Lambda(z,\mu;x)
	= \sum_{\pure\in\pures} \bracks*{\payv_{\pure}(x) z_{\pure} - \frac{1}{2} z_{\pure}^{2} - \mu z_{\pure}},
\end{equation}
where the last term is associated with the motion feasibility constraint $\sum_{\pure\in\pures} z_{\pure} = 0$.
A direct differentiation gives the optimality condition $z_{\pure} = \payv_{\pure}(x) - \mu$, and
the feasibility constraint
yields $\mu = \abs{\pures}^{-1} \sum_{\pure\in\pures} \payv_{\pure}(x)$.
Substituting back into in \eqref{eq:RGD-cost} yields
\begin{equation}
\label{eq:PD-interior}
\dot x_{\pure}
	= \payv_{\pure}(x) - \frac{1}{\abs{\pures}} \sum_{\pure\in\pures} \payv_{\pure}(x).
\end{equation}
As we discussed in \cref{sec:prelims} (cf.~\cref{ex:proj}), the system \eqref{eq:PD-interior} describes the (Euclidean) projection dynamics of \cite{NZ97} on $\intstrat$.
\end{example}

\begin{example}
\label{ex:cost-Shah}
For a basic state-dependent choice for the cost of motion, let
\begin{equation}
\label{eq:cost-Shah}
\cost(x;z)
	= \frac{1}{2} \sum_{\pure\in\pures} \frac{z_{\pure}^{2}}{x_{\pure}}
\end{equation}
The Lagrangian for the maximization problem in \eqref{eq:RGD-cost} is now
\begin{equation}
\Lambda(z,\mu;x)
	= \sum_{\pure\in\pures} \bracks*{\payv_{\pure}(x) z_{\pure} - \frac{z_{\pure}^{2}}{2x_{\pure}} - \mu z_{\pure}}.
\end{equation}
Differentiating now yields the optimality condition $z_{\pure} = x_{\pure} \payv_{\pure}(x) - \mu x_{\pure}$, and feasibility implies that $\mu = \sum_{\purealt\in\pures} x_{\purealt} \payv_{\purealt}(x)$.
Substituting in \eqref{eq:RGD-cost}, we obtain 
\begin{equation}
\label{eq:RD-interior}
\dot x_{\pure}
	= x_{\pure} \bracks*{\payv_{\pure}(x) - \insum_{\purealt\in\pures} x_{\purealt} \payv_{\purealt}(x)}.
\end{equation}
The system \eqref{eq:RD-interior} defines the replicator dynamics of \cite{TJ78} 
(cf.~\cref{ex:Rep}).
Although the derivation above assumed that $x$ is interior, the expression \eqref{eq:RD-interior} actually describes the replicator dynamics on all of $\strat$;
we explain why this is so in \cref{sec:boundary}.
\end{example}

\subsection{Costs of motion and Riemannian metrics}
\label{sec:metrics}

We now proceed with a reinterpretation of the costs of motion using notions from geometry. 
The fundamental notion here is that of a \emph{Riemannian metric}, a position-dependent variant of the ordinary (Euclidean) scalar product between vectors.%
\footnote{To be clear, a Riemannian metric is not a metric in the sense of measuring distances between points in a metric space, but it induces such a distance function in a canonical way.  For a comprehensive introduction to this topic, see the masterful account of \cite{Lee97,Lee03}.} 

To start, we recall  that a \emph{scalar product} on a subspace  $\vecspace$ of $\R^{\pures}$ is a bilinear pairing $\product{\argdot}{\argdot}\from \vecspace\times \vecspace\to \R$ which satisfies the following for all $w,w^\prime\in\vecspace$:
\begin{enumerate}
\item
\emph{Symmetry:}
$\product{w}{w^\prime} = \product{w^\prime}{w}$.
\item
\emph{Positive definiteness:}
$\product{w}{w} \geq 0$, with equality if and only if $w=0$.
\end{enumerate}
The \emph{norm} of a vector $w\in \vecspace$ is then defined as
\begin{equation}
\label{eq:norm}
\norm{w}
	= \product{w}{w}^{1/2}.
\end{equation}

When $\vecspace=\R^{\pures}$, 
the definition above becomes most transparent by
writing $w = \insum_{\pure} w_{\pure} \bvec_{\pure}$ and $w' = \insum_{\purealt} w_{\purealt}' \bvec_{\purealt}$ in the standard basis $\{\bvec_\pure\}_{\pure\in\pures}$ of $\R^{\pures}$.
Since $\product{\argdot}{\argdot}$ is positive definite and bilinear, there exists a positive-definite matrix $\scprod = \left( \scprod_{\pure\purealt} \right)_{\pure,\purealt\in\pures}$ such that 
\begin{subequations}
\label{eq:coords}
\begin{flalign}
\label{eq:product-coords}
\product{w}{w'}
	&= \sum_{\pure,\purealt\in\pures} w_{\pure}  \scprod_{\pure\purealt} w_{\purealt}'
	= w^{\ttop} {\scprod} w'
\intertext{and}
\label{eq:norm-coords}
\norm{w}^{2}
	&= \sum_{\pure,\purealt\in\pures} w_{\pure} \scprod_{\pure\purealt} w_{\purealt}
	= w^{\ttop} {\scprod} w.
\end{flalign}
\end{subequations}
The matrix $\scprod$ is known as the \emph{metric tensor} of $\product{\argdot}{\argdot}$ and its components are $\scprod_{\pure\purealt} = \product{\bvec_{\pure}}{\bvec_{\purealt}}$.
Clearly, a scalar product is represented uniquely by its metric tensor and vice versa, so we will move freely between the two representations in what follows.

With all this in mind, a \emph{Riemannian metric} on an open set $\open$ of $\R^{\pures}$ is a $C^{1}$-smooth assignment of scalar products $\product{\argdot}{\argdot}_{x}$ to each $x\in\open$ \textendash\ or, equivalently, as a smooth field $g(x)$ of symmetric positive-definite matrices on $\open$.
In other words, a Riemannian metric  prescribes a way of measuring lengths of and angles between displacement vectors at each $x\in\open$.

The similarity in notation between the above and the definition of costs of motion is not a coincidence.
Looking back at \eqref{eq:RGD-cost}, we see that costs of motion and Riemannian metrics are both defined by means of a $C^{1}$-smooth field of symmetric positive-definite matrices,
with costs and norms being related via
\begin{equation}
\label{eq:metric2cost}
\cost(x;z) = \frac{1}{2} z^{\ttop} g(x) z
	= \frac{1}{2} \norm{z}_{x}^{2}.
\end{equation}
We summarize this connection as follows:

\begin{observation}
Specifying a cost function on $\orthant$ is equivalent to endowing $\orthant$ with a Riemannian metric.
\end{observation}

\begin{remark}
\label{rem:DomOfg}
Defining costs of motion and Riemannian metrics on the positive orthant $\orthant$ allows us to work in standard coordinates,
and simplifies passing from one to the other.
That being said, we could equally well have taken a more parsimonious approach
by defining costs of motion $\cost(x;z)$ only for states $x \in \intstrat$ and feasible displacement vectors $z \in \zspace^{\pures}$, and similarly working with Riemannian metrics 
$\product{\argdot}{\argdot}_{x}$ on $\zspace^{\pures}$ for each $x \in \intstrat$.
In this approach, the equivalence between cost functions and Riemannian metrics can be derived from a standard bijection between quadratic forms and bilinear forms  \cite[see \eg][p.~433]{FIS02}, but at the cost of an extra degree of abstraction.  
\end{remark}

Before proceeding, it is instructive to recast our previous examples in terms of Riemannian metrics:

\begin{example}
\label{ex:metric-Eucl}
The \emph{Euclidean metric} is defined by choosing $g(x)$ to be the identity matrix:
\begin{equation}
\label{eq:metric-Eucl}
g(x)
	= \identity = \diagidentity
	\quad
	\text{for all $x\in\orthant$}.
\end{equation}
This metric corresponds to the cost function $\cost(z;x) = \frac{1}{2} \sum_{\pure} z_{\pure}^{2}$ of \cref{ex:cost-Eucl}, and yields the standard expressions $\product{w}{w'}_{x} = w^{\ttop}w'$ and $\norm{w}_{x} = \sqrt{ w^{\ttop} w}$, all independent of $x$.
\end{example}

\begin{example}
\label{ex:metric-Shah}
The \emph{Shahshahani metric} 
is defined as
\begin{equation}
\label{eq:metric-Shah}
g(x)
	= \diag(1/x_{1},\dotsc,1/x_{n})
	\quad
	\text{for all $x\in\orthant$}.
\end{equation}
This metric corresponds to the cost function $\cost(z;x) = \frac{1}{2} \sum_{\pure} z_{\pure}^{2} / x_{\pure}$ of \cref{ex:cost-Shah}, and yields the Shahshahani inner product $\product{w}{w'}_{x} = \insum_{\pure} w_{\pure} w_{\pure}'/x_{\pure}$.
In contrast to its Euclidean counterpart, the Shahshahani metric is state-dependent:
For instance, since $\norm{\bvec_{\pure}}_{x} = x_{\pure}^{-1/2}$, the set of vectors at $x$ with Shahshahani norm $1$ is squeezed toward the $x_{\pure}$ axis as $x_{\pure}$ becomes small (cf.~\cref{fig:balls-Shah}).
\end{example}

\begin{figure}[t]
\subfigure
[Euclidean unit balls ($p = 0$)]
{\label{fig:balls-Eucl}%
\includegraphics[width=.48\textwidth]{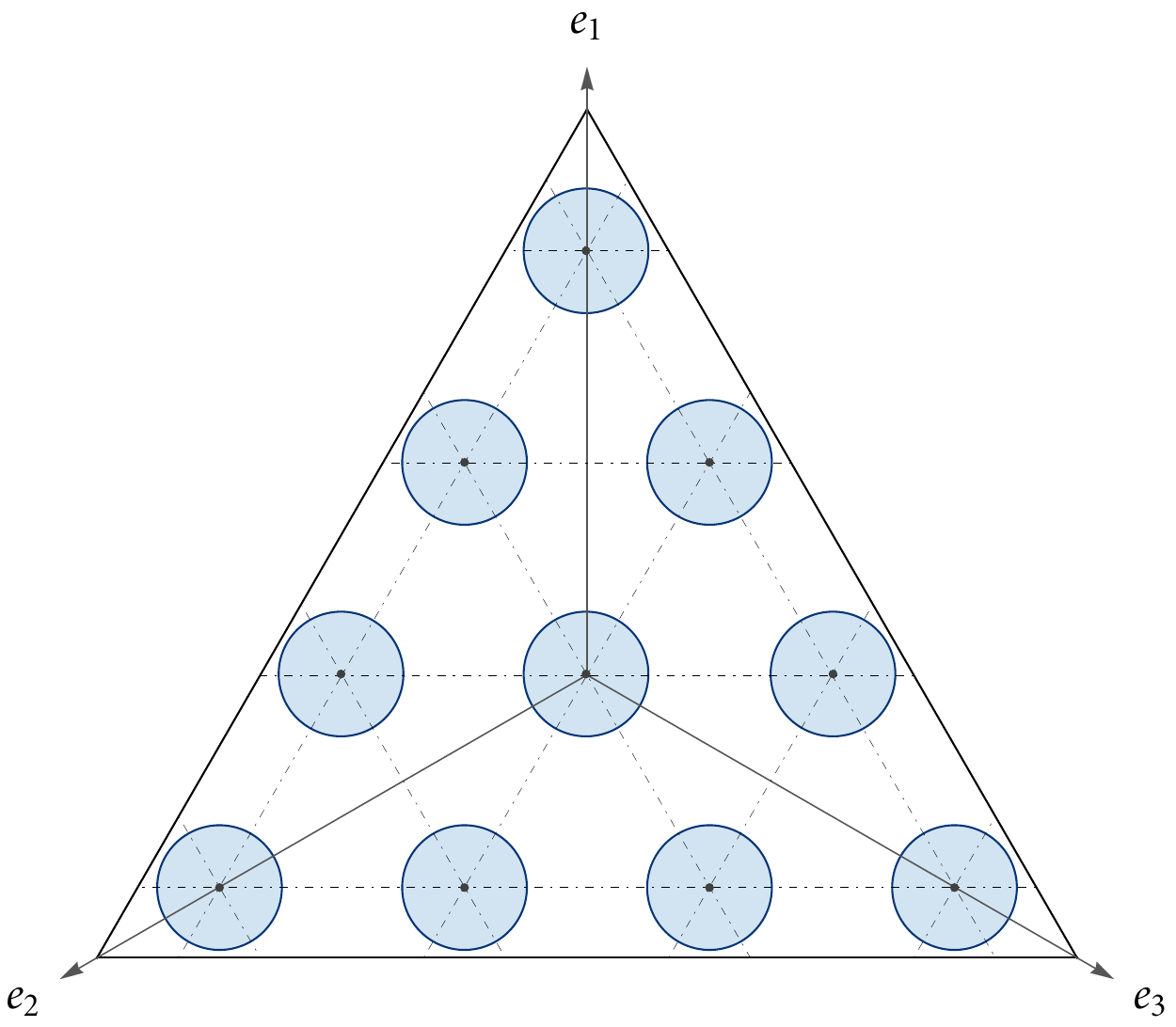}%
} 
\hfill
\subfigure
[Shahshahani unit balls ($p = 1$)]
{\label{fig:balls-Shah}%
\includegraphics[width=.48\textwidth]{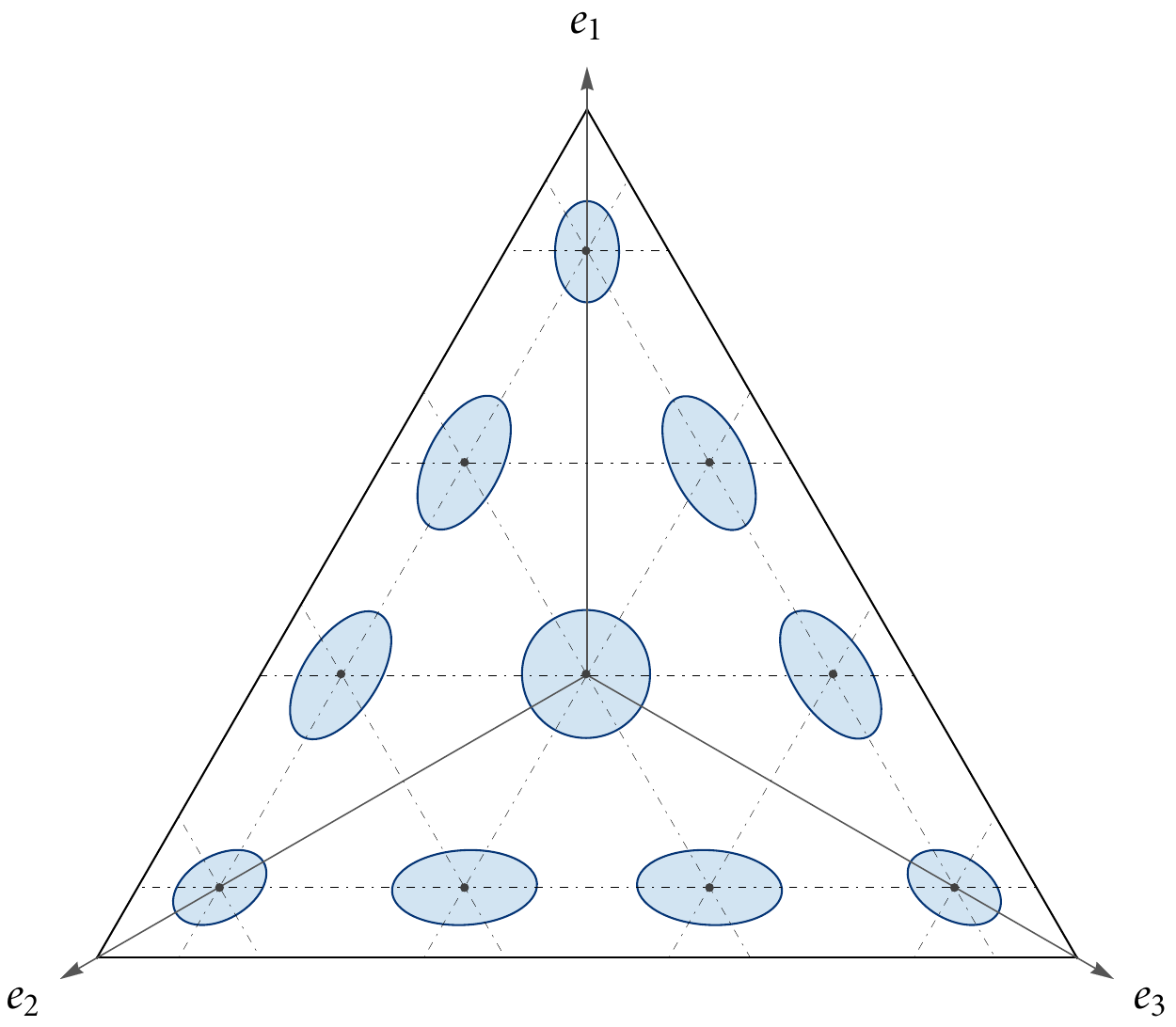}%
} 
\subfigure
[$p$-Shahshahani unit balls  ($p = 2$)]
{\label{fig:balls-pShah}%
\includegraphics[width=.48\textwidth]{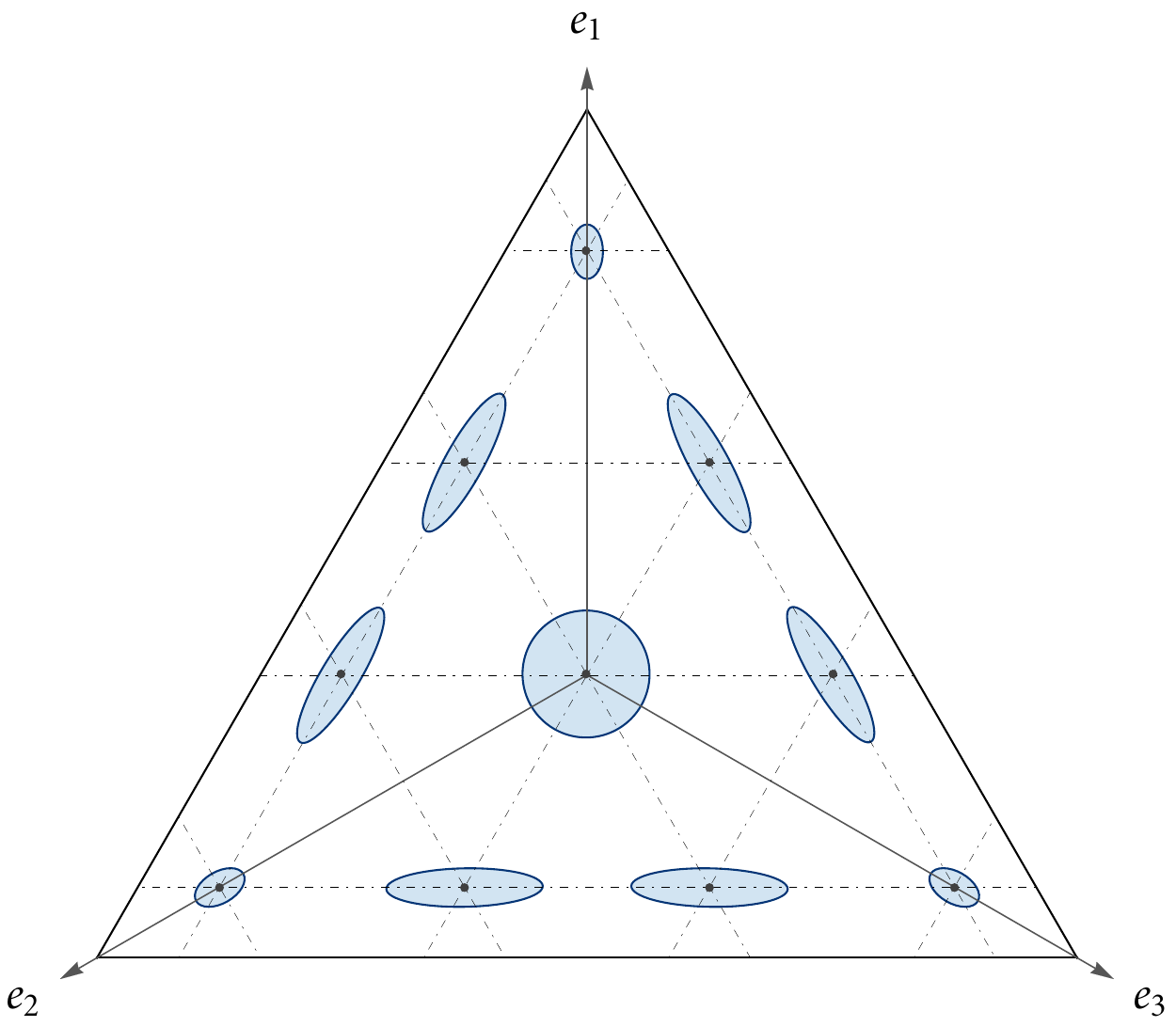}%
} 
\hfill
\subfigure
[Nested unit balls ($\pures_{1} = \{2,3\}$, $s =3$)]
{\label{fig:balls-nested}%
\includegraphics[width=.48\textwidth]{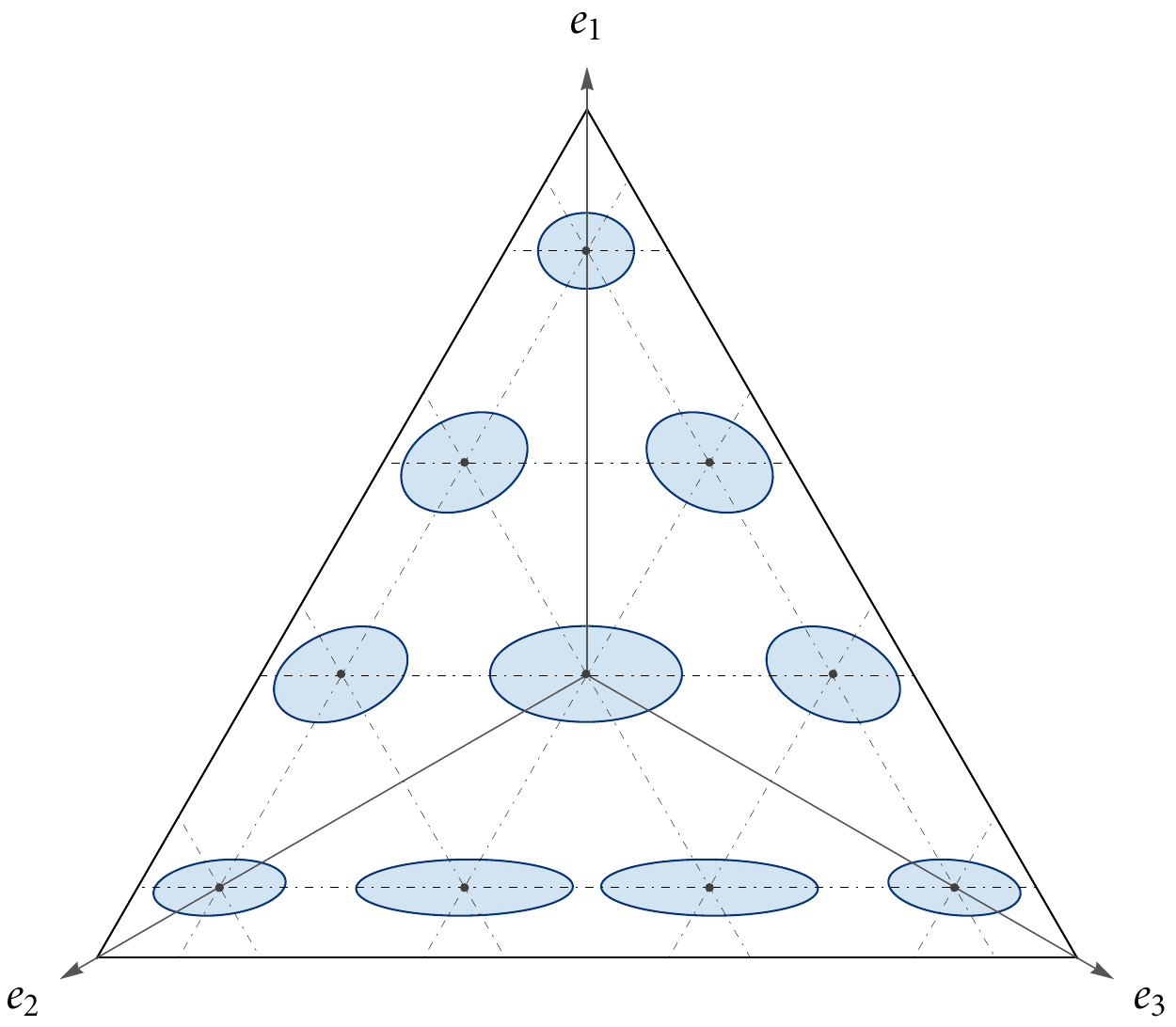}%
} 
\caption{Unit balls on the $3$-simplex under the metrics of \crefrange{ex:metric-Eucl}{ex:metric-nested}.
For each base point $x$ shown, the shaded regions comprise all tangent vectors $z$ based at $x$ that  satisfy $\norm{z}_{x}^2 \leq 1$.
}
\label{fig:balls}
\end{figure}

We now present two further classes of metrics to which we return in \cref{sec:examples}:

\begin{example}
\label{ex:metric-pShah}
For $p\ge 0$, the \emph{p-Shahshahani metric} is defined as
\begin{equation}
\label{eq:metric-pShah}
g(x)
	= \diag(1/x_{1}^{p},\dotsc,1/x_{\nPures}^{p})
	\quad
	\text{for all $x\in\orthant$}.
\end{equation}
This definition includes the Euclidean metric ($p = 0$) and the standard Shahshahani metric ($p = 1$) as special cases,
and corresponds to the cost function 
\begin{equation}
\label{eq:cost-Shahp}
\cost(x;z)
	= \frac{1}{2} \sum_{\pure\in\pures} \frac{z_{\pure}^{2}}{x^p_{\pure}}.
\end{equation}
Since $\frac1{x^p_{\pure}}/\frac1{x^p_{\purealt}} = (x_\purealt/x_\pure)^p$, specifying larger values of $p$ means raising the relative cost of changes in the use of rare strategies.  For instance, if strategy $\pure$ is half as prevalent in the population as strategy $\purealt$,
then changes in the use of $\pure$ cost $2^p$ times as much as changes in the use of $\purealt$.

\cref{fig:balls-Eucl,fig:balls-Shah,fig:balls-pShah} illustrate the effects of increasing the value of $p$ on costs of motion: when $x_\alpha$ is small, increasing $p$ increases the cost of moving toward and away from the $x_\alpha = 0$ boundary relative to the cost of moving along this boundary.
\end{example}

\begin{example}
\label{ex:metric-nested}
Let $\pures_{1},\dotsc,\pures_{m}$ be a partition of $\pures$ into $m$ groups of intrinsically similar strategies, let $\class{\pure}$ denote the group containing strategy $\pure$,
let $x_{\class{\pure}} = \sum_{\purealt\in\class{\pure}} x_{\purealt}$ denote the population share of all strategies that are ``similar'' to $\pure$ in the above partition, and let $s>0$ be a parameter representing the ``strength'' of the similarity relation.
The \emph{nested Shahshahani metric} is then defined as
\begin{equation}
\label{eq:metric-nested}
g_{\pure\purealt}(x)
	= \begin{cases}
	 \frac{\delta_{\pure\purealt}}{x_{\pure}} + s \frac{1}{x_{\class{\pure}}}
		&\quad
		\text{if $\purealt\in\class{\pure}$},
		\\
		0
		&\quad
		\text{otherwise}.
	\end{cases}
\end{equation}
While the full expression for the cost function corresponding to the metric \eqref{eq:metric-nested} is cumbersome, the cost of motion
along the basic directions $\bvec_{\purealt} - \bvec_{\pure}$ takes a fairly simple form, namely
\begin{equation}\label{eq:NRNorm}
\cost(\bvec_{\purealt} - \bvec_{\pure};x)
	= \begin{cases}
	\frac{1}{2} \!\left(\frac{1}{x_{\pure}} + \frac{1}{x_{\purealt}}\right) 
		&\quad
		\text{if $\purealt\in\class{\pure}$},
		\\
	\frac{1}{2} \!\left(\frac{1}{x_{\pure}} + \frac{1}{x_{\purealt}}\right) + \frac{1}{2}s \!\left(\frac{1}{x_{\class{\pure}}} +\frac{1}{x_{\class{\purealt}}} \right)
		&\quad
		\text{otherwise}.
	\end{cases}
\end{equation}
Under \eqref{eq:NRNorm}, switches between strategies in the same group take the same form as under the Shahshahani cost function from \cref{ex:cost-Shah,ex:metric-Shah}.
On the other hand, switches between strategies in different groups are more costly, with the additional costs being inversely proportional to population shares of the groups and proportional to the strength of the similarity relation.

\cref{fig:balls-nested} illustrates the costs of motion \eqref{eq:NRNorm} from various states in $\strat$ for partition $\pures_{1} = \{1\}$, $\pures_{2}= \{2, 3\}$ and  similarity strength $s= 3$.  The unit balls near the $x_1=0$ boundary are elongated along that boundary to a greater extent than the balls near the other boundaries.  This reflects the fact that, mutatis mutandis, a unit of cost buys more motion between strategies $2$ and $3$ than between the other pairs of strategies.
\end{example}

\subsection{Derivation of the dynamics: the interior case}
\label{sec:interior}

With the above machinery at hand, we can provide an explicit description of the game dynamics under study on the interior $\intstrat$ of $\strat$.
To do so, fix a population game $\game \equiv \game(\pures,\payv)$ and a Riemannian metric $g$ on $\orthant$.
Then,
by \eqref{eq:metric2cost},
the associated Riemannian game dynamics are
\begin{equation}
\label{eq:RGD-metric}
\dot x
	= \argmax_{z\in\zspace^{\pures}} \, \bracks*{\gain^\payv(z;x) - \cost(z;x)}
	= \argmax_{z\in\zspace^{\pures}} \, \bracks*{\sum_{\pure\in\pures} \payv_{\pure}(x)z_{\pure} - \tfrac{1}{2} \norm{z}_{x}^{2}}.
\end{equation}

As in \cref{ex:cost-Eucl,ex:cost-Shah}, to obtain an explicit expression for the vector of motion that solves the maximization problem \eqref{eq:RGD-metric}, consider the Lagrangian
\begin{equation}
\label{eq:Lagrangian}
\Lambda(z,\mu;x)
	= \sum_{\pure\in\pures}\payv_{\pure}(x)z_{\pure} - \tfrac{1}{2}\norm{z}^2_{x}
	- \mu \sum_{\pure\in\pures} z_{\pure}.
\end{equation}
Then, interpreting  $(\payv_{\pure}(x))_{\pure\in\pures}$ as a row vector (see \cref{sec:geometry}) and writing $\onecol = (1,\dotsc,1)^{\ttop}$ for the column vector of ones in the standard basis of $\R^{\pures}$, a simple differentiation yields the first-order optimality condition
\begin{equation}
\label{eq:KKT}
\payv(x)
	= z^{\ttop} g(x) + \mu\onecol^{\ttop}.
\end{equation}
Thus, after rearranging, we get
\begin{equation}
z
	= g^{-1}(x) \bracks{\payv(x)^{\ttop} - \mu\onecol},
\end{equation}
where $g^{-1}(x)$ denotes the inverse of the matrix $g(x)$.
Using the constraint $\sum_{\pure} z_{\pure} = 0$ to solve for $\mu$ and substituting in \eqref{eq:RGD-metric}, some easy algebra leads to the explicit expression
\begin{equation}
\dot x
	= g^{-1}(x) \bracks*{\payv(x)^{\ttop} - \frac{\payv(x) g^{-1}(x) \onecol}{\onecol^{\ttop} g^{-1}(x)\onecol} \, \onecol}.
\end{equation}
Thus, if we set
\begin{equation}
\label{eq:sharp-vecs}
\payv^{\sharp}(x)
	= g^{-1}(x) \payv(x)^{\ttop}
	\quad
	\text{and}
	\quad
\normal(x)
	= g^{-1}(x) \onecol,
\end{equation}
we obtain
\begin{equation}
\label{eq:RGD-sharp}
\dot x
	= \payv^{\sharp}(x) - \frac{\product{\payv^{\sharp}(x)}{\normal(x)}_{x}}{\norm{\normal(x)}^{2}_{x}}\, \normal(x)
	= \payv^\sharp(x) - \frac{\sum_{\pure\in\pures} \payv^{\sharp}_{\pure}(x)}{\sum_{\pure\in\pures} \normal_{\pure}(x)}\,\normal(x)
\end{equation}

We now revisit our two archetypal examples in the light of the explicit expression \eqref{eq:RGD-sharp}:

\begin{example}
If $g(x) = \identity$ is the Euclidean metric, we get $\payv^{\sharp}(x) = \payv(x)^{\ttop}$ and $\normal(x) = \onecol$, so \eqref{eq:RGD-sharp} immediately boils down to \eqref{eq:PD-interior}.
Therefore, when the cost of motion is defined using Euclidean lengths, the components of the displacement vector $\dot x$ equal those of $\payv(x)$ up to a constant that ensures that $\dot x \in \zspace^{\pures}$.
\end{example}

\begin{example}
If $g(x) = \diag(1/x_{1},\dotsc,1/x_{n})$ is the Shahshahani metric of \cref{ex:metric-Shah}, we readily get $\payv_{\pure}^{\sharp}(x) = x_{\pure} \payv_{\pure}(x)$ and $\normal_{\pure}(x) = x_{\pure}$, so \eqref{eq:RGD-sharp} boils down to the replicator dynamics \eqref{eq:RD-interior}.
Thus, when costs are defined using the
Shahshahani norm of the population displacement vector,
changes in the use of rare strategies are more costly than changes in the use of common ones.
As a consequence, the initial term of $\dot x_{\pure}$ is proportional to both the payoff  $\payv_{\pure}(x)$ of strategy $\pure$ and to the mass $x_{\pure}$ of agents playing strategy $\pure$.
The second term ensures that $\dot x \in \zspace^{\pures}$, but here the normalization for strategy $\pure$ is itself proportional to $x_{\pure}$.
\end{example}

\begin{remark}
\label{rem:Kimura}
The derivation above is closely related to  \citegen{Kim58} derivation of  the replicator dynamics in \emph{common interest games}, \ie games in which $\payv(x) = (Ax)^{\ttop}$ for some symmetric matrix $A$.
Such games admit the potential function $\pot(x) = \frac{1}{2} x^{\top}\!Ax$  (cf.~\cref{eq:potential}), which reports one-half of the population's average payoff. 
Using somewhat different language, \cite{Kim58} proposed the population dynamics 
\begin{equation}
\label{eq:RGD-Kimura}
\dot x
	= \argmax\setdef*{\insum_{\pure} \frac{\pd\pot}{\pd x_{\pure}} z_{\pure}}{z \in \tcone_{\strat}(x) \text{ and } \norm{z}^2_{x} = \sigma^2_{\payv}(x)}.
\end{equation}
Here $\insum_{\pure} \frac{\pd\pot}{\pd x_{\pure}} z_{\pure}$ is the rate of change of potential along $z$, $\norm{\argdot}_{x}^{2}$ is  the Shahshahani norm
and
$\sigma_{\payv}^{2}(x) = \sum_{\pure} x_{\pure} \bracks{\payv_{\pure}(x) - \sum_{\purealt} x_{\purealt} \payv_{\purealt}(x)}^{2}$ denotes the variance in the population's payoffs at state $x$.  
It is easy to verify that \eqref{eq:RGD-Kimura} boils down to the replicator dynamics for the potential game ${\payv}(x) = (Ax)^{\ttop}$.
\end{remark}

\subsection{The boundary case}
\label{sec:boundary}

We now turn to an important dichotomy that arises when extending the definition of the dynamics \eqref{eq:RGD-metric} to the boundary of $\strat$.
To begin, recall from \eqref{eq:tspace-simplex} that the \emph{tangent space} $\tspace_{\clorthant}(x)$ to the nonnegative orthant $\clorthant$ at $x$ is the linear subspace
\begin{equation}
\label{eq:tspace-orth}
\tspace_{\clorthant}(x)
	= \setdef{z\in\R^{\pures}}{z_{\pure} = 0\text{ whenever }x_{\pure}=0}
	= \R^{\supp(x)}.
\end{equation}
We then say that a Riemannian metric $g$ on $\orthant$ is \emph{extendable} to $\clorthant$ if the map $x \mapsto g^{-1}(x)$ on $\orthant$ admits a (necessarily unique) $C^{1}$-smooth extension to $\clorthant$ which we denote by $g^{\sharp}$ (so $g^{\sharp}(x) \equiv g^{-1}(x)$ for all $x\in\orthant$), and which satisfies
\begin{equation}
\label{eq:extension}
\tspace_{\clorthant}(x)
	\subseteq \im g^{\sharp}(x)\;\text{ for all }x\in\clorthant.
\end{equation}
In the above, $\im g^{\sharp}(x)$ is the image (column space) of $g^{\sharp}(x)$;
we henceforth call this set the \emph{domain} of $g$ at $x$ and denote it by $\domg(x)$.
\cref{prop:extension} in \cref{app:geometry} shows that if $g$ is extendable in the sense of \eqref{eq:extension}, then the field of scalar products  associated with $g$ also admits a unique continuous extension from $\orthant$ to $\clorthant$, with $\product{\argdot}{\argdot}_{x}$ defined on $\domg(x)$.

In what follows, we focus on two basic forms of extendability.
First,
if $\domg(x) = \R^\pures$ for all $x \in \clorthant$, we say that $g$ is \emph{full-rank extendable};
instead,
if $\domg(x) = \tspace_{\clorthant}(x)$ for all $x \in \clorthant$, we say that $g$ is \emph{minimal-rank extendable}.
We henceforth use the term ``extendable'' to refer to these two cases exclusively.

\begin{example}
\label{ex:Eucl-ext}
The Euclidean metric has $g^{\sharp}(x) = g^{-1}(x) = \identity$ for all $x\in\clorthant$, so it is full-rank extendable by default.
\end{example}

\begin{example}
\label{ex:Shah-ext}
The Shahshahani metric has $g^{\sharp}(x) = g^{-1}(x) = \diag(x_{1},\dotsc,x_{n})$, so $\domg(x) =\tspace_{\clorthant}(x) = \R^{\supp(x)}$ for all $x\in\clorthant$.
Thus, the Shahshahani metric is minimal-rank extendable, and the induced scalar product on $\domg(x) = \R^{\supp(x)}$ is
\begin{equation}
\product{w}{w'}_{x}
	= \insum_{\pure\in\supp(x)} w_{\pure} w_{\pure}' / x_{\pure}
	\quad
	\text{for all }w,w' \in \tspace_{\clorthant}(x).
\end{equation}
\end{example}

\begin{remark}
Intuitively, minimal-rank extendable metrics partition $\clorthant$ into the relative interiors of each of its faces (including $\orthant$ itself).
We will see that under the dynamics generated by such metrics, the relative interior of each face of $\strat$ is an invariant set.
\end{remark}

To extend the definition of the dynamics to the boundary $\bd(\strat)$ of $\strat$, we introduce the cone of \emph{$g$-admissible} vectors
\begin{equation}
\metcone(x)
	= \tcone_{\strat}(x)\cap\domg(x),
\end{equation}
This cone, which comprises all tangent vectors $z\in\tcone_{\strat}(x)$ that also lie in $\domg(x)$, specifies the possible directions of motion at a given state $x \in \strat$.
In particular, when $x \in \intstrat$ is interior, we have $\tcone_{\strat}(x) = \tspace_{\strat}(x) = \zspace^{\pures}$ and $\domg(x) = \R^{\pures}$.
Thus, the $g$-admissible set is the hyperplane
\begin{equation}
\label{eq:ProjDomInt}
\metcone(x)
	= \zspace^{\pures} \cap \R^{\pures}
	= \zspace^{\pures},
\end{equation}
as anticipated in \cref{eq:RGD-metric}.
Further instances of $g$-admissible cones are depicted in \cref{fig:cones}.


\begin{figure}
\footnotesize

\subfigure
[Full-rank extendability.]
{\label{fig:cones-Eucl}
\begin{tikzpicture}
[>=stealth,
vecstyle/.style = {->, line width=.25pt},
edgestyle/.style={-, line width=.5pt, black},
nodestyle/.style={circle, fill=Black,inner sep = .5pt},
plotstyle/.style={color=DarkGreen!80!Cyan,thick}]

\def\unit{.2\textwidth}
\def\costhirty{0.8660256}
\def\cosfortyfive{0.7071068}
\def\veclength{.5}

\coordinate (O) at (0,0);

\coordinate (e2) at (-\costhirty*\unit,-\unit/2);
\coordinate (e3) at (\costhirty*\unit,-\unit/2);
\coordinate (e1) at (0,\unit);

\coordinate (test1) at (-.05*\unit,.4*\unit);
\coordinate (test2) at (0,-\unit/2);

\node [nodestyle] (O) at (O) {};
\node [nodestyle, label = {[label distance=3] 210:{$\bvec_{2}$}}] (e2) at (e2) {};
\node [nodestyle, label = {[label distance=3] -30:{$\bvec_{3}$}}] (e3) at (e3) {};
\node [nodestyle, label = {[label distance=5] 90:{$\bvec_{1}$}}] (e1) at (e1) {};

\filldraw [MyBlue!15] (test1) circle (1.1*\veclength);
\filldraw [MyBlue!15] ($(test2) + 1.1*\veclength*(1,0)$) arc [start angle =0, end angle = 180, radius = 1.1*\veclength];

\node [nodestyle] (test1) at (test1) {};
\node [nodestyle] (test2) at (test2) {};

\draw [line width=.25pt,black!50] (O.center) to (e2.center);
\draw [line width=.25pt,black!50] (O.center) to (e3.center);
\draw [line width=.25pt,black!50] (O.center) to (e1.center);

\draw [vecstyle] (e2.center) to ($1.1*(e2)$);
\draw [vecstyle] (e3.center) to ($1.1*(e3)$);
\draw [vecstyle] (e1.center) to ($1.1*(e1)$);

\draw [edgestyle] (e2.center) to (e3.center);
\draw [edgestyle] (e3.center) to (e1.center);
\draw [edgestyle] (e1.center) to (e2.center);

\draw [vecstyle, black] (test1.center) to ($(test1) + \veclength*(0,1)$);
\draw [vecstyle, black] (test1.center) to ($(test1) - \veclength*(0,1)$);
\draw [vecstyle, black] (test1.center) to ($(test1) + \veclength*(\costhirty,1/2)$);
\draw [vecstyle, black] (test1.center) to ($(test1) - \veclength*(\costhirty,1/2)$);
\draw [vecstyle, black] (test1.center) to ($(test1) + \veclength*(-\costhirty,1/2)$);
\draw [vecstyle, black] (test1.center) to ($(test1) - \veclength*(-\costhirty,1/2)$);

\draw [vecstyle, black] (test2.center) to ($(test2) + \veclength*(1,0)$);
\draw [vecstyle, black] (test2.center) to ($(test2) + \veclength*(\cosfortyfive,\cosfortyfive)$);
\draw [vecstyle, black] (test2.center) to ($(test2) + \veclength*(0,1)$);
\draw [vecstyle, black] (test2.center) to ($(test2) + \veclength*(-\cosfortyfive,\cosfortyfive)$);
\draw [vecstyle, black] (test2.center) to ($(test2) - \veclength*(1,0)$);

\node [inner sep = 1pt] (label) at ($(O) - \veclength*(1.25,.5)$) {$\metcone(x)$};
\draw [vecstyle,MyBlue, bend left] (label.150) to ($(test1) + .9*\veclength*(-1/2,-\costhirty)$);
\draw [vecstyle,MyBlue, bend right] (label.-150) to ($(test2) + .9*\veclength*(-\costhirty,1/2)$);

\end{tikzpicture}
} 
\hfill
\subfigure
[Minimal-rank extendability..]
{\label{fig:cones-Shah}
\begin{tikzpicture}
[>=stealth,
vecstyle/.style = {->, line width=.25pt, black},
edgestyle/.style={-, line width=.5pt, black},
nodestyle/.style={circle, fill=Black,inner sep = .5pt},
plotstyle/.style={color=DarkGreen!80!Cyan,thick}]

\def\unit{.2\textwidth}
\def\costhirty{0.8660256}
\def\cosfortyfive{0.7071068}
\def\veclength{.5}

\coordinate (O) at (0,0);

\coordinate (e2) at (-\costhirty*\unit,-\unit/2);
\coordinate (e3) at (\costhirty*\unit,-\unit/2);
\coordinate (e1) at (0,\unit);

\coordinate (test1) at (-.05*\unit,.4*\unit);
\coordinate (test2) at (0,-\unit/2);

\node [nodestyle] (O) at (O) {};
\node [nodestyle, label = {[label distance=3] 210:{$\bvec_{2}$}}] (e2) at (e2) {};
\node [nodestyle, label = {[label distance=3] -30:{$\bvec_{3}$}}] (e3) at (e3) {};
\node [nodestyle, label = {[label distance=5] 90:{$\bvec_{1}$}}] (e1) at (e1) {};

\filldraw [MyBlue!15] (test1) circle (1.1*\veclength);

\node [nodestyle] (test1) at (test1) {};
\node [nodestyle] (test2) at (test2) {};

\draw [line width=.25pt,black!50] (O.center) to (e2.center);
\draw [line width=.25pt,black!50] (O.center) to (e3.center);
\draw [line width=.25pt,black!50] (O.center) to (e1.center);

\draw [vecstyle] (e2.center) to ($1.1*(e2)$);
\draw [vecstyle] (e3.center) to ($1.1*(e3)$);
\draw [vecstyle] (e1.center) to ($1.1*(e1)$);

\draw [edgestyle] (e2.center) to (e3.center);
\draw [edgestyle] (e3.center) to (e1.center);
\draw [edgestyle] (e1.center) to (e2.center);
\draw [MyBlue, line width = .75pt] ($(e2) - 1.05*\veclength*(1,0)$) to ($(e3) + 1.05*\veclength*(1,0)$);

\draw [vecstyle, black] (test1.center) to ($(test1) + \veclength*(0,1)$);
\draw [vecstyle, black] (test1.center) to ($(test1) - \veclength*(0,1)$);
\draw [vecstyle, black] (test1.center) to ($(test1) + \veclength*(\costhirty,1/2)$);
\draw [vecstyle, black] (test1.center) to ($(test1) - \veclength*(\costhirty,1/2)$);
\draw [vecstyle, black] (test1.center) to ($(test1) + \veclength*(-\costhirty,1/2)$);
\draw [vecstyle, black] (test1.center) to ($(test1) - \veclength*(-\costhirty,1/2)$);

\draw [vecstyle, black] (test2.center) to ($(test2) + \veclength*(1,0)$);
\draw [vecstyle, black] (test2.center) to ($(test2) - \veclength*(1,0)$);

\node [inner sep = 1pt] (label) at ($(O) - \veclength*(1.25,.5)$) {$\metcone(x)$};
\draw [vecstyle,MyBlue, bend left] (label.150) to ($(test1) + .9*\veclength*(-1/2,-\costhirty)$);
\draw [vecstyle,MyBlue] (label.-150) to ($(test2) + 1.75*\veclength*(-\costhirty,0)$);
\end{tikzpicture}
}
\label{fig:cones}
\caption{Admissible sets under the Euclidean and Shahshahani metrics.
For $x\in\intstrat$, we have $\metcone(x) = \tcone_{\strat}(x)$.
For $x\in\bd(\strat)$, we still have $\metcone(x) = \tcone_{\strat}(x)$ in the Euclidean case,
but the Shahshahani metric can only be extended to the tangent space  
$\metcone(x) = \tspace_{\strat}(x)$.}
\end{figure}
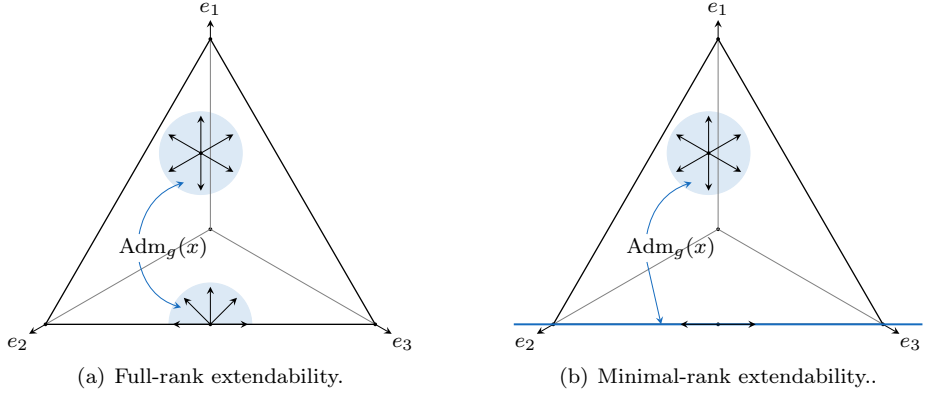


The restriction to $\domg(x)$ is needed because the norm  $\norm{z}_{x}^{2}$ is only defined for $z\in\domg(x)$.
When $g$ is extendable, the only case in which  $\domg(x)$ is not all of  $\R^{\pures}$ occurs when $x \in \bd(\strat)$ and $g$ is minimal-rank extendable, in which case $\domg(x) = \tspace_{\strat}(x) = \R^{\supp(x)}$ (cf.~\cref{ex:Shah-ext}).

\subsection{Riemannian game dynamics}
\label{sec:RieDDef}

With all this at hand, we are finally in a position to extend the definition of the dynamics to all of $\strat$.
Concretely, building on \eqref{eq:RGD-metric}, the \emph{Riemannian game dynamics} induced by an extendable $g$ are
\begin{equation}
\label{eq:RGD}
\tag{RmD}
\dot x
	= \argmax_{z\in\metcone(x)} \, \bracks*{\gain^\payv(z;x) - \cost(z;x)}
	= \argmax_{z\in\metcone(x)} \, \bracks*{\sum_{\pure\in\pures}\payv_{\pure}(x)z_{\pure} - \tfrac{1}{2} \norm{z}_{x}^{2}}
\end{equation}

Equation \eqref{eq:RGD-sharp} showed that \eqref{eq:RGD} can be expressed at interior states as
\begin{subequations}
\label{eq:RGD-coords}
\begin{equation}
\label{eq:RGD-coords1}
\dot x
	= \payv^{\sharp}(x) - \frac{\product{\payv^{\sharp}(x)}{\normal(x)}_{x}}{\norm{\normal(x)}^{2}_{x}}\, \normal(x)
	= \payv^\sharp(x) - \frac{\sum_{\pure\in\pures} \payv^{\sharp}_{\pure}(x)}{\sum_{\pure\in\pures} \normal_{\pure}(x)}\,\normal(x)
\end{equation}
where $\payv^{\sharp}(x) = g^{\sharp}(x) \payv(x)^{\ttop}$ and $\normal(x) = g^{\sharp}(x)\onecol$.
After a slight rearrangement, we can also express the dynamics as a linear transformation of payoffs $\payv(x)$:
\begin{equation}
\label{eq:RGD-coords2}
\dot x_{\pure}
	= \sum_{\purealt\in\pures} \left[ g^\sharp_{\pure\purealt}(x) - \frac{\normal_{\pure}(x) \normal_{\purealt}(x)}{\insum_{\gamma} \normal_{\gamma}(x)} \right] \payv_{\purealt}(x).
\end{equation}
\end{subequations}
Equation \eqref{eq:RGDMP} in \cref{app:dynamics} provides a concise third expression for the dynamics on $\intstrat$ in terms of a pseudoinverse matrix. 

If $g$ is minimal-rank extendable, \cref{prop:RGDMinFormula} in \cref{app:geometry} shows that \eqref{eq:RGD-coords} holds for all $x\in\strat$, provided that one uses $g^{\sharp}(x)$ in the definition \eqref{eq:sharp-vecs} of $\payv^{\sharp}(x)$ and $\normal(x)$.
\cref{prop:RGDMinFormula} also shows that, in this case, one need only take the sums in the formulas \eqref{eq:RGD-coords} over the strategies in the support of $x$. 

If instead $g$ is full-rank extendable, extending \eqref{eq:RGD-coords} to boundary states requires solving a convex program whose inequality constraints may be active.
For this reason, coordinate formulas for \eqref{eq:RGD} may depend on the support of $x$ \textendash\ and, indeed, \eqref{eq:RGD} may fail to be continuous at the boundary of $\strat$ (see \cref{ex:PD-full} below).
With this in mind, it will be convenient to call the dynamics generated by minimal-rank extendable metrics \emph{continuous Riemannian dynamics}, and those generated by full-rank extendable metrics \emph{discontinuous Riemannian dynamics}.

\subsection{Geometric derivation of the dynamics}
\label{sec:geometry}

In \eqref{eq:RGD}, the dynamics' vector of motion from $x$ is defined to maximize the difference between the gain $\gain^\payv(z;x) = \sum_{\pure\in\pures}\payv_{\pure}(x)z_{\pure}$ 
 and the cost of motion $\cost(z;x) = \frac{1}{2}\norm{z}_{x}^{2}$ over the set of admissible vectors $z\in\metcone(x)$.
We now show how these dynamics can be derived using a purely geometric approach, generalizing \citegen{Sha79} derivation of the replicator dynamics in common interest games, and \citegen{NZ97} definition of the Euclidean projection dynamics.
In what follows, we rely on some basic ideas from Riemannian geometry;
for a comprehensive treatment, we refer again to \cite{Lee97}.

To start, we introduce ideas about duality that explain our convention of writing payoffs using row vectors and the notations $\payv^\sharp(x)$ and $\normal(x)$ from \cref{sec:interior}.
As in \cref{sec:metrics}, let $W$ is a subspace of $\R^\pures$. 
A linear functional $\omega\from\vecspace\to\R$ acting on vectors $w \in W$ is called a \emph{covector}, and the space $\dspace$ of such functionals is called the \emph{dual space} of $\vecspace$.
We write $\braket{\omega}{w}$ for the action of a covector $\omega \in \dspace$ on a vector $w \in \vecspace$;
to emphasize this pairing,
the elements of $\vecspace$ and $\dspace$ are also referred to as \emph{primal} and \emph{dual} \emph{vectors} respectively.
When $W=\R^\pures$, we use the standard basis of $\R^\pures$ to write everything in matrix notation, and distinguish vectors and covectors by writing
primal vectors $w\in\R^\pures$ as column vectors
and
dual vectors $\omega\in(\R^\pures)^\ast$ as row vectors.
The action $\braket{\omega}{w}$ of $\omega$ on $w$ is then given by the matrix product $\omega\,w = \insum_{\pure} \omega_{\pure} w_{\pure}$.

After mild manipulations, the definitions of \acl{NE} \eqref{eq:Nash}, positive correlation \eqref{eq:PC}, potential games \eqref{eq:potential} and contractive games \eqref{eq:contract} can be expressed in the form $\insum_{\pure} \payv_{\pure}(x) z_{\pure}$, where $z$ is a tangent vector.
Put differently, the payoff ``vector'' $(\payv_{\pure}(x))_{\pure \in \pures}$ acts as a linear functional on displacement vectors, and so should be regarded as a \emph{covector}.  This is why we represent payoffs $\payv(x)$ in matrix notation as row vectors.

\begin{example}
\label{ex:potential2}
The defining property \eqref{eq:potential} of potential games can be expressed as
\begin{equation}
\label{eq:potential2}
\braket{D\pot(x)}{z}
	= \braket{\payv(x)}{z}
	\quad
	\text{for all $z \in \R^{\pures}$ and all $x\in\strat$.}
\end{equation}
On the \acl{LHS}, $D\pot(x)$ denotes the derivative of $\pot$ at $x$, a linear functional that acts on tangent vectors $z\in\R^{\pures}$ to yield the \emph{directional derivative} $\pot'(x;z)$.
Thus, \eqref{eq:potential2} can be expressed as an equality between covectors, viz.
$\payv(x) = D\pot(x)$.
\end{example}

Returning to our derivation of game dynamics, our aim in what follows is to find a vector field $x \mapsto \dynfield(x) \in \tcone(x)$ that agrees to the greatest possible extent with the payoff covector field $x \mapsto \payv(x)$, where this  ``agreement'' is defined in terms of the Riemannian metric $g$. The derivation requires two steps:
\begin{inparaenum}%
[\itshape i\upshape)]
\item
using a canonical transformation to convert the covector field into a vector field;
and
\item
projecting this field onto the cone of admissible vectors of motion.
\end{inparaenum}

For the first step, fix a Riemannian metric $g$ on $\orthant$ that is extendable to $\clorthant$ as defined in \cref{sec:boundary}.
The \emph{primal equivalent} of a covector $\omega \in(\R^\pures)^\ast$ at $x\in\clorthant$ is the (necessarily unique) vector $\omega^\sharp \in \domg(x)$ such that
\begin{equation}
\label{eq:sharp}
\braket{\omega}{w}
	=\product{\omega^{\sharp}}{w}_x	
	\quad
	\text{for all $w\in\domg(x)$}.
\end{equation}
In matrix notation, it is easy to verify that 
\begin{equation}
\label{eq:sharp-coords}
\omega^\sharp
	= g^{\sharp}(x) \omega^{\ttop},
\end{equation}
in agreement with the definition $\payv^\sharp(x) = g^{\sharp}(x) \payv(x)^{\ttop}$ from \eqref{eq:sharp-vecs}. 

For the second step, we transform each vector $\payv^\sharp(x)$ into a $g$-admissible vector by projecting it onto $\metcone(x)$.
Specifically, for all $x \in \strat$ and $w\in\domg(x)$, the \emph{projection} of $w$ at $x$ is defined as
\begin{equation}
\label{eq:tproj}
\tproj_{x}(w)= \argmin_{z\in\metcone(x)} \norm{w - z}_{x}.
\end{equation}
The induced Riemannian dynamics are then defined as
\begin{equation}
\label{eq:RGD-proj}
\dot x
	= \tproj_{x}(\payv^{\sharp}(x)).
\end{equation}

When $x \in \intstrat$ is interior, we have $\metcone(x)=\zspace^{\pures}$ by default, so $\tproj_{x}(w)$ is simply the orthogonal projection of $w\in\domg(x) = \R^{\pures}$ onto $\zspace^{\pures}$ with respect to $g$.
Accordingly, $\tproj_{x}(w)$ can be computed by finding a normal vector to $\zspace^{\pures}$ and subtracting this vector's contribution to $w$ (as in the first step of the Gram-Schmidt orthonormalization process).
To carry this out, observe that $\insum_{\pure} z_{\pure} = 0$ for all $z\in\zspace^{\pures}$, so the vector $\normal(x) = g^{\sharp}(x)\onecol$ defined in \eqref{eq:sharp-vecs} satisfies
\begin{equation}
\label{eq:CheckNormal}
\product{\normal(x)}{z}_{x}
	= \normal(x)^{\ttop} g(x) z
	= \onecol^{\ttop} g^{\sharp}(x) g(x) z
	= 0
	\quad
	\text{for all $z\in\zspace^{\pures}$}.
\end{equation}
This shows that $\normal(x)$ is a \emph{normal vector} to $\tspace_{\strat}(x)$ with respect to $g(x)$.
Thus, for all $x\in\intstrat$, we can express the \acl{RHS} of \eqref{eq:RGD-proj} as
\begin{flalign}
\label{eq:tproj2}
\tproj_{x}(\payv^{\sharp}(x))
	&= \payv^{\sharp}(x) - \proj_{\normal(x)} \payv^{\sharp}(x)
	= \payv^{\sharp}(x) - \frac{\product{\normal(x)}{\payv^{\sharp}(x)}_{x}}{\norm{\normal(x)}_{x}^{2}} \normal(x),
\end{flalign}
in agreement with \eqref{eq:RGD-sharp}.

More generally, for any state  $x \in \strat$ we have
\begin{flalign}
\label{eq:GeoRGDComp}
\tproj_{x}(\payv^{\sharp}(x))
	&= \argmin_{z\in\metcone(x)} \norm{\payv^{\sharp}(x) - z}_{x}
	\notag\\
	&= \argmin_{z\in\metcone(x)}
	\bracks[\Big]{\norm{\payv^{\sharp}(x)}^2_{x} +\norm{z}^2_{x} -2\product{\payv^{\sharp}(x)}{z}_{x}}
	\notag\\
	&= \argmax_{z\in\metcone(x)}
	\bracks[\Big]{\product{\payv^{\sharp}(x)}{z}_{x}-\tfrac12\norm{z}^2_{x} -\tfrac12\norm{\payv^{\sharp}(x)}^2_{x}}
	\notag\\
	&= \argmax_{z\in\metcone(x)}
	\bracks[\Big]{\braket{\payv(x)}{z} -\tfrac12\norm{z}^2_{x}}.
\end{flalign}
The dynamics \eqref{eq:RGD-proj} and \eqref{eq:RGD} are therefore identical.
We will take advantage of this geometric representation of \eqref{eq:RGD} freely in what follows.

\begin{remark}
In addition to building on Kimura's and Shahshahani's derivations of the replicator dynamics, the dual representations of Riemannian game dynamics have a close analogue in a class of game dynamics called \emph{target projection dynamics}  \citep{FBMTG94,San05}.
These dynamics are defined on $\strat$ as
\begin{equation}
\label{eq:TPJ}
\dot x
	= \argmin_{x'\in\strat} \norm{\payv(x) - x'}_{2}^{2} - x,
\end{equation}
Using a version of \eqref{eq:GeoRGDComp}, \cite{TV09} showed that \eqref{eq:TPJ} can also be expressed as 
\begin{equation}\label{eq:TPJ2}
\dot x
	= \argmax_{x'\in\strat}
	\bracks[\Big]{\braket{\payv(x)}{x'} - \tfrac{1}{2} \norm{x' - x}_{2}^{2}} - x.
\end{equation}
\end{remark}

\section{Examples}
\label{sec:examples}

We now present a variety of examples of Riemannian game dynamics.
We start by
extending the interior expressions \eqref{eq:PD-interior} and \eqref{eq:RD-interior} for our
two prototypical dynamics
to allow for boundary states:

\begin{example}
[Replicator dynamics revisited]
\label{ex:RD-full}
Let $g$ be the Shahshahani metric, so $g^\sharp_{\pure\purealt}(x) = \delta_{\pure\purealt}x_{\purealt}$, $\normal_{\pure}(x) = x_{\pure}$, and $\payv^{\sharp}_{\pure}(x)=x_{\pure} \payv_{\pure}(x)$.
Since $g$ is minimal-rank extendable, \eqref{eq:RGD-coords} yields the (continuous) Riemannian dynamics
\begin{equation}
\tag{\ref*{eq:RD}}
\dot x_{\pure}
	= x_{\pure} \left[ \payv_{\pure}(x) - \insum_{\purealt} x_{\purealt} \payv_{\purealt}(x) \right],
\end{equation}
which are the replicator dynamics of \cite{TJ78}.
The dynamics' continuity is reflected in the fact that the formula \eqref{eq:RD} is valid throughout $\strat$.
\end{example}

\begin{example}
[Projection dynamics revisited]
\label{ex:PD-full}
Let $g$ be the Euclidean metric.
Starting from formulation \eqref{eq:RGD-proj}, \cite{LS08} derived the following representation of the associated (discontinuous) Riemannian dynamics:
\begin{equation}
\label{eq:PDFormula}
\dot x_{\pure}
	= \begin{cases}
	\payv_{\pure}(x) - \abs{\pures(x)}^{-1} \sum_{\purealt\in\pures(x)} \payv_{\purealt}(x)
		&\quad
		\text{if $\pure\in\pures(x)$,}
		\\
	0
		&\quad
		\text{otherwise,}
	\end{cases}
\end{equation}
where $\pures(x)$ is a subset of $\pures$ that maximizes the average $\abs{\pures'}^{-1} \sum_{\purealt\in\pures'} \payv_{\purealt}(x)$ over all subsets $\pures'\subset\pures$ that contain $\supp(x)$.
These are the projection dynamics \eqref{eq:PD} of \cite{NZ97}.
The discontinuity of \eqref{eq:PD} is reflected in the appearance of $\supp(x)$ in \eqref{eq:PDFormula} via the definition of $\pures(x)$.
\end{example}

\begin{remark}
The dynamics \eqref{eq:RD} and \eqref{eq:PD} highlight an important qualitative difference between Shahshahani and Euclidean projections, which is representative of continuous and discontinuous Riemannian dynamics respectively.
The replicator dynamics \eqref{eq:RD} comprise a Lipschitz continuous dynamical system on $\strat$ which preserves the face structure of $\strat$, in that the relative interior of each face of $\strat$ remains invariant.
By contrast, the projection dynamics \eqref{eq:PD} may fail to be continuous at the boundary of $\strat$.
Thus, the relevant notion of a solution to \eqref{eq:PD} is that of a \emph{Carathéodory solution}, which allows for kinks at a measure zero set of times.
As a result, solutions of \eqref{eq:PD} may leave and re-enter the relative interior of any face of $\strat$ in perpetuity.
\end{remark}

The next example generalizes the previous two:

\begin{example}
[The $p$-replicator dynamics]
\label{ex:pReplicator}

For $p \geq 0$,
 let $g_{\pure\purealt}(x) = \delta_{\pure\purealt}x_{\purealt}^{-p}$ denote the $p$-Shahshahani metric introduced in \cref{ex:metric-pShah}.
We then have
$g_{\pure\purealt}^{\sharp}(x) = \delta_{\pure\purealt} x_{\purealt}^{p}$,
$\normal_{\pure}(x) = x_{\pure}^{p}$,
and $\payv_{\pure}^{\sharp}(x) = x_{\pure}^{p} \payv_{\pure}(x)$.
Thus, \eqref{eq:RGD-coords} yields the \emph{$p$-replicator dynamics}
\begin{equation}
\label{eq:pRep}
\dot x_{\pure}
	= x_{\pure}^p\!
	\left(
	\payv_{\pure}(x) -\frac{\sum_{\purealt\in\pures} x_{\purealt}^p \payv_{\purealt}(x)}{\sum_{\purealt\in\pures} x_{\purealt}^p }
	\right),
\end{equation}
valid for all interior $x\in\intstrat$.

These dynamics were first defined by \cite{Har11}.
Since $g$ is minimal-rank extendable if and only if $p > 0$, the dynamics are defined throughout $\strat$ via \cref{eq:pRep} for precisely these values of $p$.
However, the dynamics are only Lipschitz continuous for $p\geq1$; see~\cref{ex:HR} and \cref{sec:wp} below.
Three values of $p$ are worth highlighting:
\begin{enumerate}
\item
For $p=0$, we obtain the projection dynamics \eqref{eq:PD}.

\item
For $p=1$, we obtain the replicator dynamics \eqref{eq:RD}.

\item
For $p=2$, we obtain the \emph{log-barrier dynamics}, a system first examined by \cite{BL89} in the context of convex programming.%
\footnote{The reason for this name is explained in \cref{ex:HR}; see especially \cref{eq:h1h2}.}
\end{enumerate}

In the above dynamics, the value of $p$ parametrizes the costs of changes in the use of less common strategies.
This is illustrated in \cref{fig:portraits}, which presents a collection of $p$-replicator phase portraits 
in standard \acl{RPS}: 
\begin{equation}
\label{eq:RPS}
A =
	\left(
	\begin{array}{ccc}
	0	& -1	& 1
	\\
	1	& 0	& -1
	\\
	-1	& 1	& 0
	\end{array}
	\right).
\end{equation}
When $p = 0$, displacement costs are independent of the current state; thus the circular form of the payoffs \eqref{eq:RPS} generates circular closed orbits, subject to feasibility constraints (\cref{fig:portraits-Eucl}).
As $p$ increases, the costs of motion for uncommon strategies become more important relative to the game's payoffs (\cref{fig:portraits-rep,fig:portraits-qrep}).
As a direct consequence, the closed orbits of the dynamics are ``flattened'' near each face of the simplex, and are
ultimately reshaped into a nearly triangular form (\cref{fig:portraits-pprep}).%
\footnote{That all of these dynamics feature closed orbits is not coincidental \textendash\ see \cref{prop:conservative}.}

\cref{fig:portraits} also illustrates a basic dichotomy between continuous and discontinuous Riemannian dynamics.
In the discontinuous regime ($p = 0$), there is a unique forward solution from every initial condition in $\strat$.
However, solutions may enter and leave the boundary of $\strat$, and solutions from different initial conditions can merge in finite time.
In the smooth regime ($p\geq1$), solutions exist and are unique in forward and backward time, and the support of the state remains fixed along each solution trajectory.
Existence and uniqueness of solutions is treated formally in \cref{sec:wp}.
\end{example}

\begin{figure}[t]
\centering
\subfigure[$p=0$ (projection)]
{\includegraphics[width=.45\textwidth]{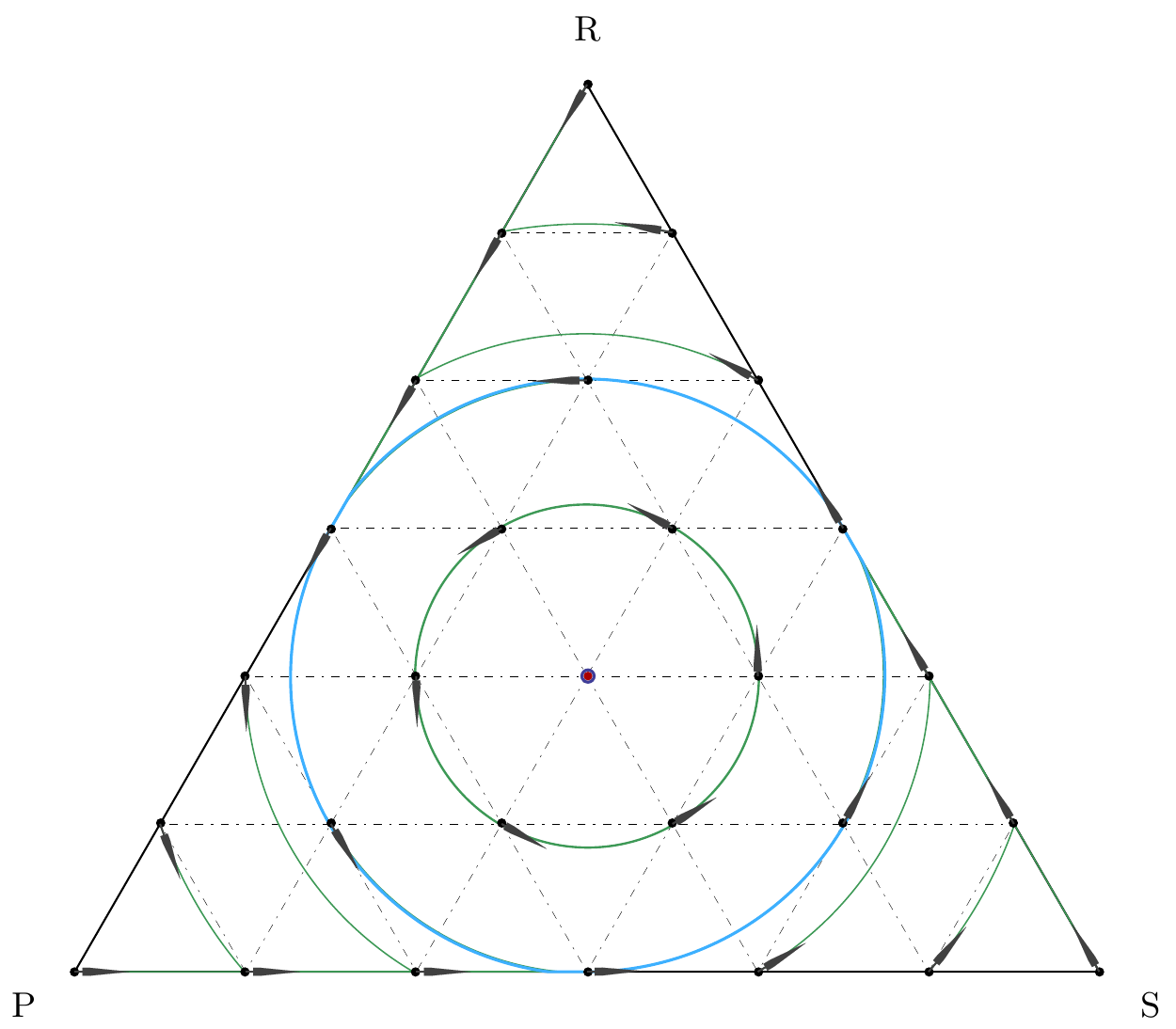}
\label{fig:portraits-Eucl}} 
\hfill
\subfigure[$p=1$ (replicator)]
{\includegraphics[width=.45\textwidth]{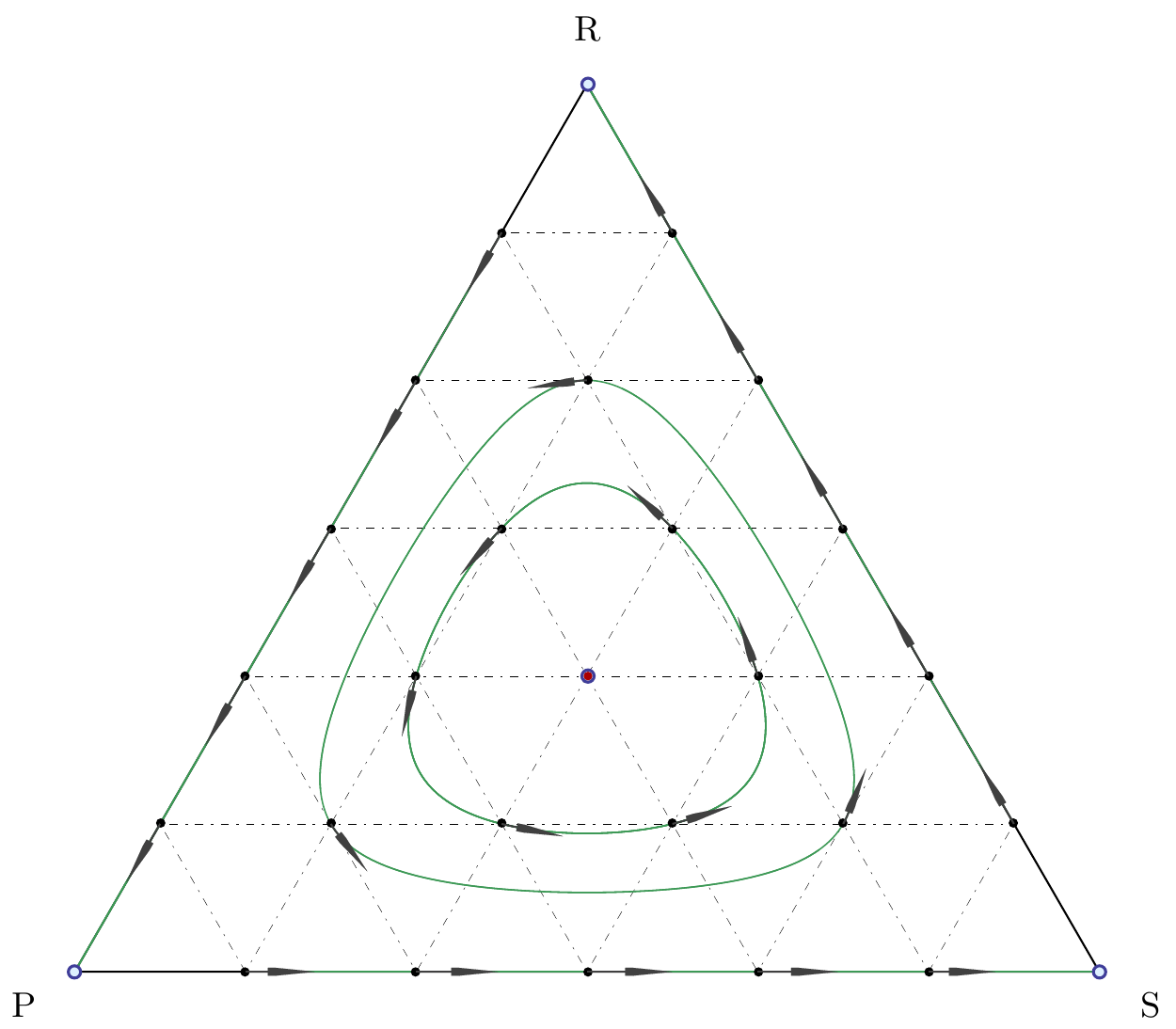}
\label{fig:portraits-rep}}
\\
\subfigure[$p=3/2$]
{\includegraphics[width=.45\textwidth]{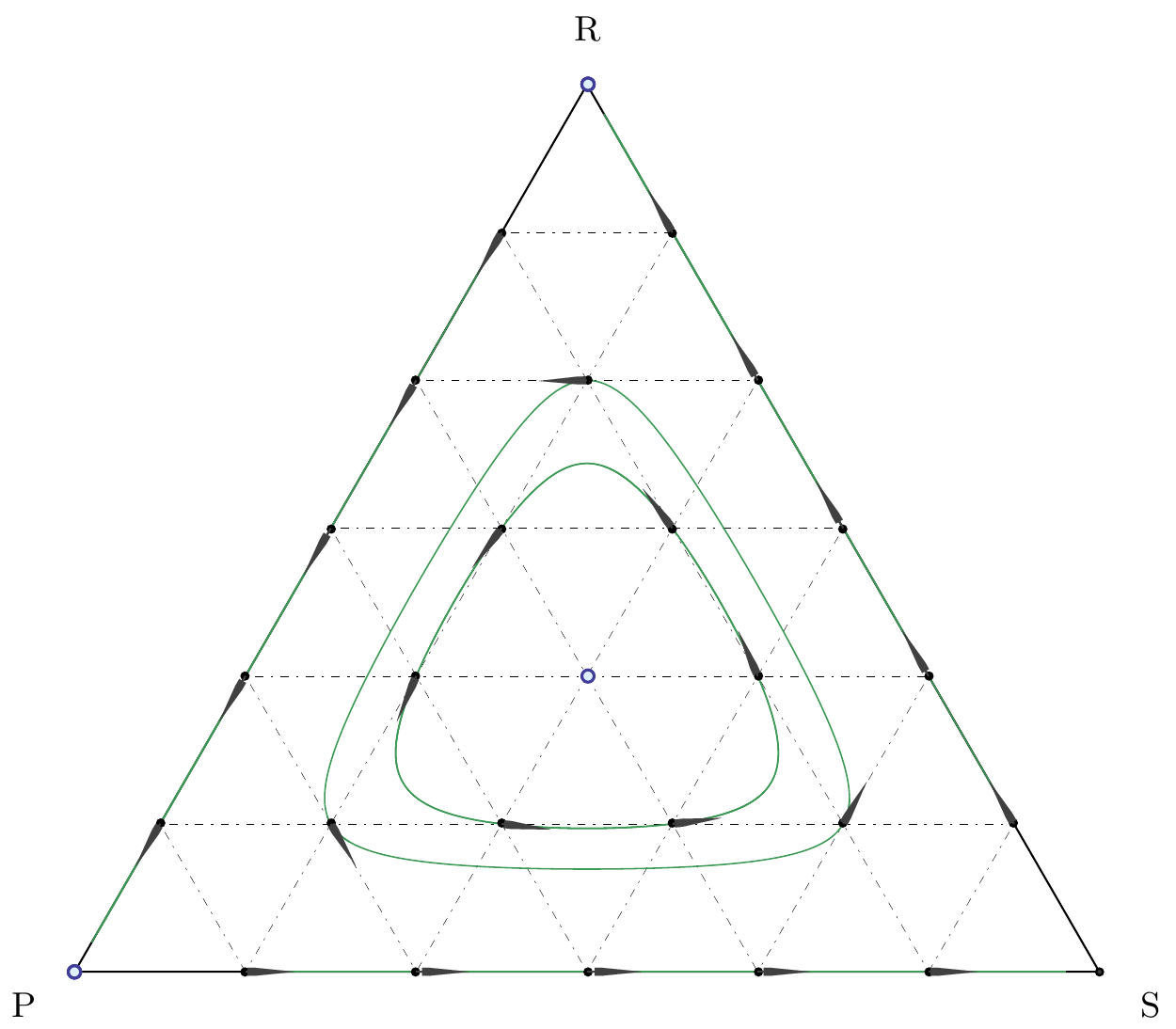}
\label{fig:portraits-qrep}}
\hfill
\subfigure[$p=5$]
{\includegraphics[width=.45\textwidth]{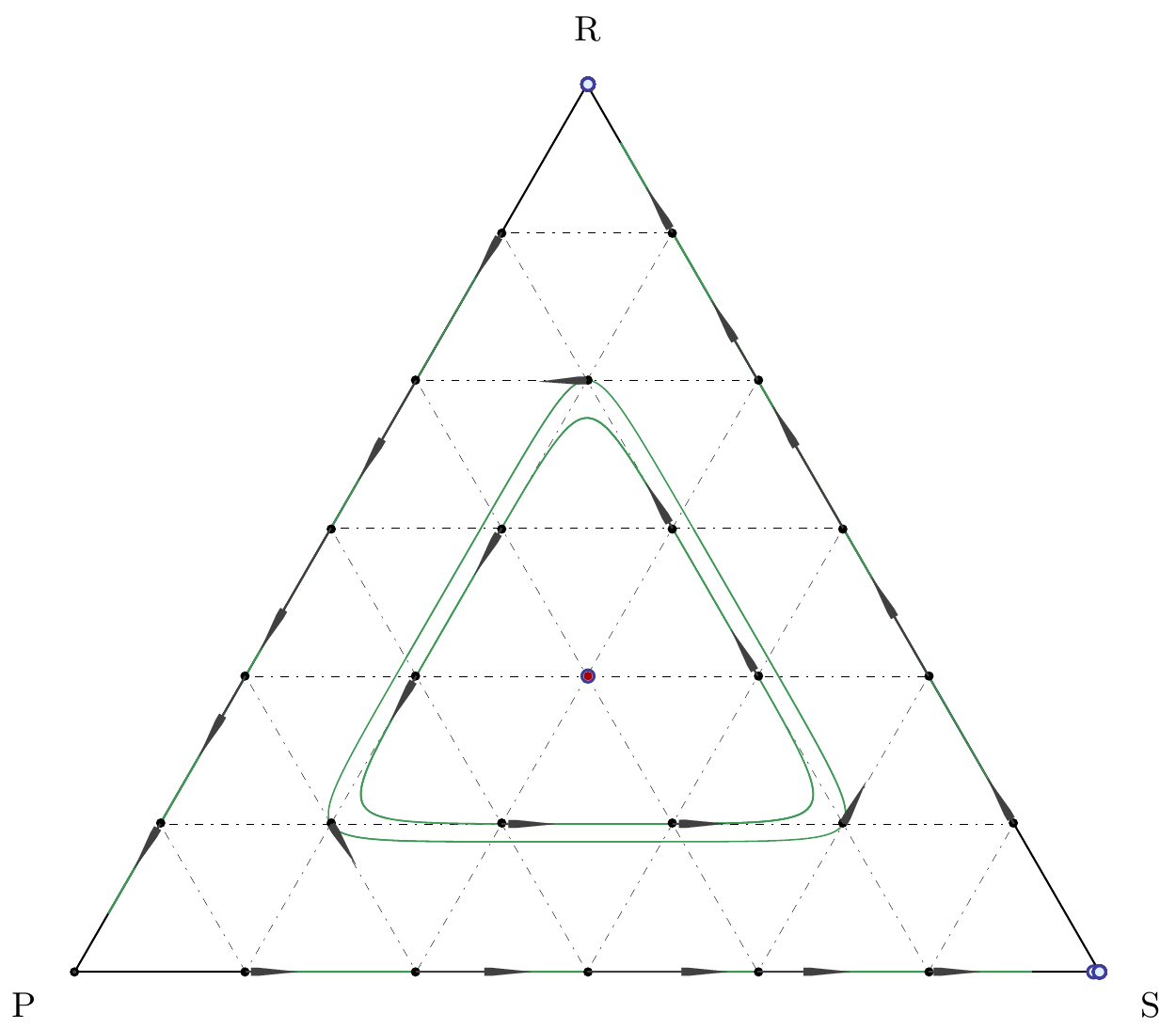}
\label{fig:portraits-pprep}}%
\caption{\small
Phase portraits of the $p$-replicator dynamics in standard \acl{RPS}.
As $p \in [0, \infty)$ increases, the shape of the closed orbits changes from circular to triangular.
When $p = 0$, solutions enter and leave the boundary of the simplex, but forward solutions exist and are unique.
For $p \geq 1$, forward and backward solutions exist, are unique, and their support is constant.
}
\label{fig:portraits}
\end{figure}

\begin{example}[Separable metrics and their dynamics]
\label{ex:separable}
A Riemannian metric $g$ on $\orthant$ is called \emph{separable} if its metric tensor is of the form
\begin{equation}
\label{eq:separable}
g(x)
	= \diag(1/\wt(x_1),\dotsc,1/\wt(x_n)),
\end{equation}
where $\wt \from [0, \infty) \to [0, \infty)$ is a continuous \emph{weighting function}  that is strictly positive on $(0, \infty)$.
For such metrics, we readily get
\begin{equation}
\label{eq:sharp-separable}
g^{\sharp}(x)
	= \diag(\wt(x_1),\dotsc,\wt(x_n)),
\end{equation}
so $g$ is minimal-rank extendable if $\lim_{z \to 0^{+}}\phi(z) = 0$ and full-rank extendable otherwise.

When \eqref{eq:RGD-coords} applies, the dynamics induced by $g$ take the form 
\begin{equation}
\label{eq:SepDyn2}
\dot x_{\pure}
	= \wt(x_{\pure})
	\bracks*{\payv_{\pure}(x) - \frac{\sum_{\purealt} \wt(x_{\purealt}) \payv_{\purealt}(x)}{\sum_{\purealt} \wt(x_{\purealt})}}.
\end{equation}
Ignoring the dynamics' behavior at the boundary, \eqref{eq:SepDyn2} was studied by \cite{Har11} under the name \emph{escort replicator dynamics}, and was further examined by \cite{MS16} and \cite{BM17} in the context of game-theoretic learning (see \cref{sec:RL}).  It is clear that the construction above generalizes immediately to allow different weighting functions for different strategies.
\end{example}

Moving beyond the separable case, Riemannian dynamics can also capture the effects of intrinsic relationships among the game's strategies.  

\begin{example}[Nested replicator dynamics]
\label{ex:nested}
In \cref{ex:metric-nested}, we defined the nested Shahshani metric as 
\begin{equation}
g_{\pure\purealt}(x)
	= \begin{cases}
	\frac{\delta_{\pure\purealt}}{x_{\pure}} + s \frac{1}{x_{\class{\pure}}}
		&\quad
		\text{if $\purealt\in\class{\pure}$},
		\\
	0
		&\quad
		\text{otherwise},
	\end{cases}
\end{equation}
where $\pures_{1},\dotsc,\pures_{m}$ is a partition of $\pures$ into groups 
of intrinsically similar strategies,
$\class{\pure}$ denotes the group containing strategy $\pure$,
$x_{\class{\pure}} = \sum_{\purealt\in\class{\alpha}} x_{\purealt}$,
and $s$ is a positive constant.
A straightforward calculation shows that
\begin{equation}
\label{eq:sharp-nested}
g_{\pure\purealt}^{\sharp}(x)
	= \begin{cases}
	 x_{\pure} \delta_{\pure\purealt} - \frac{s}{1+s} \frac{x_{\pure} x_{\purealt}}{x_{\class{\pure}}}
		&\quad
		\text{if $\purealt\in\class{\pure}$},
		\\
	0
		&\quad
		\text{otherwise}.
	\end{cases}
\end{equation}
It is evident from \eqref{eq:sharp-nested} that the metric $g$ is minimal-rank extendable.
Applying \eqref{eq:RGD-coords}, we find that $g$ generates the \emph{nested replicator dynamics}:
\begin{equation}
\label{eq:NRD}
\tag{NRD}
\dot x_{\pure}
	= x_{\pure} \bracks*{
	 \frac{s}{1+s} \left(\payv_{\pure}(x) - \frac{1}{x_{\class{\pure}}} \sum_{\purealt\in\class{\pure}} x_{\purealt} \payv_{\purealt}(x) \right)	
	 +\frac{1}{1+s}
	\left(\payv_{\pure}(x) - \sum_{\purealt\in\pures} x_{\purealt} \payv_{\purealt}(x) \right)
	}
\end{equation}
if $x_\pure>0$ and $\dot x_{\pure} = 0$ otherwise.

The imitative dynamics \eqref{eq:NRD} were introduced by \cite{MSNRD} to model settings in which agents assess strategies using two distinct procedures: at rate $\frac{s}{1+s}$, an agent only compares the payoff of his current strategy $\alpha$ to those of strategies in group $[\alpha]$; at rate $\frac{1}{1+s}$, they compare the payoff of their current strategy to that of all other strategies.

\cref{fig:nested} presents phase diagrams of the dynamics \eqref{eq:NRD} with $s = 3$ in the standard Rock-Paper-Scissors game.  The two panels illustrate the consequences of two similarity groupings.  In each case, the longest ``side'' of each closed orbit corresponds to the pair of similar strategies, which are switched between more easily than the remaining pairs of dissimilar strategies.
\end{example}

\begin{figure}[t]
\centering
\subfigure
[$\pures_1 =\{\mathrm{R},\mathrm{P}\}$, $\pures_2 =\{\mathrm{S}\}$ ]
{\includegraphics[width=.45\textwidth]{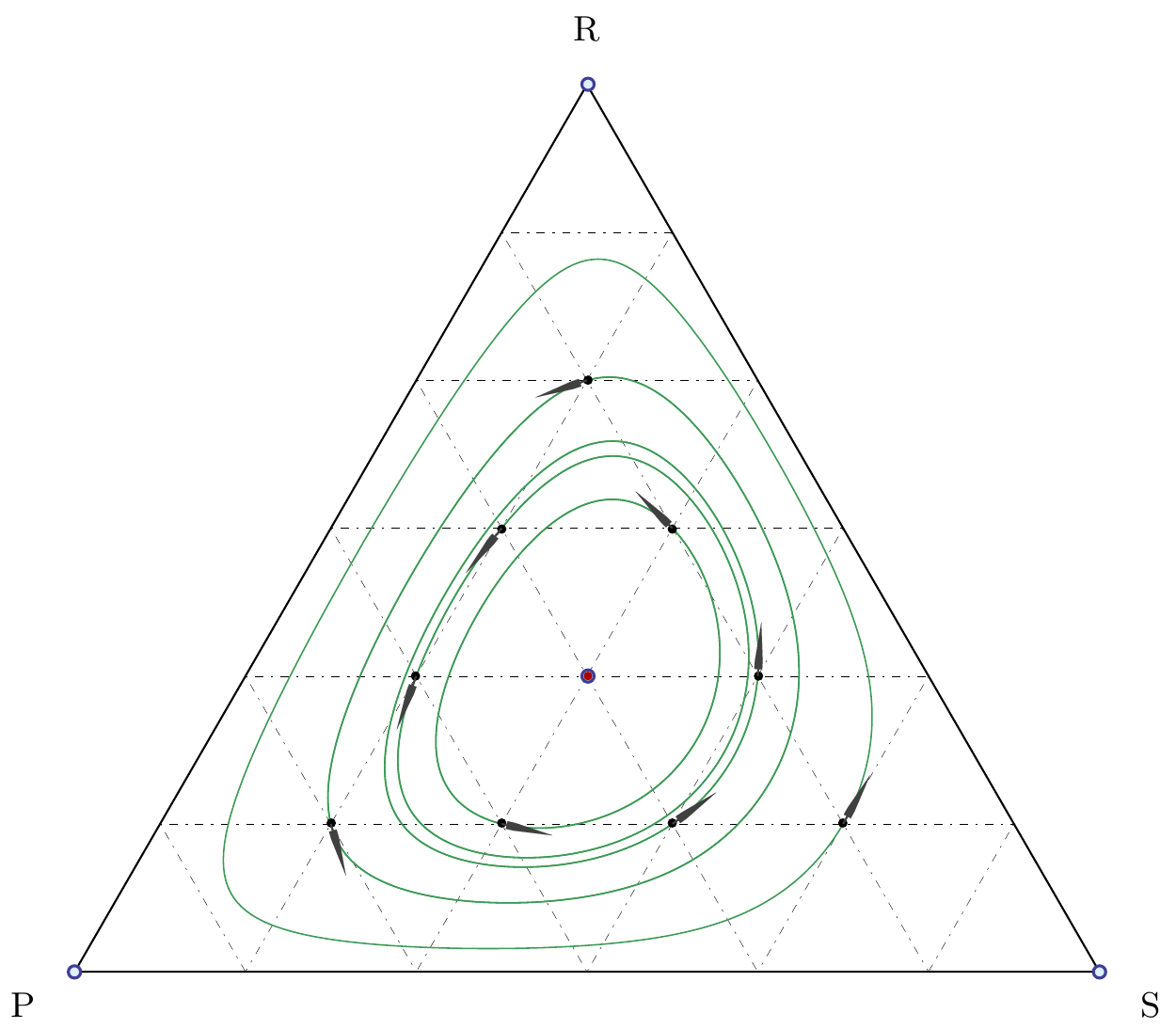}}
\hfill
\subfigure
[$\pures_1 =\{\mathrm{R},\mathrm{S}\}$, $\pures_2 =\{\mathrm{P}\}$]
{\includegraphics[width=.45\textwidth]{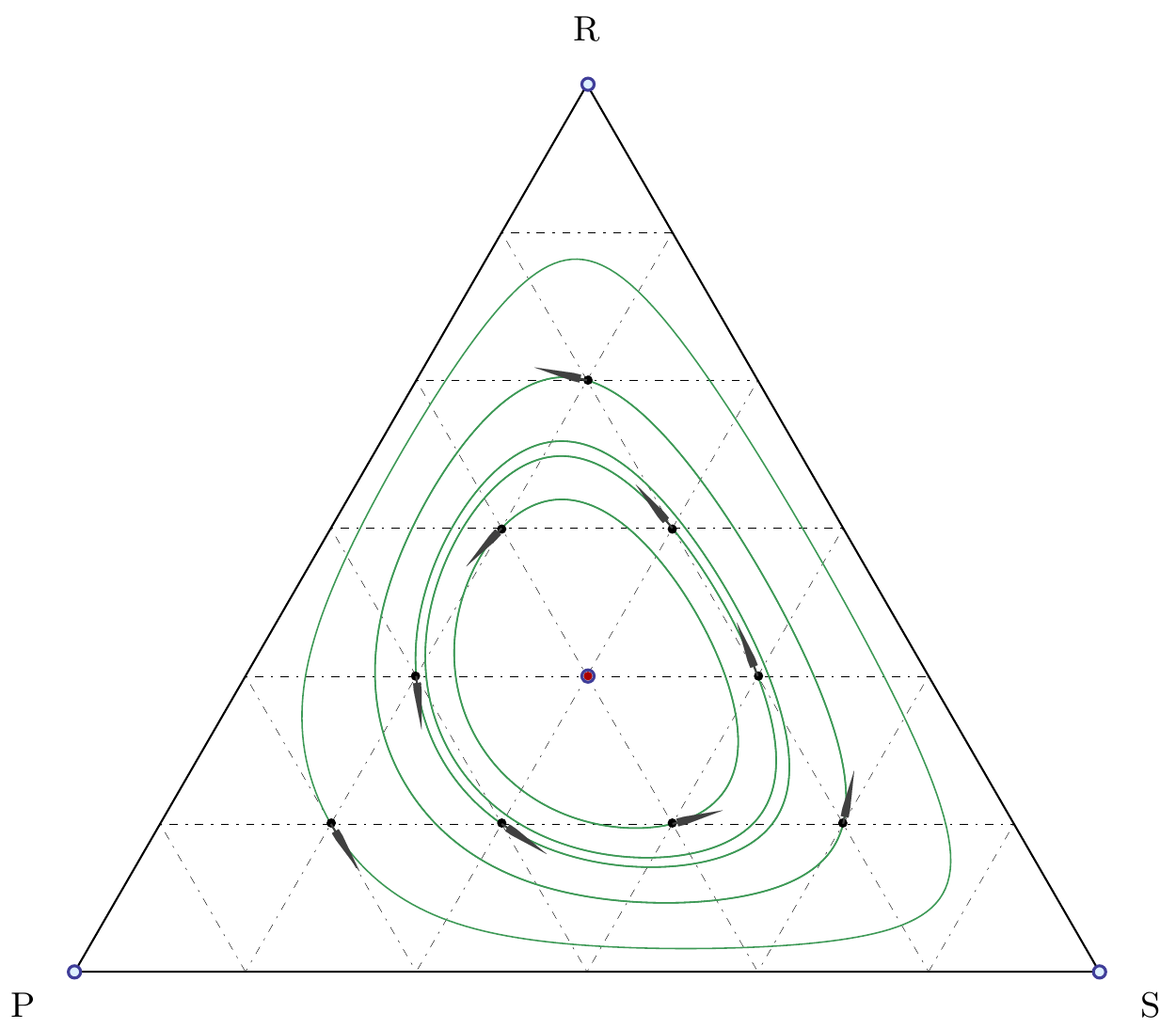}}
\caption{Phase portraits of the nested replicator dynamics in standard \acl{RPS} with $s = 3$ for two similarity groupings.}
\label{fig:nested}
\end{figure}

The following class of dynamics incorporates all of our previous examples. It is examined at depth in Sections \ref{sec:HD} and \ref{sec:RL}:

\begin{example}
[\acl{HR} metrics and their dynamics]
\label{ex:HR}
A generalization of the above class of examples can be obtained by considering Riemannian metrics that are defined as Hessians of convex functions.%
\footnote{For the origins of the idea in geometry, see \cite{Dui01} and references therein;
for applications to convex programming, see \cite{BT03} and \cite{ABB04}.}
To that end, let $h\from\clorthant\to\R$ be a continuous function on $\clorthant$ such that
\begin{enumerate}
[\textup(\itshape i\textup)]
\item
$h$ is $C^{3}$-smooth on every positive suborthant of $\clorthant$.
\item
$\hess h(x)$ is positive definite for all $x\in\orthant$.
\end{enumerate}
Then, $h$ induces a natural Riemannian metric on $\orthant$ defined as
\begin{equation}
\label{eq:HR}
g = \hess h ,
\end{equation}
or, in components:
\begin{equation}
g_{\pure\purealt}(x)
	= \frac{\pd^{2} h(x)}{\pd x_{\pure} \pd x_{\purealt}}.
\end{equation}
When this is the case, we say that $g$ is a \acdef{HR} metric and we refer to $h$ as the \emph{metric potential} of $g$.

As an example, the metric \eqref{eq:metric-nested} that generates the nested dynamics \eqref{eq:NRD} is an \ac{HR} metric with potential
\begin{equation}
\label{eq:NRPotential}
h(x)
	= \sum_{\pure\in\pures} x_{\pure} \left(
	\log x_{\pure} + s \log x_{\class{\pure}}
	\right).
\end{equation}
Moreover, every separable metric of the form \eqref{eq:separable} is an \ac{HR} metric with potential
\begin{equation}
\label{eq:decomposable}
h(x)
	= \sum_{\pure\in\pures} \theta(x_{\pure})
\end{equation}
for some smooth function $\theta\from[0,+\infty)\to\R$ with $1/\theta''(z) = \wt(z)$.
In particular, for $p\notin\{1,2\}$, the $p$-replicator dynamics are generated by the potential
\begin{equation}
\label{eq:hp}
h_{p}(x) = \sum_{\pure\in\pures} \theta_{p}(x_{\pure})
	\quad
	\text{with}
	\quad
	\theta_p(z)
= \tfrac{1}{(p-1)(p-2)}z^{2-p},
\end{equation}
and for $p=1$ and $p=2$, the corresponding potential functions are
\begin{equation}
\label{eq:h1h2}
h_1(x)
	= \sum_{\pure\in\pures} x_{\pure} \log x_{\pure}
	\quad
	\text{and}
	\quad
h_{2}(x)
	= -\sum_{\pure\in\pures} \log x_{\pure},
\end{equation}
respectively.%
\footnote{It is possible to define the potential $h_{p}$ for all values of $p$ using a single formula.
Let $\theta_p(z) = (z^{2-p} + p(p-2)z - (p-1)^{2})/((p-1)(p-2))$ when $p \ne \{1, 2\}$, and define $\theta_1$ and $\theta_2$ by analytic continuation.
Linear and constant terms do not affect the resulting metric, and the explicit formulas for $\theta_1$ and $\theta_2$ follow from the fact that $\lim_{a \to 0} (z^{a} - 1)/a = \log z$.}
The values $p = 0$, $p =1$, and $p =2$ partition the class of $p$-replicator dynamics into seven cases whose properties we summarize in \cref{tab:pReplicator}.%
\footnote{When $p\geq2$, the potential $h_{p}$ becomes infinite on the boundary of $\clorthant$, violating a standing assumption for $h$;
we address this technicality in \cref{rem:GenHess}.
Also, the (negative) \emph{Tsallis entropy} \citep{Tsa88} mentioned in \cref{tab:pReplicator} is defined as $S_{q}(x) = (q-1)^{-1} \sum_{\pure} (x_{\pure}^{q} - x_{\pure})$ for $q\in(0,1)$.}


\begin{table}[tbp]
\centering
\renewcommand{\arraystretch}{1.4}
\small
\begin{tabular}{cccc}
\hline
\noalign{\vspace{1pt}}
$p $ 
	&\textsc{name}
	&\textsc{regularity}
	&\textsc{potential}
	\\
\hline
\hline
$0$ &projection
	&discontinuous
	&quadratic
	\\
\hline
$(0,1)$
	&------
	&not Lipschitz
	&power law
	\\
\hline
$1$	
	&replicator
	&smooth
	&Gibbs entropy
	\\
\hline
$(1,2)$
	&------
	&smooth
	&Tsallis entropy
	\\
	\hline
$2$
	&log-barrier	
	&smooth
	&logarithmic
	\\
\hline
$(2,\infty)$
	&------
	&smooth
	&inverse power law
	\\
\hline
\end{tabular}
\vspace{2ex}
\caption{\small
Regularity of the $p$-replicator dynamics and behavior of the metric potential function $h_p$.
}
\label{tab:pReplicator}
\end{table}

Definition \eqref{eq:HR} is an integrability condition on the matrix field $g$.
As with vector fields on simply connected domains, this can be characterized by a symmetry condition on the derivatives of $g$,%
\footnote{This characterization follows from the integrability condition for ordinary vector fields (i.e. symmetry of the Jacobian matrix) and the symmetry of $g(x)$.}
namely
\begin{equation}
\label{eq:integrability}
\frac{\partial g_{\pure\gamma}}{\partial x_\purealt}
	=\frac{\partial g_{\purealt\gamma}}{\partial x_\pure}
	\quad
	\text{for all $\pure,\purealt,\gamma\in\pures$}.
\end{equation}
Conditions \eqref{eq:HR} and \eqref{eq:integrability} differ fundamentally from integrability conditions appearing in previous work on game dynamics, which are imposed on the \emph{vector fields} that define the dynamics.%
\footnote{See \cite{HMC01b}, \cite{HS09}, and \cite{San10c, San14}.}
In \cref{sec:HD}, we show that this integrability property provides important theoretical tools for the analysis of the induced Riemannian dynamics,
which we call \emph{Hessian game dynamics}.
\end{example}

\section{Microfoundations via revision protocols}
\label{sec:protocols}

To provide microfoundations for deterministic game dynamics \eqref{eq:ED}, one typically specifies a stochastic revision process that induces \eqref{eq:ED} in the so-called ``mean field'' limit.
To do so, suppose that agents in the population are recurrently chosen at random and given the opportunity to switch strategies.
What agents do when facing such opportunities is described by a \emph{revision protocol} $\rho$ whose components $\rho_{\pure\purealt}(x,\pi)$ describe the rates at which $\pure$-strategists who have received revision opportunities switch to strategy $\purealt$, as a function of the current population state $x$ and payoff vector $\pi$.%
\footnote{\cite{Wei95} and \cite{BW96} introduce revision protocols for imitative dynamics.
\cite{San10,San15} extends this approach to more general classes of dynamics.}

Together, a population game $\game\equiv\game(\pures,\payv)$ and a revision protocol $\rho$ induce the \emph{mean dynamics}:
\begin{equation}
\label{eq:MD}
\tag{MD}
\dot x_{\pure}
	= \sum_{\purealt \ne \pure}
	\bracks[\big]{x_{\purealt} \rho_{\purealt\pure}(x, \payv(x)) - x_{\pure} \rho_{\pure\purealt}(x, \payv(x))},
\end{equation}
which describe the rate of change in the use of each strategy $\pure$ as the difference between inflows into $\pure$ from other strategies and outflows from $\pure$ to other strategies.
For a fixed protocol $\rho$, \eqref{eq:MD} can be viewed as a map from population games $\payv$ to laws of motion on $\strat$, as described in Section \ref{sec:ED}.%
\footnote{Solutions to \eqref{eq:MD} may further be viewed as approximations to the sample paths of stochastic evolutionary models generated by the game $\game$ and protocol $\rho$:
for a comprehensive treatment, see \cite{BW03} and \cite{RS13}.}

The prototype for this construction is, again, the replicator dynamics \eqref{eq:RD}.
Three well-known protocols that generate \eqref{eq:RD} are:
\begin{subequations}
\label{eq:RDF}
\begin{align}
\label{eq:RDF1}
\rho_{\pure\purealt}(x,\pi)
	&=x_{\purealt} \pi_{\purealt},
	\\
\label{eq:RDF2}
\rho_{\pure\purealt}(x,\pi)
	&=-x_{\purealt} \pi_{\pure},
	\\
\label{eq:RDF3}
\rho_{\pure\purealt}(x,\pi)
	&=x_{\purealt} \left[ \pi_{\purealt} -\pi_{\pure} \right]_{+},
\end{align}
\end{subequations}
with $\pi$ assumed nonnegative in \eqref{eq:RDF1} and nonpositive in \eqref{eq:RDF2}.%
\footnote{Since the replicator dynamics (and all Riemannian game dynamics) are invariant to equal shifts in all strategies' payoffs, these assumptions about payoffs are innocuous.}
The $x_{\purealt}$ appearing in the right-hand sides allows us to interpret \crefrange{eq:RDF1}{eq:RDF3} as imitative protocols, with a revising agent picking a candidate strategy by observing the choice and the payoff of a randomly chosen opponent.
The protocols differ in how payoffs determine the rates at which switches are consummated.
Protocols \eqref{eq:RDF1} and \eqref{eq:RDF2}, due to \cite{Wei95} and \cite{BW96}, are respectively called \emph{imitation of success} and \emph{imitation driven by dissatisfaction}.
In the former, imitation rates increase linearly in the opponent's payoff;
in the latter, imitation rates decrease linearly in the revising agent's own payoff.
Protocol \eqref{eq:RDF3} is due to \cite{Hel92} and \cite{Sch98}, and is called \emph{pairwise proportional imitation}.
Under \eqref{eq:RDF3}, a revising agent only considers switching if the opponent's payoff is higher than their own, and then does so at a rate proportional to the payoff difference.
Substituting any of these protocols into \eqref{eq:MD} and rearranging yields the replicator dynamics \eqref{eq:RD}.

We now show that the revision protocols from this example can be generalized to cover wider ranges of Riemannian game dynamics, focusing again on
interior population states:%
\footnote{Under minimal-rank extendible metrics, the result to follow also applies on the boundary. Handling boundary states under full-rank extendable metrics requires modifications of the sort described in \cite{LS08}, a direction we do not pursue here.}

\begin{proposition}
\label{prop:RieFound}
Let $g$ be an extendable Riemannian metric such that $g^{\sharp}(x)$ is nonnegative for all $x\in\intstrat$.
Then up to a change of speed, the following protocols generate \eqref{eq:RGD} as their mean dynamics on $\intstrat$:
\begin{subequations}
\label{eq:RGDF}
\begin{flalign}
\label{eq:RGDF1}
\rho_{\pure\purealt}(x,\pi)
	&= \frac{(g^{\sharp}(x)\onecol)_{\pure}}{x_{\pure}} \, (\pi g^{\sharp}(x))_{\purealt},
	\\
\label{eq:RGDF2}
\rho_{\pure\purealt}(x,\pi)
	&= -\frac{(\pi g^{\sharp}(x))_{\pure}}{x_{\pure}} \, (g^{\sharp}(x)\onecol)_{\purealt},
\intertext{where $\pi$ is assumed nonnegative in \eqref{eq:RGDF1} and nonpositive in \eqref{eq:RGDF2}.
In addition, if $g(x)$ is diagonal, the dynamics \eqref{eq:RGD} are also generated \textpar{up to a change of speed} by the protocol}
\label{eq:RGDF3}
\rho_{\pure\purealt}(x,\pi)
	&= \frac{g_{\pure\pure}^{\sharp}(x)}{x_{\pure}} \, g_{\purealt\purealt}^{\sharp}(x) [ \pi_{\purealt} -\pi_{\pure} ]_{+}.
\end{flalign}
\end{subequations}
\end{proposition}

\begin{proof}
Substitute \eqref{eq:RGDF1}--\eqref{eq:RGDF3} with $\pi =\payv(x)$, $\payv^{\sharp}(x) = (\payv(x) g^{\sharp}(x))^{\hspace{-1pt}\top}$, and $\normal(x) 
= (\onerow g^{\sharp}(x))^{\hspace{-1pt}\top}$ into \eqref{eq:MD} to obtain
\begin{equation}
\label{eq:SymRmd}
\dot x
	= \payv^\sharp(x)\,\sum_{\purealt\in\pures} \normal_{\purealt}(x)
	- \normal(x)\,\sum_{\purealt\in\pures} \payv^{\sharp}_{\purealt}(x).
\end{equation}
Changing the speed at state $x$ by dividing the right-hand side of \eqref{eq:SymRmd} by $s(x) 
=\insum_{\purealt} \normal_{\purealt}(x)$ yields form \eqref{eq:RGD-coords1} of \eqref{eq:RGD}.
\end{proof}

After a change of speed, the Riemannian dynamics \eqref{eq:RGD} take the symmetric form \eqref{eq:SymRmd}, and this symmetry is a source of the appealing properties of the dynamics established below.
By contrast, the random assignment of revision opportunities implies that, under the mean dynamics \eqref{eq:MD}, the outflow rate from each strategy $\pure$ to other strategies is proportional to the popularity $x_{\pure}$ of the original strategy, resulting in an expression that is not symmetric.
The factor $x_{\pure}$ appearing in the denominators in \eqref{eq:RGDF} also lets us recover the symmetric expression \eqref{eq:SymRmd} from \eqref{eq:MD}. 

The asymmetric treatment of current and candidate strategies under \eqref{eq:RGDF} is illustrated by our running examples:

\begin{example}[$p$-replicator dynamics]
Since $p$-replicator dynamics are generated by the Riemannian metric $g(x) = \diag(1/x_1^{p},\dotsc,1/x_n^{p})$, \eqref{eq:RGDF3} implies that these dynamics are induced by the revision protocols
\begin{equation}
\label{eq:pRepRP}
\rho_{\pure\purealt}(x,\pi)
	= x_{\pure}^{p-1} x_{\purealt}^{p} \, \pospart{\pi_{\purealt} -\pi_{\pure}}.
\end{equation}
When $p=1$, we have $x_{\pure}^{p-1} =1$ and $x_{\purealt}^p = x_{\purealt}$, so \eqref{eq:pRepRP} boils down to the pairwise proportional imitation protocol \eqref{eq:RGDF3} and induces the replicator dynamics \eqref{eq:RD}.
When $p = 0$, we have $x_{\pure}^{p-1} =x_{\pure}^{-1}$ and $x_{\purealt}^p = 1$, so \eqref{eq:pRepRP} gives
\begin{equation}
\label{eq:ProjRP}
\rho_{\pure\purealt}(x,\pi)
	= x_{\pure}^{-1} \, \pospart{\pi_{\purealt} -\pi_{\pure}},
\end{equation}
and induces the projection dynamics \eqref{eq:PD} on $\intstrat$. Protocol \eqref{eq:ProjRP} was introduced by \cite{LS08}, who interpret it as a model of ``revision driven by insecurity'': agents playing rare strategies are particularly likely to consider revising, while candidate strategies are chosen without regard for their current levels of use.
\end{example}

While the revision protocols \eqref{eq:RGDF} are capable of generating many Riemannian dynamics \eqref{eq:RGD}, one can sometimes construct simpler protocols that take advantage of the structure of smaller classes of Riemannian dynamics.
For the microfoundations of the nested replicator dynamics \eqref{eq:NRD} and extensions thereof, we refer the reader to \cite{MSNRD}.

\section{General properties}
\label{sec:analysis}

In this section, we derive some general results for \eqref{eq:RGD}.
In \cref{sec:wp} we state a basic but technically challenging result on the existence and uniqueness of solutions.
In \cref{sec:basic} we show that the dynamics exhibit positive correlation with the game's payoffs, and we characterize the dynamics' rest points as either restricted equilibria or \aclp{NE}.
Finally, in \cref{sec:GCPG} we study the global behavior of the dynamics in potential games.

\subsection{Existence and uniqueness of solutions}
\label{sec:wp}

To illustrate the possibilities for existence and uniqueness of solutions, it is useful to start with a simple example.
Specifically, consider the $p$-replicator dynamics of \cref{ex:pReplicator} for a $2$-strategy game with action set $\pures = \{1,2\}$ and payoff functions $\payv_{1}(x) = 1$, $\payv_{2}(x) = 0$.

When $p=1$, we obtain the toy replicator equation
\begin{equation}
\dot x_{1}
	= x_{1}(1-x_{1}).
\end{equation}
Solutions to this equation exist and are unique for all $t \in (-\infty, \infty)$, and the support of $x(t)$ is invariant.
The pure states $0$ and $1$ are both rest points, and it is easy to check that the unique solution with initial condition $x_{1}(0) = a \in (0,1)$ is $x_{1}(t) = a / [a + (1 - a) e^{-t}]$.

When $p=0$, we obtain the Euclidean projection dynamics
\begin{equation}\label{eq:ToyEPD}
\dot x_{1}
	= \begin{cases}
	1/2
		&\quad
		\text{if $x_{1} < 1$}
		\\
	0
		&\quad
		\text{if $x_{1} = 1$}.
	\end{cases}
\end{equation}
For every initial condition $x_{1}(0) \in [0,1]$, this equation admits the unique forward solution $x_{1}(t) =x_{1}(0)+ t/2$ for $t \in [0, 2(1-x_{1}(0)))$ and $x_{1}(t)=1$ thereafter.
Evidently, the support of $x(t)$ is not invariant;
also, backward solutions are not defined for all time, and solutions are not smooth in $t$ when $x_{1}=1$ is reached.

Finally, when $p=1/2$, we obtain the differential equation
\begin{equation}
\label{eq:non-Lipschitz}
\dot x_{1}
	= \frac{\sqrt{x_{1} (1-x_{1})}}{\sqrt{x_{1}} + \sqrt{1-x_{1}}}.
\end{equation}
Although this equation admits forward (and backward) solutions from every initial condition, these are no longer unique.
Starting at $x_{1}(0) = 0$, we have
the stationary solution $x_{1}(t) = 0$ for $t \in[0, \infty)$;
furthermore, one can verify by a direct \textendash\ albeit tedious \textendash\ calculation that there is another solution, namely $x_{1}(t) = \frac{1}{2} + \frac{t-2}{4} \sqrt{1 + t - t^{2}/4}$ for $t \in [0, 4)$ and $x_{1}(t) = 1$ thereafter.
Additional solutions may linger at $x_{1} = 0$ before emulating the previous solution trajectory.

The differences in behavior in the three cases above can be traced back to the properties of the underlying Riemannian metrics.
First, the replicator dynamics are generated by the Shahshahani metric, which is minimal-rank extendable to all of $\strat$.
In this case the induced dynamics \eqref{eq:RGD} are Lipschitz continuous, so existence and uniqueness of solutions is guaranteed by the Picard\textendash Lindelöf theorem (along with an argument to account for $\strat$ being closed).
Moreover, the support of $x(t)$ is constant, and solutions exist in both forward and backward time \cite[Theorems 4.A.5 and 5.4.7]{San10}.

On the other hand, the Euclidean projection dynamics \eqref{eq:PD} are generated by a full-rank extendable metric.
In such cases, the induced dynamics \eqref{eq:RGD} are typically discontinuous, so the relevant solution notion is that of a \emph{Carathéodory solution}, an absolutely continuous trajectory that satisfies \eqref{eq:RGD} for almost all $t\geq0$.
In the case of \eqref{eq:PD}, \cite{LS08} showed that every initial condition admits a unique Carathéodory forward solution;
however, different solution orbits can merge in finite time, as illustrated in the previous example and in \cref{fig:portraits-Eucl}.

The following proposition shows that this behavior of \eqref{eq:RD} and \eqref{eq:PD} is representative of the minimal-rank and full-rank extendable cases respectively:

\begin{proposition}
\label{prop:wp}
Let $g$ be an extendable Riemannian metric.
\begin{enumerate}
[\textup(i\textup)]
\item
\label{itm:wp-min}
If $g$ is minimal-rank extendable, \eqref{eq:RGD} admits a unique global solution from every initial condition in $\strat$;
moreover, each solution has constant support.

\item
\label{itm:wp-full}
If $g$ is full-rank extendable, \eqref{eq:RGD} admits a unique forward Carathéodory solution from every initial condition in $\strat$.
\end{enumerate}
\end{proposition}

\cref{prop:wp} justifies the terminology \emph{continuous} and \emph{discontinuous} that we introduced in \cref{sec:RieDDef} to refer to dynamics induced by minimal-rank and full-rank metrics.
The nontrivial part of \cref{prop:wp} is the proof of part \eqref{itm:wp-full}:
despite an apparent similarity, this result is considerably harder than the corresponding result of \cite{LS08} for \eqref{eq:PD}, so we relegate its proof to \cref{app:proofs}.
The main reason for this difficulty is that known 
uniqueness proofs for projected differential equations depend crucially on the Riemannian metric being constant throughout the dynamics' state space, an assumption that obviously fails here.

Of course, as can be seen from the continuous \textendash\ but not \emph{Lipschitz} continuous \textendash\ system \eqref{eq:non-Lipschitz}, \eqref{eq:RGD} may fail to admit unique solutions from initial conditions at the boundary of $\strat$ if the underlying metric does not admit a \emph{Lipschitz} continuous extension to the boundary of $\strat$.
To avoid the resulting complications, we do not consider dynamics that are continuous but not Lipschitz continuous in the rest of the paper.

\subsection{Basic properties}
\label{sec:basic}

We now establish some basic relationships between \eqref{eq:RGD} and the payoffs of the underlying game.
We first show that \eqref{eq:RGD} respects positive correlation:

\begin{proposition}
\label{prop:PC}
The dynamics \eqref{eq:RGD} satisfy \eqref{eq:PC}.
\end{proposition}

\begin{proof}
Let $\dynfield(x) = \tproj_{x}(\payv^{\sharp}(x))$.
We then claim that
\begin{equation}
\label{eq:PC-derivation}
\braket{\payv(x)}{\dynfield(x)}
	= \product{\payv^{\sharp}(x)}{\dynfield(x)}_x
	\geq \product{\tproj_{x}(\payv^{\sharp}(x))}{\dynfield(x)}_x
	= \norm{\dynfield(x)}^{2}_x
	\geq0,
\end{equation}
with equality if and only if $\dynfield(x) = 0$.
The only step in \eqref{eq:PC-derivation} needing justification is the first inequality.
For this step, we split the analysis into three cases.
First, if $x \in \intstrat$, the inequality binds because $\tproj_{x}$ orthogonally projects $ \R^{\pures}$ onto $\zspace^{\pures} = \metcone(x) $, which contains $\dynfield(x)$.
Second, if $x \in \bd(\strat)$ and $g$ is minimal-rank extendable, then $\payv^{\sharp}(x) \in \R^{\supp(x)}$, so the inequality binds because $\tproj_{x}$ projects $\R^{\supp(x)}$ orthogonally onto $\zspace^{\pures} \cap \R^{\supp(x)} = \metcone(x) $, which contains $\dynfield(x)$.
Finally, if $x \in \bd(\strat)$ and $g$ is full-rank extendable, $\tproj_{x}$ is the closest point projection of $\R^{\pures}$ onto the tangent cone $\tcone_{\strat}(x)$.
Hence, by Moreau's decomposition theorem \citep{HUL01}, we infer that $\payv^{\sharp}(x) - \tproj_{x}(\payv^{\sharp}(x))$ lies in the normal cone 
\begin{equation}
\label{eq:ncone}
\ncone_{\strat}(x)
	= \setdef{w\in\R^{\pures}}{\product{w}{z}_{x}\leq 0 \text{ for all } z\in \tcone_{\strat}(x)}.
\end{equation} 
Since $\dynfield(x)\in\metcone(x)= \tcone_{\strat}(x)$, the first inequality in \eqref{eq:PC-derivation} is immediate.
\end{proof}

\cref{prop:PC} is not particularly surprising:
after all, the basic postulate behind \eqref{eq:RGD} is that the dynamics' vector of motion is the closest feasible approximation to the game's payoff field, with the notion of closeness determined by the underlying Riemannian metric
(or, equivalently, cost function).
As we show below, this alignment can be exploited further to characterize the dynamics' rest points.

To that end, recall that one of the main attributes of the Euclidean projection dynamics \eqref{eq:PD} is \emph{Nash stationarity}:
\begin{equation}
\label{eq:NS}
\tag{NS}
\text{$\eq\in\strat$ is a rest point if and only if it is a \acl{NE}.}
\end{equation}
This property does not hold under the replicator dynamics:
for instance, every pure state of $\strat$ is stationary under \eqref{eq:RD}.
In this case, \eqref{eq:NS} is replaced by the notion of \emph{restricted stationarity}:%
\footnote{Recall here that $x^\ast$ is a \emph{restricted equilibrium} if all strategies in its support earn equal payoffs.}
\begin{equation}
\label{eq:RS}
\tag{RS}
\text{$\eq\in\strat$ is a rest point if and only if it is a restricted equilibrium.}
\end{equation}

Our next result shows that this difference between the projection and the replicator dynamics is representative of the discontinuous and continuous cases, and highlights one advantage of the former over the latter:

\begin{proposition}
\label{prop:stationary}
\leavevmode
\begin{enumerate}
[\textup(\itshape i\textup)]
\item
\label{itm:stationary-min}
Continuous Riemannian dynamics satisfy \eqref{eq:RS}. 
\item
\label{itm:stationary-full}
Discontinuous Riemannian dynamics satisfy \eqref{eq:NS}. 
\end{enumerate}
\end{proposition}

\begin{proof}
For \eqref{itm:stationary-min}, recall that the coordinate expression \eqref{eq:RGD-coords} for \eqref{eq:RGD} always holds when $g$ is minimal-rank extendable, and $\dot x_{\pure} = 0$ whenever $x_{\pure} = 0$.
Therefore, it suffices to check that $\eq\in\intstrat$ is a rest point if and only if all the components of $\payv(\eq)$ are equal.
To that end, note that $\eq$ is a rest point of \eqref{eq:RGD-coords} if and only if 
\begin{equation}
\label{eq:pfofrs}
\payv^{\sharp}(\eq)
	= \frac{\insum_{\gamma} \payv^{\sharp}_{\gamma}(\eq)}{\insum_{\gamma} \normal_{\gamma}(\eq)} \normal(\eq)
	\propto \normal(\eq).
\end{equation}
In turn, this means that $\eq$ is a rest point of \eqref{eq:RGD} if and only if $\payv^{\sharp}(\eq) \propto \normal(\eq)$;
our claim then follows from the fact that $g^{\sharp}(\eq)$ is invertible.

For \eqref{itm:stationary-full}, assume that $g$ if full-rank extendable and fix some $\eq\in\strat$.
It is easy to show that $\eq$ is a Nash equilibrium if and only if it satisfies the variational characterization
\begin{equation}\label{eq:Nash-variational} 
0	\leq
	\braket{\payv(\eq)}{x - \eq} = \product{\payv^{\sharp}(\eq)}{x-\eq}_{\eq}
	\quad
	\text{for all $x\in\strat$},
\end{equation}
which says that $\payv^{\sharp}(\eq)$ lies in the normal cone $\ncone_{\strat}(\eq)$ of $\strat$ at $\eq$ (cf.~Eq.~\ref{eq:ncone} above).
Moreau's decomposition theorem then yields $\payv^{\sharp}(\eq) \in \ncone_{\strat}(\eq)$ if and only if $\tproj_{\eq}(\payv^{\sharp}(\eq)) = 0$, so our assertion follows.
\end{proof}

\begin{remark}
We note without proof
that shifting all strategies' payoffs by the same amount has no effect on \eqref{eq:RGD},
and rescaling all strategies' payoffs by the same factor only changes the speed at which solution paths are traversed. 
In addition, 
on the face of $\strat$ spanned by a subset $\pures'$ of $\pures$, continuous dynamics are invariant to changes in the payoffs of strategies outside of $\pures'$.
\end{remark}

\subsection{Global convergence in potential games}
\label{sec:GCPG}

Recall here that $\game\equiv\gamefull$ is a potential game if $\payv_{\pure}(x) = \pd_{\pure}\pot(x)$ for some potential function $\pot\from\strat\to\R$ (cf.~\cref{ex:potential}).
It then follows from \cref{prop:PC} that $f$ is a \emph{strict global Lyapunov function} for \eqref{eq:RGD}, meaning that its value increases along \eqref{eq:RGD} whenever the dynamics are not at rest.%
\footnote{Definitions concerning stability and convergence are collected in \cref{app:stability}.}

For continuous Riemannian dynamics, a standard Lyapunov argument implies that all $\omega$-limit points of \eqref{eq:RGD} are rest points \textendash\ and hence, by \cref{prop:stationary}, restricted equilibria of $\game$.
However, this argument does not extend to discontinuous dynamics and \aclp{NE} because it requires continuity of solutions with respect to initial conditions, a requirement which is difficult to prove in our case.
To circumvent this obstacle, we establish a \ac{lsc} bound on the rate of change of the game's potential function.
This bound then allows us to apply \cref{prop:omega} in \cref{app:stability}, which shows that, for dynamics on a compact set, such a bound on the rate of change of a Lyapunov function guarantees global convergence.

\begin{proposition}
\label{prop:potential}
Let $\game$ be a potential game with potential function $\pot$.
Then, $\pot$ is a strict Lyapunov function for \eqref{eq:RGD} and every $\omega$-limit point of \eqref{eq:RGD} is a rest point of \eqref{eq:RGD}.
These are restricted equilibria if \eqref{eq:RGD} is continuous, and \aclp{NE} if \eqref{eq:RGD} is discontinuous.
\end{proposition}

\begin{proof}
Let $\dynfield(x) = \tproj_{x}(\payv^{\sharp}(x))$ and let $x(t)$ be a solution of \eqref{eq:RGD}.
Then, \cref{prop:PC} yields
\begin{equation}
\label{eq:Lyapunov-pot}
\frac{d}{dt}\pot(x(t))
	= \braket{D\pot(x(t))}{\dot x(t)}
	= \braket{\payv(x(t))}{\dynfield( x(t))}
	\geq 0,
\end{equation}
with equality if and only if $\dynfield(x(t)) = 0$.
Hence, $\pot$ is a strict global Lyapunov function for \eqref{eq:RGD}.

When \eqref{eq:RGD} is (Lipschitz) continuous, a standard argument shows that every $\omega$-limit point of \eqref{eq:RGD} is a rest point thereof \cite[see e.g.][Theorem 7.B.3]{San10}.
The discontinuous case however requires a different treatment.
To start, note that
\begin{equation}
\frac{d}{dt}\pot(x(t))
	= \braket{\payv(x(t))}{\dynfield( x(t))}
	= \product{\payv^{\sharp}(x)}{\tproj_{x}(\payv^{\sharp}(x))}_x
	\geq \norm{\dynfield( x(t))}^2_x
	\geq 0,
\end{equation}
where the first inequality follows from Moreau's decomposition theorem.
Both inequalities bind if and only if $\dynfield(x(t)) = 0$;
since the speed function $x \mapsto \norm{\dynfield(x)}_x$ is \acl{lsc} (cf. \cref{lem:RLSC}),
\cref{prop:omega} shows that every $\omega$-limit point of \eqref{eq:RGD} is a rest point.
\end{proof}

The classic analyses of \cite{Kim58} and \cite{Sha79} showed that in common interest games, average payoffs are increased at a maximal rate under the replicator dynamics, provided that ``maximal'' is defined with respect to the Shahshahani metric.
We conclude this section by deriving an analogous principle for all Riemannian game dynamics.
To state it, define the \emph{gradient} of a smooth function $f\from\orthant \to \R$ with respect to $g$ by
\begin{equation}
\label{eq:gradient}
\grad f (x)
	= (Df(x) \,g^{-1}(x))^\ttop,
\end{equation}
that is, as the (necessarily unique) vector satisfying
\begin{equation}
\braket{Df(x)}{z}
	= \product{\grad f(x)}{z}_x 
	\quad
	\text{for all $z\in \R^{\pures}$, $x\in\orthant$}.
\end{equation}
Geometrically, the vector $\grad f(x)$ represents the direction of maximal increase of the function $f$ at $x$ with respect to the metric $g$.%
\footnote{Specifically, this means that $\grad f(x) = \argmax\setdef{D_{z}f(x)}{\norm{z}_{x} = 1}$;
that this is so follows from the definition of $\grad f(x)$ and the Cauchy-Schwarz inequality.}
We then have:

\begin{proposition}
Let $\game$ be a potential game with potential function $\pot$ and let $g$ be an extendable Riemannian metric.
Then, for all $x\in\intstrat$, the vector field that defines \eqref{eq:RGD} is the projection of $\grad\pot$ onto $\tspace_{\strat}(x)$ with respect to $g$.
\end{proposition}

\begin{proof}
Since $\payv(x) = Df(x)$, we have $\tproj_{x}(\payv^{\sharp}(x)) = \tproj_{x}(\grad f(x))$, as claimed.
\end{proof}

Hence, at interior states, the dynamics \eqref{eq:RGD} increase the value of potential at a maximal rate under the geometry defined by $g$, subject to feasibility.
For discontinuous dynamics, this conclusion remains true even at boundary states.
For continuous dynamics, the interior of each face of $\strat$ is invariant under \eqref{eq:RGD}, so this conclusion holds provided that feasibility is understood to incorporate this additional constraint.

\section{Hessian game dynamics}
\label{sec:HD}

By virtue of the integrability property that defines them, potential games have desirable convergence properties under a wide range of evolutionary dynamics.
By contrast, convergence results for other classes of games \textendash\ for instance, contractive games and games with an \ac{ESS} \textendash\ require additional structure, often taking the form of integrability properties built into the dynamics themselves.%
\footnote{See \cite{HS07}, \cite{San10c}, and \cite{Zus18}.}

In this section, we show that the integrability of \acl{HR} metrics allows us to generalize several properties of the replicator dynamics and the Euclidean projection dynamics to a substantially broader class of dynamics.  These \emph{Hessian game dynamics},  introduced in \cref{ex:HR}, take the form
\begin{equation}
\label{eq:HD}
\tag{HD}
\dot x
	= \argmax_{z \in \metcone(x)} \bracks[\Big]{\braket{\payv(x)}{z} - \tfrac12\norm{z}^2_{x}},
	\quad
g
	= \hess h,
\end{equation}
where the continuous function $h\from\clorthant\to\R$ is $C^{3}$-smooth on every positive suborthant of $\clorthant$, and where $\hess h(x)$ is positive definite for all $x\in\orthant$.
As we demonstrate below, the integrability built into the dynamics  \eqref{eq:HD} is the source of a variety of stability and convergence results.

A key element of our analysis is the so-called \emph{Bregman divergence}, which we introduce in \cref{sec:Bregman}.
In \cref{sec:contractive}, we establish global convergence to equilibrium in contractive games and local stability of \acp{ESS}, while \cref{sec:averages} demonstrates the convergence of time-averages of interior trajectories to \acl{NE} and provides sufficient conditions for permanence.
Finally, \cref{sec:dominated} establishes the elimination of strictly dominated strategies under continuous Hessian dynamics.

\subsection{Bregman divergences}
\label{sec:Bregman}

When used as a tool for establishing convergence, Lyapunov functions typically measure some sort of ``distance'' between the current state and a target state $\base$.
For Hessian dynamics, a natural point of departure is the potential function $h$ of the metric $g=\hess h$ that defines them.
However, since the (game-specific) target state $\base$ is independent of $g$, there is no reason that $h$ itself should serve as a Lyapunov function.
Instead, taking advantage of the convexity of $h$, we consider the difference between $h(\base)$ and the best linear approximation of $h(\base)$ from $\notbase$.

Formally, the \emph{Bregman divergence} of $h$ \citep{Bre67} is defined as
\begin{equation}
\label{eq:Bregman}
\breg(\base,\notbase)
	= h(\base) - h(\notbase ) - h'(\notbase ;\base - \notbase),
	\qquad
	\base, \notbase\in \strat,
\end{equation}
where $h'(\notbase ;\base - \notbase )$ is the one-sided derivative of $h$ at $\notbase $ along $\base - \notbase $, i.e.
\begin{equation}
h'(\notbase ;\base - \notbase )
	= \lim_{t\to0^{+}} t^{-1} \left[ h(\notbase + t (\base - \notbase ) ) - h(\notbase ) \right].
\end{equation}
Since $h$ is convex, we have 
\begin{equation}\label{eq:BPD}
\breg(\base,\notbase) \geq 0,\:\text{ with equality if and only if }\base = \notbase.
\end{equation} 
On the other hand, $\breg$ is not symmetric in $\base$ and $\notbase$, so it is not a bona fide distance function on $\strat$;
rather, $\breg(\base,\notbase )$ describes the remoteness of $\notbase$ from the base point $\base$, hence the name ``divergence''.

Revisiting our two archetypal examples, the Euclidean metric is generated by the quadratic potential $h(x) = \tfrac{1}{2} \sum_{\pure} x_{\pure}^{2}$.
\begin{subequations}
Definition \eqref{eq:Bregman} then yields the \emph{Euclidean divergence}
\label{eq:Bregs}
\begin{equation}
\label{eq:Breg-eucl}
\deucl(\base,\notbase )
	= \frac{1}{2}\insum_{\pure} (\notbase_{\pure} - \base_{\pure})^{2},
\end{equation}
which is (uncharacteristically) symmetric in $\base$ and $\notbase$.
Analogously, the Shahshahani metric is generated by the (negative) entropy $h(x) = \insum_{\pure} x_{\pure} \log x_{\pure}$.
A short calculation shows that the corresponding divergence function is the \acdef{KL} divergence
\begin{equation}
\label{eq:KL}
\dkl(\base,\notbase )
	= \insum_{\pure:\,\base_{\pure}>0} \base_{\pure} \log( \base_{\pure}/\notbase _{\pure}),
\end{equation}
which has been used extensively in the analysis of the replicator dynamics \citep{Wei95,HS98}.
\end{subequations}

The key qualitative difference between the Euclidean divergence \eqref{eq:Breg-eucl} and the \ac{KL} divergence \eqref{eq:KL} is that the former is finite for all $\notbase,\base\in\strat$, whereas the latter blows up to $+\infty$ when $\supp(\base) \nsubseteq \supp(\notbase)$.
The reason for this blow-up is that the entropy function $h(x) = \sum_{\pure} x_{\pure} \log x_{\pure}$ becomes infinitely steep as any boundary point $\notbase$ of $\strat$ is approached from the interior of $\strat$, i.e.
\begin{equation}
\label{eq:steep}
\txs
\sup_{\pure\in\pures} \abs{\pd_{\pure} h(\notbase_{n})}
	\to \infty
	\quad
	\text{for every interior sequence $\notbase_{n}$ converging to $\notbase$.}
\end{equation}
When this is the case for all $\notbase\in\bd(\strat)$, we say that $h$ is \emph{steep} \citep{HS02,ABB04}.
At the opposite end of the spectrum, if $Dh(x)$ exists for all $x\in\strat$, we say that $h$ is \emph{nonsteep}.

The link between the steepness of $h$ and the finiteness of the associated Bregman divergence is provided by the following lemma:

\begin{lemma}
\label{lem:Bregman}
Fix $\base\in\strat$ and let $\good(\base)$ denote the union of the relative interiors of the faces of $\strat$ that contain $\base$, i.e.
\begin{equation}
\label{eq:domain}
\good(\base)
	\equiv \setdef{\notbase\in\strat}{\supp(\base) \subseteq \supp(\notbase)}.
\end{equation}
If $h$ is steep, we have $\breg(\base,\notbase)<\infty$ for all $x\in\good(\base)$;
by contrast, if $h$ is nonsteep, we have $\breg(\base,\notbase) < \infty$ for all $x\in\strat$.
\end{lemma}

\begin{proof}
If $h$ is steep and $\notbase \in \good(\base)$, the smoothness of $h$ on the face of $\strat$ spanned by $\supp(x)\supseteq\supp(\eq)$ implies that the directional derivative $h'(\notbase;\base-\notbase)$ exists and is finite, so $\breg(\base,\notbase)$ is itself finite.
If instead $h$ is nonsteep, $h'(\notbase;\base-\notbase)$ exists and is finite for all $\notbase\in\strat$, so again $\breg(\base,\notbase) < \infty$.
\end{proof}

Beyond the positive definiteness property \eqref{eq:BPD}, the attribute of the Bregman divergence that recommends it as a Lyapunov function for \eqref{eq:HD} is that the level sets of $\breg(\base,\argdot)$ are perpendicular to all rays emanating from $\base$ under $g = \hess h.$
Formally, we have:

\begin{lemma}
\label{lem:Bregman-grad}
Let $g = \hess h $ be an extendable \ac{HR} metric and let $\base\in\strat$.
Then, for every smooth curve $\notbase(t)$ with constant support containing that of $\base$, we have:
\begin{equation}
\label{eq:Bregman-grad}
\frac{d}{dt} \breg(\base,\notbase(t))
	= \product{\dot \notbase(t)}{\notbase(t) - \base}_{\notbase(t)}.
\end{equation}
In particular, if $\breg(\base,x(t))$ is constant, $\dot x(t)$ is perpendicular to $x(t) - \base$.
Finally, if $h$ is nonsteep, the above conclusions hold for every smooth curve $\notbase(t)$ on $\strat$.
\end{lemma}

\begin{proof}
The proof is a direct application of the chain rule:
\begin{flalign}
\frac{d}{dt} \breg(\base, x)
	&= - \insum_{\pure}\left[
	\frac{\pd h}{\pd x_{\pure}} \dot \notbase_{\pure}
	+ \frac{\pd h}{\pd x_{\pure}} \frac{d}{dt}(\base_{\pure} - \notbase_{\pure})
	+ \insum_{\purealt} \frac{\pd^{2}h}{\pd x_{\pure} \pd x_{\purealt}} (\base_{\pure} - \notbase_{\pure}) \dot \notbase_{\purealt}
	\right]
	\notag\\
	&= \insum_{\pure,\purealt} (\base_{\pure} - x_{\pure}) g_{\pure\purealt}(x) \dot x_{\purealt}
	= \product{\dot x}{x-\base}_{\notbase},
\end{flalign}
where all summations are taken over the (constant) support $\pures'\equiv\supp(\notbase(t))$ of $\notbase(t)$ and we used the fact that $\dot \notbase_{\pure} = 0$ for $\pure\notin\pures'$.
Finally, in the nonsteep case, $h$ is smooth throughout $\strat$, so the above holds for every smooth curve $x(t)$.
\end{proof}


Within the class of \acl{HR} metrics, steepness of $h$ roughly corresponds to minimal-rank extendability of the metric $g =\hess h$, and nonsteepness to full-rank extendability.
These analogies fail when the steepness of $h$ does not adequately control the regularity of $g$ near the boundary of $\strat$, or when $g$ is minimal-rank extendable but generates non-Lipschitz dynamics.
Bearing this in mind, we use the term \emph{continuous Hessian dynamics} for Riemannian dynamics generated by a minimal-rank extendable metric $g = \hess h $ with steep $h$, and the term \emph{discontinuous Hessian dynamics} for Riemannian dynamics generated by a full-rank extendable metric $g = \hess h$ with nonsteep $h$.
In what follows, we will tacitly assume that the dynamics \eqref{eq:HD} are either continuous or discontinuous.

\subsection{Contractive games and \aclp{ESS}}
\label{sec:contractive}

Recall that a population game $\game\equiv\gamefull$ is called
\emph{contractive} if $\braket{\payv(x')-\payv(x)}{x' - x} \leq 0$ for all $x,x'\in\strat$,
\emph{strictly contractive} if the inequality is strict whenever $x \neq x'$,
and \emph{conservative} if the inequality always binds (cf. \cref{ex:contract}).
As is well known, the set of \aclp{NE} of any contractive game is convex, and every strictly contractive game admits a unique \acl{NE} \citep{HS09}.

Combining the defining inequality of strictly contractive games with the variational characterization of \aclp{NE} \eqref{eq:Nash-variational},
it follows that the (necessarily unique) \acl{NE} of a strictly contractive game satisfies the inequality
\begin{equation}
\label{eq:GESS}
\braket{\payv(x)}{x -\eq }
	\leq 0
	\quad
	\text{with equality only if $x=\eq$}.
\end{equation}
\cite{HS09} call a state satisfying \eqref{eq:GESS} a \acdef{GESS}.
This is the global version of the seminal local solution concept of \cite{MSP73}: if \eqref{eq:GESS} holds for all $x \neq \eq$ in a neighborhood of $\eq$, then $\eq$ is called an \acdef{ESS}.%
\footnote{This concise characterization of evolutionary stability is due to \cite{HSS79}.}

It is well known that the \ac{GESS} $\eq$ of a strictly contractive game
attracts all solutions of the replicator dynamics  whose initial support contains that of $\eq$;
by comparison, $\eq$ attracts \emph{all} orbits of the Euclidean projection dynamics \eqref{eq:PD}.
\cref{thm:contractive} extends these results to all Hessian dynamics \eqref{eq:HD}.

\begin{theorem}
\label{thm:contractive}
Let $\game$ be the \textup(necessarily unique\textup) \acl{NE} of a strictly contractive game $\game$.
Then:
\begin{enumerate}
[\textup(i\textup)]

\item
For all continuous Hessian dynamics, $\breg(\eq,\argdot)$ is a strict decreasing Lyapunov function on $\good(\eq)$, and $\eq$ is asymptotically stable with basin $\good(\eq)$.

\item
For all discontinuous Hessian dynamics, $\breg(\eq,\argdot)$ is a strict decreasing global Lyapunov function, and $\eq$ is globally asymptotically stable.
\end{enumerate}
\end{theorem}

\begin{proof}
We begin with the continuous case.
By \cref{prop:wp}\eqref{itm:wp-min}, every solution $x(t)$ of \eqref{eq:HD} has constant support.
Hence, if $x(0) \in \good(\eq)$, \cref{lem:Bregman-grad} yields:
\begin{subequations}
\label{eq:Breg-calc}
\begin{flalign}
\frac{d}{dt} \breg(\eq,x)
	&
	= \product{\dot x}{x - \eq}_{x}
	= \product{\tproj_{x}(\payv^{\sharp}(x))}{x - \eq}_{x}
	\notag\\
	&\label{eq:Breg-calc2}
	= \product{\payv^{\sharp}(x)}{x - \eq}_{x}
	= \braket{\payv(x)}{x - \eq},
	\\
	&\label{eq:Breg-calc3}
	\leq 0,
\end{flalign}
\end{subequations}
where we used the definition of $\tproj_{x}$ for minimal-rank extendable metrics (cf.~\cref{sec:geometry}) to obtain \eqref{eq:Breg-calc2} and the definition \eqref{eq:GESS} of a \ac{GESS} for \eqref{eq:Breg-calc3}.
Since equality in \eqref{eq:Breg-calc} holds if and only if $x = \eq$, we conclude that $\breg(\eq,x)$ is a strict Lyapunov function on $\good(\eq)$.
If we can show in addition that $x(t)$ has no $\omega$-limit points in $\strat\setminus\good(\eq)$, then asymptotic stability with basin $\good(\eq)$ follows from standard arguments (see e.g.~\cite{San10}, Theorem 7.B.3).

Assume therefore that $x(t)$ admits an $\omega$-limit point $\olimit \neq \eq$, so $x(t_{n})\to \olimit$ for some sequence of times $t_{n}\uparrow\infty$.
Since $\abs{\dot x_{\pure}(t)}$ is bounded from above by $\dynfield_{\max} \equiv \sup_{x\in\strat} \max_{\purealt} \abs{\dynfield_{\purealt}(x)} < \infty$, there exists an open neighborhood $U$ of $\olimit$ and positive $a,\delta,n_{0} > 0$ such that $x(t) \in U$ and $\braket{\payv(x(t))}{x(t) - \eq} \leq -a < 0$ for all $t\in[t_{n},t_{n}+\delta]$ and all $n\geq n_{0}$.
Hence, by \eqref{eq:Breg-calc}, we get
\begin{equation}\label{eq:and}
\breg(\eq,x(t_{n}+\delta)) - \breg(\eq,x(0))
	\leq \int_{0}^{t_{n}+\delta} \braket{\payv(x(s))}{x(s) - \eq} \dd s
	\leq -a(n-n_0)\delta.
\end{equation}
We thus get $\liminf_{t\to\infty} \breg(\eq,x(t)) = -\infty$, a contradiction.

For the discontinuous case, note first that since $\eq - x \in \tcone_{\strat}(x)$, Moreau's decomposition theorem implies that $\product{\tproj_{x}(\payv^{\sharp}(x))}{x - \eq}_{x} \leq \product{\payv^{\sharp}(x)}{x-\eq}_{x}$.
Thus, replacing the first equality in \eqref{eq:Breg-calc2} by an inequality, \eqref{eq:Breg-calc} shows that $\breg(\eq,\argdot)$ is a strict global Lyapunov function for \eqref{eq:HD}.
Global asymptotic stability then follows from \cref{prop:omega}.
\end{proof}

The only implication of $\game$ being strictly contractive used in the previous proof is that its \acl{NE} is a \ac{GESS}.
More generally, if a game admits an \ac{ESS} $\eq$, applying the above arguments in a neighborhood of $\eq$ defined by a level set of $\breg(\eq,\cdot)$ yields the following result:

\begin{theorem}
\label{thm:ESS}
Evolutionarily stable states are asymptotically stable under \eqref{eq:HD}.
\end{theorem}

\cite{HS90}, \cite{Hop99b}, and \cite{Har11} all offer results on the local stability of \emph{interior} evolutionarily stable states under Riemannian game dynamics.
\cite{HS90} showed that under all Riemannian (not necessarily Hessian) dynamics \eqref{eq:RGD}, the function $\Lyap(x) = \product{x-\eq}{x-\eq}_{\eq}$ is a strict local Lyapunov function for interior \acp{ESS} $\eq$, implying that $\eq$ is asymptotically stable.
Likewise, \cite{Hop99b} used linearization to establish local stability of \emph{regular} interior \acp{ESS} \citep{TJ78} under \eqref{eq:RGD}.
Finally, \cite{Har11} employed a version of the argument above to prove asymptotic stability of \acp{ESS} for separable Hessian dynamics of the form \eqref{eq:separable}.

An important case of contractive games that do not admit an \ac{ESS} is the class of conservative games, which include population games generated by matching in symmetric zero-sum games.
Under the replicator dynamics, the \ac{KL} divergence does not provide a strict Lyapunov function for conservative games, but rather a constant of motion.
The following result extends this conclusion to all Hessian dynamics:

\begin{proposition}
\label{prop:conservative}
Let $\eq$ be a \acl{NE} of a conservative game $\game$.
Then, $\breg(\eq,\argdot)$ is a constant of motion along any interior solution segment of \eqref{eq:HD}.
\end{proposition}

\begin{proof}
Simply note that \eqref{eq:Breg-calc} binds if $\game$ is conservative and $x\in\intstrat$.
\end{proof}

\begin{remark}
\label{rem:GenHess}
In the definition of \eqref{eq:HD}, we required that $h$ be finite throughout $\strat$.
This requirement is unnecessary for the preceding results when $\eq$ is interior;
however, if $\eq$ lies on the boundary of $\strat$, the proofs of \cref{thm:contractive,thm:ESS} do not go through because $\breg(\eq, \argdot)$ is no longer well-defined throughout $\good(\eq)$.
Nevertheless, the results themselves remain true if $g = \hess h$ is separable, allowing us to handle the $p$-replicator dynamics for $p \geq 2$ (\cref{ex:pReplicator}).
To prove this, it suffices to replace the implicit summation in $h'(x;\eq-x)$ over all strategies with a sum extending over only the strategies that lie in the support of $\eq$.
\end{remark}

\subsection{Convergence of time-averaged trajectories and permanence}
\label{sec:averages}

We now extend two classic results for the replicator dynamics in random matching games (\cref{ex:matching}) to Hessian dynamics.
The results for these games take advantage of the linearity of payoffs $\payv_{\pure}(x) = \insum_{\purealt\in\pures} A_{\pure\purealt} x_{\purealt}$ in the population state.

The first such result states that if a solution $x(t)$ of the replicator dynamics stays a positive distance away from the boundary of the simplex, then the time-averaged orbit
\(
\bar x(t)
	= t^{-1} \int_{0}^{t} x(s) \dd s
\)
converges to the set of \aclp{NE} of the underlying game \citep{SSHW81}.
The class of games to which this result applies includes zero-sum games (cf.~\cref{prop:conservative}) and games satisfying sufficient conditions for permanence (cf.~\cref{prop:perm} below).
The following proposition shows that this convergence property extends to all Hessian dynamics:

\begin{proposition}
\label{prop:averages}
Let $\game$ be a random matching game and let $x(t)$ be a solution orbit of \eqref{eq:HD}.
If $x(t)$ is contained in a compact subset of $\intstrat$, the time-averaged orbit $\bar x(t) = t^{-1} \int_{0}^{t} x(s) \dd s$ converges to the set of Nash equilibria of $\game$.
\end{proposition}

In the case of the replicator dynamics, this is proved by introducing the auxiliary variables $\legendre_{\pure} = \log x_{\pure}$ and using the fact that $\dot\legendre = \dot x/x$.
To extend this proof to \eqref{eq:HD}, we instead define $\legendre$ via the Bregman divergence of $h$:

\begin{proof}[Proof of \cref{prop:averages}]
Let $\legendre_{\pure} = \breg(\bvec_{\pure},x)$, so $\dot \legendre_{\pure} = \braket{\payv(x)}{x - \bvec_{\pure}} = \braket{\payv(x)}{x} - \payv_{\pure}(x)$ by \cref{lem:Bregman-grad}.
Then, for all $\pure,\purealt\in\pures$, we get
\begin{equation}
\label{eq:Legendre-diff}
\legendre_{\pure}(t) - \legendre_{\purealt}(t)
	= c_{\pure\purealt} + \int_{0}^{t} \left[ \payv_{\purealt}(x(s)) - \payv_{\pure}(x(s)) \right] \dd s,
\end{equation}
where $c_{\pure\purealt} = \legendre_{\pure}(0) - \legendre_{\purealt}(0)$.
Since $x(t)$ is contained in a compact subset of $\intstrat$, \cref{lem:Bregman} implies that $\sup_{t} \legendre_{\pure}(t) < \infty$ for all $\pure\in\pures$.
Thus, dividing both sides of \eqref{eq:Legendre-diff} by $t$ and taking the limit $t\to\infty$, we obtain
\begin{equation}
\label{eq:av3}
\lim_{t\to\infty} \left[ \payv_{\pure}(\widebar x(t)) - \payv_{\purealt}(\widebar x(t)) \right]
	= 0,
\end{equation}
where we have used the linearity of $\payv_\pure(x) = \sum_\purealt A_{\pure\purealt}x_\purealt$ in $x$ to bring the integral into the arguments of $\payv_\pure$ and $\payv_\purealt$.

Equation \eqref{eq:av3} implies that if $\bareq$ is an $\omega$-limit point of $\bar x(t)$, then $\payv_{\pure}(\bareq) = \payv_{\purealt}(\bareq)$ for all $\pure,\purealt\in\pures$, so $\bareq$ is a \acl{NE} of $\game$.
Since $\strat$ is compact, every solution of \eqref{eq:HD} converges to its $\omega$-limit set, and
our assertion follows.
\end{proof}

\cref{prop:averages} applies when the population share of each strategy remains bounded away from zero along all interior solution trajectories, a property known as \emph{permanence}. 
Formally, a dynamical system on $\strat$ is called \emph{permanent} if there exists a threshold $\delta>0$ such that every interior solution satisfies $\liminf_{t\to\infty} x_{\pure}(t) \geq \delta$ for all $\pure\in\pures$.

\cite{HS98} establish a sufficient condition for permanence under the replicator dynamics.
\cref{prop:perm} extends this result to all continuous Hessian dynamics, providing a sufficient condition for \cref{prop:averages} to apply:

\begin{proposition}
\label{prop:perm}
Let $\game$ be a random matching game.
Assume that the dynamics \eqref{eq:HD} are continuous and there exists some $p\in\intstrat$ such that
\begin{equation}
\label{eq:perm-cond}
\braket{\payv(\eq)}{p - \eq}
	> 0
	\quad
	\text{for all boundary rest points $\eq$ of \eqref{eq:HD}}.
\end{equation}
Then, the dynamics \eqref{eq:HD} are permanent.
\end{proposition}

The proof of \cref{prop:perm} follows the proof technique of Theorem 13.6.1 of \cite{HS98}, and is presented in \cref{app:proofs}.

\subsection{Dominated strategies}
\label{sec:dominated}

We conclude by considering the elimination and survival of dominated strategies under \eqref{eq:HD}.
To that end, recall that $\pure\in\pures$ is \emph{strictly dominated} by $\purealt\in\pures$ if $\payv_{\pure}(x) < \payv_{\purealt}(x)$ for all $x\in\strat$.
More generally, $p\in\strat$ is \emph{strictly dominated} by $q\in\strat$ if $\braket{\payv(x)}{p} < \braket{\payv(x)}{q}$ for all $x\in\strat$, meaning that the average payoff of a small influx of mutants is always higher when the mutants are distributed according to $q$ rather than $p$.
We then say that $p\in\strat$ \emph{becomes extinct} along $x(t)$ if $\min\{x_{\pure}(t):\pure\in\supp(\base)\} \to 0$ as $t\to\infty$ \textendash\ or equivalently, if there are no $\omega$-limit points of $x(t)$ in $\good(p)$.

Under the replicator dynamics, it is well known that dominated strategies become extinct along every interior solution trajectory \citep{Aki80}.
As we show below, this elimination result extends to all continuous Hessian dynamics \eqref{eq:HD}:

\begin{proposition}
\label{prop:dominated}
Under all continuous Hessian dynamics \eqref{eq:HD}, strictly dominated strategies become extinct along every interior solution orbit.
\end{proposition}

\begin{proof}
The proof follows a standard argument for the replicator dynamics, replacing the \ac{KL} divergence \eqref{eq:KL} with the Bregman divergence \eqref{eq:Bregman}.
Specifically,
\cref{lem:Bregman-grad} implies that along any interior solution $x(t)$,
\begin{multline}
\label{eq:DSProof}
\frac{d}{dt}\left(\breg(p,x) - \breg(q,x)\right)
	= \product{\dot x}{x - p}_{x} - \product{\dot x}{x - q}_{x}
	\\
	= \product{\tproj_{x}(\payv^{\sharp}(x))}{q - p}_{x}
	= \product{\payv^{\sharp}(x)}{q-p}_x
	= \braket{\payv(x)}{q-p},
\end{multline}
where the penultimate equality uses the fact that $x\in\intstrat$.
Since $q$ strictly dominates $p$ and $\payv$ is continuous, we have $\braket{\payv(x)}{q-p} \geq a$ for some positive constant $a>0$, implying that $\breg(p,x(t))\to\infty$.
Hence, by \cref{lem:Bregman}, we conclude that $x(t)$ has no $\omega$-limit points in $\good(p)$.
\end{proof}

In general, the conclusion of \cref{prop:dominated} is false for discontinuous Hessian dynamics:
\cite{SDL08} construct a four-strategy game with a strictly dominated strategy that is played recurrently by a nonnegligible fraction of the population under the Euclidean projection dynamics \eqref{eq:PD}.
The argument above shows that this strategy must become less common when the state is in the interior of $\strat$;
however, solutions to \eqref{eq:PD} are able to enter and leave the boundary of $\strat$, and while there, dominated strategies may become more common.
We conjecture that this construction can be suitably extended to all discontinuous Hessian dynamics, but we do not tackle this issue here.

\section{Links with reinforcement learning}
\label{sec:RL}

\renewcommand{\vecspace}{\R^{n}}
\renewcommand{\dspace}{(\R^{n})^{\ast}}

Under a variety of reinforcement learning processes for normal form games, mixed strategies evolve according to the replicator dynamics \textendash\ see e.g. \cite{BS97}, \cite{Pos97}, \cite{Rus99}, \cite{Hop02} and \cite{HSV09}.
We conclude the paper by describing a broader connection between reinforcement learning and Hessian game dynamics.

\subsection{Reinforcement learning}

Our starting point is a class of reinforcement learning dynamics for $N$-player normal form games introduced by \cite{CGM15} and \cite{MS16}.
Over the course of play, each player maintains a \emph{score vector} representing the cumulative payoffs of each of his strategies;
then, at each moment in time, the player selects a mixed strategy by applying a \emph{choice map} to this score vector, similar in function to the perturbed best response maps used in stochastic fictitious play and perturbed best response dynamics \citep{FL98,HS02,HS07}.

Formally, let $\payv_{k\pure}(x)$ denote the expected payoff of the $\pure$-th strategy of player $k$ at mixed strategy profile $x = (x_{1},\dotsc, x_{N})$ in an $N$-player normal form game.
The \emph{choice map} $\choice_{k}$ of player $k$ is then defined as
\begin{equation}
\label{eq:choice}
\choice_{k}(y_{k})
	= \argmax_{x_k \in \strat_k} \{\braket{y_{k}}{x_{k}} - h_k(x_{k}) \},
\end{equation}
where $\strat_{k} \equiv \simplex(\pures_{k})$ denotes the mixed strategy space of player $k$ ($\pures_{k}$ being the corresponding strategy set), and $h_{k}\from\strat_k\to\R$ is a smooth, strongly convex \emph{penalty function}.
The reinforcement learning process described above can then be written as
\begin{equation}
\label{eq:RL}
\tag{RL}
\begin{aligned}
\dot y_{k}
	&= \payv_{k}(x)
	\\
x_{k}
	&= \choice_{k}(y_{k}).
\end{aligned}
\end{equation}

Now, let $g_{k} = \hess h_{k}$ and write $\normal_{k\pure}(x) = \sum_{\purealt\in\pures_{k}} g_{k,\pure\purealt}^{-1}(x)$.
\cite{MS16} showed that when the mixed strategy profile $x(t)$ is interior, its evolution under \eqref{eq:RL} is given by
\begin{equation}
\label{eq:RLD}
\tag{RLD}
\dot x_{k\pure}
	= \sum_{\purealt\in\pures_{k}} \left[
		g_{k,\pure\purealt}^{-1}(x) - \frac{\normal_{k\pure}(x) \normal_{k\purealt}(x)}{\sum_{\gamma} \normal_{k\gamma}(x)}
	\right]
	\payv_{k\purealt}(x).
\end{equation}
A comparison with  \eqref{eq:RGD-coords2} shows that the dynamics of mixed strategies under \eqref{eq:RL} agree with the Hessian dynamics \eqref{eq:HD} at interior states.%
\footnote{While we have defined Riemannian game dynamics for single population games and reinforcement learning for $N$-person normal form games, this difference is of no consequence.
One can similarly define Riemannian game dynamics for multipopulation games, or a symmetrized reinforcement learning process for symmetric two-player normal form games (cf.~ \cref{sec:HDandRL}).}

\subsection{A common derivation of \eqref{eq:HD} and \eqref{eq:RLD}}\label{sec:CDHDRLD}

The derivation of \eqref{eq:HD} here and of \eqref{eq:RLD} in \cite{MS16} have very different starting points, leaving the reasons behind their equivalence somewhat mysterious.   
We now make these reasons clearer by deriving both dynamics using a common set of tools. 
Here we present the basic idea behind the argument, using the theory of convex duality to establish the equivalence of versions of \eqref{eq:HD} of \eqref{eq:RLD} defined on the entire positive orthant $\orthant$.%
\footnote{The usefulness of convex conjugates in analyzing maps of the form \eqref{eq:choice} is well known in learning and optimization \textendash\ see e.g.~\cite{NY83}, \cite{HS02}, \cite{SS11}, and \cite{MZ18}.}  Establishing the equivalence of the original processes on $\intstrat$ using this approach requires further ideas, which we present in \cref{sec:HDandRL}.

The starting point for both \eqref{eq:HD} and \eqref{eq:RLD} is the potential function $h$, assumed here to be smooth, strongly convex, and \emph{steep} at the boundary of $\clorthant$ in the sense of \eqref{eq:steep}.
The \emph{convex conjugate} of $h$ is then defined as%
\footnote{For a comprehensive treatment, see \citealp{Roc70} or \citealp{HUL01}.}
\begin{equation}
\label{eq:conjugate}
h^\ast(y)
	= \sup_{x \in\orthant} \{\braket{y}{x} - h(x)\},
	\quad
	y\in(\R^\pures)^\ast.
\end{equation}
Since $h$ is steep and strongly convex, the supremum above is attained at a unique point $\choicef(y) \in \orthant$ \cite[Theorem 26.5]{Roc70}.%
\footnote{Since $h$ is strongly convex, it is bounded below by a strictly convex quadratic function \citep[Theorem B.4.1.1]{HUL01}, which in turn implies that the domain of $h^\ast$ is $\dspace$ \cite[Corollary 13.1]{Roc70}.}
By the first-order optimality conditions for \eqref{eq:conjugate}, we then get
\begin{equation}
\label{eq:choice-app}
\choicef(y)
	\equiv \argmax_{x\in\orthant} \{\braket{y}{x} - h(x)\}
	= (Dh)^{-1}(y),
\end{equation}
where $(Dh)^{-1} $ is the inverse function of $Dh$. 
Applying the envelope theorem to \eqref{eq:conjugate}, we then obtain 
\begin{equation}
\label{eq:Legendre}
Dh^{\ast}(y)
	= \choicef(y).
\end{equation}
It follows from \eqref{eq:choice-app} and \eqref{eq:Legendre} that $Dh(Dh^{\ast}(y)) \equiv y$, so differentiating and rearranging yields
\begin{equation}
\label{eq:Legendre-hess}
\hess h^{\ast} (y)
	=( \hess h (x))^{-1}
\end{equation}
with the inverse of the  Hessian matrix of $h$ being evaluated at $x=\choicef(y)$.

We now introduce and compare full-dimensional analogues of continuous Hessian dynamics and (symmetric) reinforcement learning, taking the orthant $\orthant$ as the state space (and so defining payoffs $\payv$ on $\orthant$).
In the case of \eqref{eq:HD}, the full-dimensional domain makes projections redundant, so the induced dynamics take the form
\begin{equation}
\label{eq:HD+}
\tag{HD$_{\clorthant}$}
\dot x
	=\payv^{\sharp}(x)
	=g^{-1}(x) \, \payv(x)^{\ttop}
	=( \hess h (x))^{-1}\, \payv(x)^{\ttop}.
\end{equation}
For its part, the full-dimensional analogue of \eqref{eq:RL} is
\begin{equation}
\label{eq:RL+}
\tag{RL$_{\clorthant}$}
\begin{aligned}
\dot y
	&= \payv(x),
	\\
x
	&= \choicef(y),
\end{aligned}
\end{equation}
with $\choicef$ defined as in \eqref{eq:choice-app} above.
Differentiating \eqref{eq:RL+} and applying \eqref{eq:Legendre} and \eqref{eq:Legendre-hess} then yields
\begin{equation}
\label{eq:RLD+}
\tag{RLD$_{\clorthant}$}
\dot x
	= D\choicef(y)\, \payv(x)^{\ttop}
	= \hess h^{\ast}(y)  \,\payv(x)^{\ttop}
	=(\hess h(x))^{-1} \, \payv(x)^{\ttop}.
\end{equation}

We can express this argument in words.
By definition, (full-dimensional) Riemannian dynamics are obtained by transforming payoffs at an interior population state $x$ using the matrix $g^{-1}(x)$.
In the Hessian case, the metric $g$ admits a potential function $h$, so we have $g^{-1}(x) = (\hess h(x))^{-1}$ by definition.
As for reinforcement learning, differentiation shows that the dynamics of the mixed strategy $x$ are obtained by transforming payoffs at $x$ by the derivative matrix $D\choicef(y) $ of the choice map \eqref{eq:choice-app}.
Basic facts about convex conjugacy imply that this derivative is equal to $(\hess h (x))^{-1}$, where $h$ is the penalty function that generates $\choicef$.
This argument establishes the equivalence of \eqref{eq:HD+} and \eqref{eq:RLD+}.
For a corresponding argument for the original processes \eqref{eq:HD} and \eqref{eq:RLD}, see \cref{sec:HDandRL}.

\subsection{Boundary behavior in the nonsteep regime}

Since continuous Hessian dynamics coincide with the reinforcement scheme \eqref{eq:RL} when the penalty functions $h_{k}$ are steep, certain results for interior trajectories \textendash\ \cref{prop:averages,prop:dominated} in particular \textendash\ can be obtained directly from the analysis of \cite{MS16}.
On the other hand, in the nonsteep regime, \eqref{eq:RLD} and \eqref{eq:HD} agree at interior states,
but their behaviors at the boundary differ in a fundamental way.
Specifically, while  the boundary behavior of discontinuous Hessian dynamics is defined using closest point projections, the reinforcement learning process \eqref{eq:RL} for nonsteep $h$ can no longer be reduced to mixed strategy dynamics at all.
Instead one must work explicitly with the score variables $y_k$, which continue to aggregate payoff data of all strategies, even those that are not used.
Among other things, this means that a strong cumulative performance of an unused strategy will return it to use.  

This difference between how the processes are defined on the boundary has important consequences.
For instance, while \cite{SDL08} show that strictly dominated strategies may survive under the Euclidean projection dynamics, \cite{MS16} prove that the reinforcement learning process \eqref{eq:RL}  \emph{always} eliminates dominated strategies, whether the penalty functions are steep or not.
Under both processes, dominated strategies are initially eliminated along interior solution trajectories.
But while they may resurface under \eqref{eq:HD}, the score variables of \eqref{eq:RL} continue to register the poor performance of these strategies, ensuring that they remain extinct for all time.

\appendix

\section{Connections with other game dynamics}
\label{app:dynamics}

\newcommand{\zvecspace}{\R_{0}^{n}}
\newcommand{\mxspace}{\R^{n\times n}}
\newcommand{\mxspacea}{\R^{\pures\times\pures}}
\newcommand{\SPD}{S}

Throughout this appendix, we write $\onecol\in\vecspace$ for the $n$-dimensional column vector of ones and $\zproj = I - \frac1n\onecol\onecol^\ttop$ for the Euclidean orthogonal projection of $\vecspace$ onto $\zvecspace = \setdef{z\in\vecspace}{\onecol^{\ttop} z = 0} = \vspan(\onecol)^{\bot}$.%
\footnote{In the above and what follows, $W^{\bot}$ denotes the orthogonal complement of a subspace $W$ of $\vecspace$, defined with respect to the ordinary Euclidean metric.
Even though this might seem to suggest that the Euclidean metric plays a special role in what follows, it is just an artifact of writing everything in coordinates instead of abstractly;
for a detailed discussion, see \cite{Lee97}.}
Recall also that the (\emph{Moore-Penrose}) \emph{pseudoinverse} of a matrix $M \in \mxspace$ is the unique matrix $M^+\in \mxspace$ such that
\begin{inparaenum}%
[\textup(\itshape i\textup)]
\item
$M^{+} y = 0$ whenever $y\in\range(M)^{\bot}$;
and
\item
$M^{+}y = x$ whenever $x\in \ker(M)^{\bot}$, $y\in\range(M)$, and $Mx = y$ \citep[Sec.~6.7]{FIS02}.
\end{inparaenum}
Since a symmetric matrix $M\in\mxspace$ satisfies $\range(M) = \ker(M)^{\bot}$, we have the following well-known algebraic characterization of pseudoinverses:

\begin{lemma}
\label{lem:pseudo}
Let $M\in \mxspace$ be a symmetric matrix.
Then, $M^{+}$ is the unique matrix that
\begin{inparaenum}%
[\textup(\itshape i\textup)]
\item
inverts $M$ on $\ker(M)^{\bot} = \range(M)$;
and
\item
satisfies $\ker(M^{+}) = \ker(M)$.
\end{inparaenum}
\end{lemma}

\subsection{Interior equivalence of Riemannian dynamics and Hopkins' dynamics}
\label{sec:Hopkins}

We now derive the equivalence between \eqref{eq:RGD} and \citeauthor{Hop99b}' dynamics \eqref{eq:Hopkins} on $\intstrat$, as noted in \cref{sec:previous}.
To begin with, we say that $\Hop\in\mxspace$ is a \emph{Hopkins matrix} if it is positive definite with respect to $\zvecspace$ and maps $\onecol$ to $0$.
The following lemma establishes a basic characterization of Hopkins matrices:

\begin{lemma}
\label{lem:HBoth}
\leavevmode
\begin{enumerate}
[\textup(i\textup)]
\item
If $\SPD\in\mxspace$ is symmetric positive-definite, then $(\zproj\SPD\zproj)^{+}$ is a Hopkins matrix and
\begin{equation}
\label{eq:H1}
(\zproj\SPD\zproj)^{+}
	= \SPD^{-1} - \frac{\SPD^{-1}\onecol\onecol^{\ttop}\SPD^{-1}}{\onecol^{\ttop}\SPD^{-1}\onecol}.
\end{equation}
\item 
Conversely, if $\Hop$ is a Hopkins matrix, then $\Hop = (\zproj\SPD\zproj)^+$, where $\SPD = \Hop +\onecol\onecol^{\ttop}$ is symmetric positive definite.
\end{enumerate}
\end{lemma}

\begin{proof}
To prove part (i), let $\SPD$ be symmetric positive-definite.
Then the symmetric matrix $\zproj \SPD \zproj $ is positive definite with respect to $\zvecspace$ and maps $\onecol$ to 0, so $\range(\zproj \SPD \zproj)=\ker(\zproj \SPD \zproj)^{\bot}= \zvecspace$.
If we denote the right-hand side of \eqref{eq:H1} by $\bar\SPD$, a straightforward calculation shows that $\bar\SPD \zproj \SPD \zproj z = z$ for all $z \in \zvecspace$ and $\bar\SPD\onecol = 0$.
Thus \cref{lem:pseudo} implies that $\bar\SPD=(\zproj\SPD\zproj)^+$.
That this is a Hopkins matrix is immediate from the fact that $\range(\zproj \SPD \zproj)=\ker(\zproj \SPD \zproj)^\bot= \zvecspace$ and \cref{lem:pseudo}.

To prove part (ii), let $\Hop$ be a Hopkins matrix and let $\SPD = \Hop +\onecol\onecol^{\ttop}$.
Clearly $\SPD$ is symmetric positive-definite.
Moreover, writing out $(\zproj\SPD\zproj)^+$ using the right-hand side of \eqref{eq:H1} and simplifying the result yields $(\zproj\SPD\zproj)^+=\Hop$.
\end{proof}

\cref{prop:HopDRD} below establishes the equivalence between \eqref{eq:Hopkins} and \eqref{eq:RGD} on $\intstrat$, and provides a concise third representation for both dynamics.  Observe that \eqref{eq:RGDNewMat} is expression \eqref{eq:RGD-coords2} for \eqref{eq:RGD}  on $\intstrat$, but written in matrix form.

\begin{proposition}
\label{prop:HopDRD}
Let $\dot x = \dynfield(x)$ be a dynamical system on $\intstrat$.
Then, the following are equivalent:
\begin{subequations}
\begin{enumerate}
[\textup(i\textup)]
\item
There is a smooth field of Hopkins matrices $\HopFcn\from\intstrat\to\mxspacea$ such that
\begin{alignat}{2}
\label{eq:HopkinsMat}
\dynfield(x)
	&= \HopFcn(x)\,\payv(x)^{\ttop}
	&\quad
	&\text{for all $x\in\intstrat$}.
\intertext{\item There is a smooth field of symmetric positive-definite matrices $\hessmat\from\intstrat\to\mxspacea$ such that}
\label{eq:RGDMP}
\dynfield(x)
	&= (\zproj \hessmat(x) \zproj)^{+} \payv(x)^{\ttop}
	&\quad
	&\text{for all $x\in\intstrat$}.
\intertext{\item There is a smooth Riemannian metric $g$ on $\orthant$ such that}
\label{eq:RGDNewMat}
\dynfield(x)
	&= \left(g^{-1}(x) - \frac{g^{-1}(x)\onecol\onecol^{\ttop} g^{-1}(x)}{\onecol^{\ttop} g^{-1}(x) \onecol}\right)\payv(x)^\ttop
	&\quad
	&\text{for all $x\in\intstrat$}.
\end{alignat}
\end{enumerate}
\end{subequations}
\end{proposition}

\begin{proof}
The equivalence of (i) and (ii) follows from \cref{lem:HBoth}, and
the equivalence of   (ii) and (iii) follows from \cref{lem:HBoth}(i) with $\SPD=\hessmat(x)=g(x)$.%
\footnote{To formally complete the argument that (ii) implies (iii), we observe without proof that the field $\SPD$ on $\intstrat$ can be smoothly extended to $\orthant$.}
\end{proof}

\subsection{Continuous Hessian dynamics and reinforcement learning}
\label{sec:HDandRL}

We now complete the common derivation of \eqref{eq:HD} and \eqref{eq:RLD} on $\intstrat$ initiated in \cref{sec:CDHDRLD}.
\cref{prop:HopDRD} above shows that the continuous Hessian dynamics \eqref{eq:HD} can be expressed as
\begin{equation}
\label{eq:NewDyn}
\dot x
	= (\zproj \hessmat(x) \zproj)^{+} \payv(x)^{\ttop}.
\end{equation}
We now extend the argument from \cref{sec:CDHDRLD} to show that the reinforcement learning dynamics \eqref{eq:RLD} also take this form.

To that end, let $\Undx = \{x - \frac1{n}\onecol \from x \in \intstrat\}\subset \zspace^{n}$, and define $\hh \from \Undx\to\R$ as
\begin{equation}
\label{eq:hh}
\hh(\undx)
	= h(\undx+ \tfrac{1}{n}\onecol).
\end{equation}
Since $\intUndx$ is open relative to $\zspace^{n}$, the derivative $D\hh$ of $\hh$ at $\undx \in \intUndx$ can be represented by a covector $D\hh(\undx)$ in  $(\zspace^{n})^\ast$.  Imposing the Euclidean metric on $\zspace^{n}$ for convenience, we can identify $(\zspace^{n})^\ast$ with the set of covectors whose components sum to zero, and write the derivative and Hessian of $\hh$ at $\undx \in \intUndx$ as
\begin{flalign}
\label{eq:dhh}
D\hh(\undx)
	&=  Dh(\undx+ \tfrac1{n}\onecol)\zproj,
	\\
\label{eq:Hesshh}
\hess \hh(\undx)
	&= \zproj H(\undx+ \tfrac1{n}\onecol)\zproj.
\end{flalign}
The convex conjugate $\hh^\ast $ of $\hh$ is then defined as
\begin{equation}
\label{eq:hhast}
\hh^\ast(y_0) = \max_{x \in\Undx} \{\braket{y}{\undx} - \hh(\undx)\},
	\quad
	y_0\in (\zspace^{n})^\ast ,
\end{equation}
where the fact that the domain of $\hh^\ast$ is all of $(\zspace^{n})^\ast$ follows from the compactness of $\Undx$.
By the basics of convex conjugation \citep[Chapter 26]{Roc70}, the maps $D\hh^\ast \from(\zspace^{n})^\ast\to\Undx $ and $D\hh \from\Undx \to (\zspace^{n})^\ast $ are inverse to one another, so we have
\begin{equation}
\label{eq:bothHess}
\hess \hh^\ast(y_0)
	=  (\hess \hh(d\hh^\ast(y_0)) )^{-1},
\end{equation}
with both sides of the equality in \eqref{eq:bothHess} understood as linear maps from $\zspace^{n}$ to itself.

Now recall that the (symmetric) reinforcement learning process is defined by
\begin{equation}\tag{\ref{eq:RL}}
\begin{aligned}
\dot y
	&= \payv(x)
	\\
x
	&= \choice(y).
\end{aligned}
\end{equation}
The choice map $\choice$, defined  in \eqref{eq:choice} in terms of $h$, satisfies $\choice(y) = \choice(y\zproj)$ for all $y \in \R^n$.
Thus, applying definition \eqref{eq:hh}, we can express  $\choice$ in terms of $\hh$ as
\begin{equation}
\label{eq:Choicehh}
\choice(y)
	= \argmax_{x \in \strat} \{y\zproj x - h(x) \} 
	= \argmax_{\undx \in \Undx} \{y\zproj \undx - \hh(\undx) \} + \tfrac{1}{n}\onecol
	= D\hh^{\ast}(y\zproj) + \tfrac{1}{n}\onecol.
\end{equation}
Hence, substituting \eqref{eq:Choicehh} into \eqref{eq:RL}, differentiating, and using \cref{eq:bothHess,eq:Choicehh,eq:RL,eq:Hesshh} yields
\begin{flalign}
\dot x
	&= \hess \hh^\ast(\payv(x)\zproj) \,\zproj \payv(x)^\ttop
	= (\hess \hh (D\hh^\ast(\payv(x)\zproj)))^{-1} \,\zproj \payv(x)^\ttop
	\notag\\
	&= (\hess \hh (x - \tfrac1{n}\onecol))^{-1}\,\zproj \payv(x)^\ttop
	= (\zproj H(x)\zproj)^+\payv(x)^\ttop,
\end{flalign}
as specified in \eqref{eq:NewDyn}.

\section{Extensions of Riemannian metrics}
\label{app:geometry}

Here we present some technical results concerning the extension of Riemannian metrics from $\orthant$ to $\clorthant$.  \cref{prop:extension} shows that an extendable metric $g$ on $\clorthant$ induces a well-defined scalar product at all points of $\clorthant$:

\begin{proposition}
\label{prop:extension}
Let $g$ be an extendable Riemannian metric on $\clorthant$.
Then, for all $x\in\clorthant$, there exists a unique scalar product $\product{\argdot}{\argdot}_{x}$ on $\domg(x) $ such that $\product{w}{w'}_{x_{k}} \to \product{w}{w'}_{x}$ for all $w,w'\in\domg(x)$ and for every interior sequence $x_{k}\to x$.
Moreover, if $g$ is minimal-rank extendable, we have $g_{\pure\purealt}^{\sharp}(x) = 0$ whenever $\pure,\purealt\notin\supp(x)$.
\end{proposition}

\begin{proof}
Fix some $x\in\clorthant$ and write $g^{\sharp}(x) = Q^{\top}\Lambda Q$ where the diagonal matrix $\Lambda$ consists of the eigenvalues of $g^{\sharp}(x)$ and $Q$ is an orthogonal matrix ($Q^{\top} = Q^{-1}$) whose columns are the eigenvectors of $g^{\sharp}(x)$.
Since $g^{\sharp}(x)$ is positive-semidefinite, its eigenvalues are nonnegative.
Furthermore, since $\domg(x) = \im g^{\sharp}(x)$, every eigenvector of a nonzero eigenvalue of $g^{\sharp}(x)$ must lie in $\domg(x)$:
indeed, if $g^{\sharp}(x) z = \lambda z$ for some $\lambda>0$, we will also have $z = g^{\sharp}(x)z\lambda^{-1}$, i.e. $z\in\im g^{\sharp}(x) = \domg (x)$.
As a result, we may write $g^{\sharp}(x) = \sum_{\lambda>0} \lambda\, \eigvec_{\lambda} \eigvec_{\lambda}^{\top}$ where the summation is taken over all positive eigenvalues $\lambda>0$ of $g^{\sharp}(x)$ (assumed for convenience to be distinct) and $\eigvec_{\lambda}$ is the corresponding column of $Q$.
The metric tensor of the induced scalar product $\product{\argdot}{\argdot}_{x}$ at $x$ is then defined as the pseudoinverse $( g^{\sharp}(x) )^{+} $ of $g(x)$ (see \cref{app:dynamics}), given here by
\begin{equation}
\label{eq:product-eig}
( g^{\sharp}(x) )^{+} 
= \insum_{\lambda>0} \lambda^{-1} \eigvec_{\lambda} \eigvec_{\lambda}^{\top}.
\end{equation}
Our continuity and uniqueness claims are then immediate.

Finally, to show that $g_{\pure\purealt}^{\sharp}(x) = 0$ if $\pure\notin\supp(x)$ and $g$ is minimal-rank extendable, simply note that $g_{\pure\purealt}^{\sharp}(x) = \bvec_{\pure}^{\top} g^{\sharp}(x) \bvec_{\purealt} = \insum_{\lambda} \lambda\, \bvec_{\pure}^{\top} \eigvec_{\lambda} \eigvec_{\lambda}^{\top} \bvec_{\purealt} = 0$ because all eigenvectors of $g^{\sharp}(x)$ with positive eigenvalues lie in $\R^{\supp(x)} = \domg(x)$.
\end{proof}

\begin{remark}
\label{rem:gsharpBdMin}
In the minimal-rank case, the scalar product $\product{\argdot}{\argdot}_{x}$ on $\domg(x)$ can be represented by the matrix $g(x) = (g^{\sharp}(x))^+$.
This means that the submatrix $(g_{\pure\purealt}(x))_{\pure,\purealt \in \supp(x)}$ is the inverse of the submatrix $(g_{\pure\purealt}^{\sharp}(x))_{\pure,\purealt \in \supp(x)}$ while the remaining components (which have no geometric significance) are set to $0$.
\end{remark}

Next we establish the correctness of the formulas \eqref{eq:RGD-coords} for dynamics generated by a minimal-rank extendable metrics.

\begin{proposition}\label{prop:RGDMinFormula}
The formulas \eqref{eq:RGD-coords}
describe \eqref{eq:RGD} generated by a minimal-rank extendable metric $g$ at all states $x \in\strat$;
moreover, the sums in \eqref{eq:RGD-coords} need only be taken over $\purealt\in\supp(x)$.
\end{proposition}

\begin{proof}
The case $x \in \intstrat$ is covered in the text.
Otherwise, if $x \in \bd(\strat)$,
the corresponding $g$-admissible set is
\begin{equation}
\label{eq:ProjDomBdMin}
\metcone(x)
	= \tcone_{\strat}(x) \cap \tspace_{\clorthant}(x)
	 = \zspace^{\pures} \cap \R^{\supp(x)}= \tspace_{\strat}(x).
\end{equation}
By definition, $\payv^\sharp(x)$ and $\normal(x)$ lie in $\im g^{\sharp}(x) =\R^{\supp(x)}$.
Moreover, $\normal(x)$ is normal to $\tspace_{\strat}(x)=\zspace^{\pures} \cap \R^{\supp(x)}$ because
\begin{equation}
\label{eq:CheckNormal2}
\product{\normal(x)}{z}_{x}
	= \onerow\hspace{1pt} g^{\sharp}(x)\hspace{1pt} g(x) z
	= \braket{\onerow}{z}
	= 0
	\quad
	\text{for all $z\in\zspace^{\pures}$},
\end{equation}
where the second equality follows from \cref{rem:gsharpBdMin}.

Eq.~\eqref{eq:GeoRGDComp} shows that the \acl{RHS} of \eqref{eq:RGD} is equal to $\tproj_{x}(\payv^{\sharp}(x))$.
In light of the facts above, \eqref{eq:tproj2} shows that $\tproj_{x}(\payv^{\sharp}(x))$ is equal to the \acl{RHS} of \eqref{eq:RGD-coords1}, the novelty being that $\payv^\sharp_\pure(x)$ and $\normal_{\pure}(x)$ vanish whenever $\pure\notin\supp(x)$.
These claims prove the first statement in the proposition.
To establish the second statement, simply observe that $g_{\pure\purealt}^{\sharp}(x) = 0$ whenever $\pure$ or $\purealt$ is not in $\supp(x)$ by \cref{prop:extension}, so including $\purealt\notin\supp(x)$ in the sums in \eqref{eq:RGD-coords} is irrelevant.
\end{proof}

\section{Convergence and stability in dynamical systems}
\label{app:stability}

\newcommand{\rest}{\mathrm{RP}}
\newcommand{\ball}{B}

\subsection{Definitions}

Throughout this appendix, we focus on the dynamics
\begin{equation}
\tag{\ref*{eq:ED}}
\dot x
	= \dynfield(x),
	\quad
	x\in\strat,
\end{equation}
and we assume that they admit unique solutions from every initial condition.
With this in mind, we say that $\eq$ is an $\omega$-\emph{limit point} of the solution orbit $x(t)$ if there is an increasing sequence of times $t_{n}\uparrow\infty$ such that $x(t_{n}) \to \eq$.
We further say that $\eq$ is \emph{Lyapunov stable} if, for every neighborhood $U$ of $\eq$, there exists a neighborhood $U'$ of $\eq$ such that every solution orbit $x(t)$ that starts in $U'$ is contained in $U$ for all $t\geq0$.
Finally, we say that
$\eq$ is \emph{attracting} if there is a neighborhood $U$ of $\eq$ such that every solution that starts
in $U$ converges to $\eq$,
and
$\eq$ is called \emph{asymptotically stable} if it is Lyapunov stable and attracting.
In this case, the maximal (relatively) open set of states from which solutions converge to $\eq$ is called the \emph{basin} of $\eq$;
if the basin of $\eq$ is all of $\strat$, we say that $\eq$ is \emph{globally asymptotically stable}.

\subsection{A global convergence result}

A standard result from dynamical systems states that if a smooth dynamical system on a compact set admits a strict global Lyapunov function, all $\omega$-limit points are rest points \citep[see e.g.][Theorem 7.B.3]{San10}.  
The proof of this result relies on the continuity of solutions on initial conditions, a property which is not easily established for discontinuous dynamics.
In \cref{prop:omega} below, we present a global convergence result that does not require continuity of solutions in initial conditions, but instead relies on a \ac{lsc} lower bound on the derivative of the Lyapunov function.
To state it, let
\begin{equation}
\rest	
	= \setdef{x\in\strat}{\dynfield(x) = 0}
\end{equation}
denote the set of rest points of the dynamics \eqref{eq:ED}.
We then have:

\begin{proposition}
\label{prop:omega}
Let $x(t)$ be an absolutely continuous solution orbit of \eqref{eq:ED} and let $\Gamma_{+} = x(\R_{+})$ denote the set of points visited by $x(t)$.
Assume further that $\rest$ is closed and there exist functions $\Lyap\from\strat\to\R$ and $\rate\from\strat\to\R_{+}$ such that
\begin{enumerate}
[\textup(i\textup)]
\item
\label{itm:omega-smooth}
$\Lyap$ is differentiable in a neighborhood of $\Gamma_{+}$.
\item
\label{itm:omega-rest}
$\rate$ is \acl{lsc} and $\rate(x) = 0$ if and only if $x \in \rest$.
\item
\label{itm:omega-Lyap}
$\braket{D\Lyap(x)}{\dynfield(x)} \geq \rate(x)$ for all $x\in \Gamma_{+}$.

\end{enumerate}
Then, $x(t)$ converges to $\rest$.
\end{proposition}

\begin{proof}
By absolute continuity and Conditions \eqref{itm:omega-smooth} and \eqref{itm:omega-Lyap} above, we get
\begin{equation}
\label{eq:Lyapunov-int}
\Lyap(x(t)) - \Lyap(x(0))
	= \int_{0}^{t} \braket{D\Lyap(x(s))}{\dynfield(x(s))} \dd s
	\geq \int_{0}^{t} \rate(x(s)) \dd s
	\geq 0,
\end{equation}
i.e. $\Lyap$ is nondecreasing along $x(t)$.
Furthermore, since $\strat$ is compact, $x(t)$ admits at least one $\omega$-limit point $\eq\in\cl(\Gamma_{+}) \subseteq \strat$.
Assume now that $x(t)$ admits an $\omega$-limit point $\olimit$ such that $\dynfield(\olimit) \neq 0$.
Since $\rest$ is closed, Condition \eqref{itm:omega-rest} implies that there is a compact neighborhood $K$ of $\olimit$ and some $a>0$ such that $\rate(x) \geq a > 0$ for all $x\in K$.
With this in mind, we consider two complementary cases below:

\smallskip
\paragraph{\emph{Case 1}}
Suppose there exists some $T\geq 0$ such that $x(t) \in K$ for all $t \geq T$.
Then $\rate(x(t)) \geq a$ for all $t\geq T$, so \eqref{eq:Lyapunov-int} yields $\lim_{t\to\infty} \Lyap(x(t)) = \infty$, a contradiction.

\smallskip
\paragraph{\emph{Case 2}}
Assume instead that, for all $T \geq0$, we have $x(t) \notin K$ for some $t\geq T$.
In this case, there exist open neighborhoods $U$ and $U'$ of $\olimit$ with $\cl(U) \subseteq U' \subset K$,
and interlaced sequences $t_{n},t_{n}'\uparrow\infty$ such that, for all $n$:
\begin{inparaenum}%
[\textup(\itshape i\textup)]
\item
$t_{n} < t_{n}' < t_{n+1}$;
\item
$x(t_{n}) \in U$, $x(t_{n}') \in K\setminus U'$;
and
\item
$x(t)\in K$ whenever $t\in[t_{n},t_{n}']$.
\end{inparaenum}
Then, since $\abs{\dot x_{\pure}(t)}$ is bounded from above by $\dynfield_{\max} \equiv \sup_{x\in\strat} \max_{\purealt} \abs{\dynfield_{\purealt}(x)} < \infty$, the time intervals $\delta_{n} \equiv t_{n}' - t_{n}$ will be bounded from below by $\delta_{\min} \equiv \dist(\cl(U),K\setminus U') / \dynfield_{\max} > 0$.
We thus get
\begin{flalign}
\Lyap(x(t_{n}')) - \Lyap(x(0))
	\geq \int_{0}^{t_{n}'} \rate(x(s)) \dd s
	\geq \sum_{j=1}^{n} \int_{t_{j}}^{t_{j}'} \rate(x(s)) \dd s
	\geq  an\delta_{\min},
\end{flalign}
i.e. $\Lyap(x(t_{n}')) \to \infty$, a contradiction.
\end{proof}

To apply \cref{prop:omega} to discontinuous Riemannian dynamics in potential games, we need the following result:

\begin{lemma}
\label{lem:RLSC}
The speed of motion $\norm{\dynfield(x)}_{x}$ of the dynamics \eqref{eq:RGD} is \ac{lsc} on $\strat$.
\end{lemma}

\begin{proof}
If the dynamics \eqref{eq:RGD} are continuous, our claim follows immediately from the continuity of the underlying metric \textendash\ in fact, $\norm{\dynfield(x)}_{x}$ is continuous in this case.
Otherwise, if \eqref{eq:RGD} is discontinuous, recall that $\dynfield(x) \equiv \tproj_{x}(\payv^{\sharp}(x))$ is simply the projection of $\payv^{\sharp}(x)$ on the tangent cone $\tcone_{\strat}(x)$ to $\strat$ at $x$ (because $\metcone(x) = \tcone_{\strat}(x)$ in that case).
Therefore, if we write $\dynfield^{\perp}(x) = \payv^{\sharp}(x) - \dynfield(x)$ for the projection of $\payv^{\sharp}(x)$ on the normal cone $\ncone_{\strat}(x)$ to $\strat$ at $x$, Moreau's decomposition theorem and the Cauchy-Schwarz inequality yield
\begin{equation}
\product{\payv^{\sharp}(x)}{z}_{x}
	= \product{\dynfield(x) + \dynfield^{\perp}(x)}{z}_{x}
	= \product{\dynfield(x)}{z}_{x}
	\leq \norm{\dynfield(x)}_{x} \norm{z}_{x},
\end{equation}
for all $z\in\tcone_{\strat}(x)$,
with the inequality binding if and only if $z\propto\dynfield(x)$.
We thus obtain the characterization
\begin{equation}
\label{eq:SpeedMin}
\norm{\dynfield(x)}_{x}
	= \max_{z \in \tcone_\strat(x)\cap\ball(x)} \product{\payv^{\sharp}(x)}{z}_{x},
\end{equation}
where $\ball(x) = \setdef{z\in\R^{\pures}}{\norm{z}_{x}\leq1}$.

Note now that the correspondence $x \mapsto \tcone_\strat(x)$ is \ac{lsc} because it is constant on the interior of each face of $\strat$ and  $\tcone_{\strat}(x) \subseteq \tcone_{\strat}(y)$ whenever $\supp(x) \subseteq \supp(y)$.
This shows that the constraint correspondence $x \mapsto \tcone_{\strat}(x) \cap \ball(x)$ of \eqref{eq:SpeedMin} is \ac{lsc};
since the objective function $\product{\payv^{\sharp}(x)}{z}_{x}$ of \eqref{eq:SpeedMin} is jointly continuous in $x$ and $z$, a precursor to the maximum theorem \citep[Lemma 16.30]{AB99} implies that $x \mapsto \norm{\dynfield(x)}_{x}$ is itself \ac{lsc}, as claimed.
\end{proof}

\section{Additional proofs}
\label{app:proofs}

In this appendix, we collect some proofs that are too technical for the main text.
We begin with the proof of \cref{prop:wp} regarding the existence and uniqueness of solutions to \eqref{eq:RGD}.
As noted in \cref{sec:analysis}, we only need to prove part \eqref{itm:wp-full}, which concerns the case of full-rank extendable metrics.
Existence of forward solutions of \eqref{eq:RGD} on $\strat$ follows from general results of \cite{AC84} on solutions to discontinuous differential equations;
see \cite{LS08} for a summary of their argument.
Thus, it remains to show that forward solutions to \eqref{eq:RGD} from each initial condition in $\strat$ are unique.
This conclusion follows from the following lemma:

\begin{lemma}
\label{lem:Gronwall}
Let $x(t)$ and $x'(t)$ be solutions to \eqref{eq:RGD}, and let 
\begin{equation}
\label{eq:DefP}
P(t)
	= \norm{x'(t)-x(t)}_{x(t)}^{2} \, e^{-\lambda t}.
\end{equation}
If $\lambda>0$ is large enough, then $P(t)$ is nonincreasing. 
\end{lemma}

Given this lemma, we immediately obtain:
\begin{proof}[Proof of \cref{prop:wp}]
Let $x(t)$ and $x'(t)$ be solutions of \eqref{eq:RGD} with $x(0)= x'(0)$.
We then get $P(t) = P(0) = 0$ for all $t\geq 0$, so $x(t)= x'(t)$ for all $t \geq 0$.
\end{proof}

To prove \cref{lem:Gronwall}, we need one final auxiliary result.
Let $\dynfield(x) = \tproj_{x}(\payv^{\sharp}(x))$ denote the \acl{RHS} of \eqref{eq:RGD}.
Then, as we show below, $\dynfield$ satisfies a one-sided Lipschitz condition with respect to the underlying metric:

\begin{lemma}
\label{lem:1SL}
There exists some $K_{\dynfield} >0$ such that 
\begin{equation}
\label{eq:1SL}
\product{\dynfield(x')-\dynfield(x)}{x'-x}_{x}
	\leq K_{\dynfield} \norm{x'-x}_{x}^{2}
	\quad
	\text{for all $x, x' \in \strat$}.
\end{equation}
\end{lemma}

\begin{proof}[Proof of \cref{lem:1SL}]
Write $w(x)= \payv^\sharp(x)$ and $w^\perp(x) = w(x) - \tproj_{x}(w(x))$.
Since $g$ is full-rank extendable, $\tproj_{x}$ is the orthogonal projection onto $\tcone_{\strat}(x)$ with respect to $\product{\argdot}{\argdot}_{x}$.
We thus obtain
\begin{flalign}
\label{eq:1SLpf1}
\langle\dynfield(x') - \dynfield(x),x'-x\rangle_{x}
	&= \product{w(x') - w(x)}{x' - x}_{x} - \product{w^\perp(x') - w^\perp(x)}{x' - x}_{x}
	\notag\\
	&= \product{w(x') - w(x)}{x' - x}_{x}
	+ \product{w^{\perp}(x)}{x' - x}_{x}
	\notag\\
	&+ \product{w^{\perp}(x')}{x - x'}_{x'}
	+ (w^{\perp}(x'))^\ttop(g(x') - g(x))(x' - x)\notag\\
	&\leq K_{w} \norm{x' - x}^2_{x} + (w^{\perp}(x'))^\ttop(g(x') - g(x))(x' - x),
\end{flalign}
where the bound for the first term in the last line follows from the Cauchy-Schwarz inequality and the Lipschitz continuity of $w$, while the rest follows from Moreau's decomposition theorem.
To bound the last term, write $g_{\pure}(x)$ for the $\pure$-th row of $g(x)$, let $W_{\max}^{\perp} = \max_{\pure\in\pures}\max_{x\in\strat} w_{\pure}^{\perp}(x)$, and let $\norm{\cdot}_{2}$ denote the standard Euclidean norm.
Then, if $C>0$ is chosen sufficiently large, we get
\begin{flalign}
\label{eq:1SLpf2}
(w^{\perp}(x'))^{\ttop}
	&(g(x') - g(x)) (x' - x)
	\notag\\
	&\leq \insum_{\pure\in\pures} w_{\pure}^{\perp}(x') \norm{g_\pure(x') - g_\pure(x)}_{2} \norm{x'-x}_{2}
	\notag\\
	&\leq W_{\max}^{\perp} \norm{x' - x}_{2} \insum_{\pure\in\pures} \norm{g_\pure(x') - g_\pure(x)}_{2}
	\leq W_{\max}^{\perp} C \norm{x' - x}_{x}^{2}.
\end{flalign}
In the above, the first inequality is an immediate corollary of the Cauchy-Schwarz inequality;
the last one follows from the equivalence of norms on $\R^{\pures}$ and the fact that $g$ is $C^{1}$ on $\strat$;
finally, $C$ can be chosen independently of $x$ and $x'$ because $\strat$ is compact.
Combining \eqref{eq:1SLpf1} and \eqref{eq:1SLpf2} completes our proof.
\end{proof}

With \cref{lem:1SL} at hand, we finally obtain:

\begin{proof}[Proof of \cref{lem:Gronwall}]
Define $\dot g(x) = (\dot g_{\pure\purealt}(x))_{\pure,\purealt\in\pures}$ by $\dot g_{\pure\purealt}(x) = \braket{Dg_{\pure\purealt}(x)}{\dynfield(x)} = \insum_{\kappa\in\pures} \dynfield_{\kappa}(x) \, \pd_{\kappa} g_{\pure\purealt}(x)$, and let $K_{g} = \max_{\pure,\purealt\in\pures}\max_{x\in\strat} \dot g_{\pure\purealt}(x) < \infty$ (recall that $g$ is $C^{1}$).
Then, for all $t \geq 0$ such that $x(t)$ and $x'(t)$ are differentiable, we have
\begin{flalign}
\dot P
	&= 2 \product{\dynfield(x') - \dynfield(x)}{x' - x}_{x} \,e^{-\lambda t}
	+(x' - x)^\ttop \dot g(x) (x' - x) \, e^{-\lambda t}
	- \lambda \norm{x' - x}_{x}^{2} \, e^{-\lambda t}
	\notag\\
	&\leq -(\lambda - 2 K_{\dynfield} - K_{g}) \,\norm{x'-x}^2_{x}\,\, e^{-\lambda t},
\end{flalign}
where we used \cref{lem:1SL} to bound the second term in the first line.
Taking $\lambda > 2K_{\dynfield} + K_{g}$ then yields $\dot P \leq 0$;
since $x(t)$ and $x'(t)$ are absolutely continuous, we conclude that $P(t)$ is nondecreasing.
\end{proof}

We close this appendix with the proof of our permanence criterion:

\begin{proof}[Proof of \cref{prop:perm}]
Define $P\from\clorthant\to\R$ as $P(x) = -\exp\left(\sum_{\pure} p_{\pure} \breg(\bvec_{\pure}, x)\right)$ for $x\in\orthant$ and $P(x) = 0$ for $x\in\bd(\clorthant)$.
The steepness of $h$ implies that $P$ is continuous, while \cref{lem:Bregman-grad} implies that $\frac{d}{dt}\log(P(x)) = \Psi(x) \equiv \braket{\payv(x)}{p - x}$ for all $x \in \intstrat$.
Hence, by Theorem 12.2.1 of \cite{HS98}, it suffices to show that the function $\Psi$ is an average Lyapunov function for \eqref{eq:HD}, meaning that, for every initial condition $x(0)\in\bd(\strat)$, there is a $t>0$ such that
\begin{equation}
\label{eq:Lyap-avg}
\frac{1}{t}\int_{0}^{t} \Psi(x(s)) \dd s
	> 0.
\end{equation}

We proceed by induction on the cardinality of the support of the initial condition.
The claim is trivial if this cardinality is $1$.
For the inductive step, suppose that \eqref{eq:Lyap-avg} holds when the cardinality is $k \in \{1, \ldots , \abs{\pures}-2\}$, and consider an initial condition $x(0)$ whose support $\pures'$ has cardinality $k +1$.
If $x(t)$ converges to the boundary of the face $\strat'$ of $\strat$ spanned by $\pures'$, then our claim follows from the inductive hypothesis and the same arguments as in the proof of Theorem 12.2.2 in \cite{HS98}.
If instead $x(t)$ does not converge to the boundary of $\strat'$, then there exists a $\delta>0$ and an increasing sequence of times $t_{n}\uparrow\infty$ with $x_{\pure}(t_{n}) \geq \delta > 0$ for all $\pure\in\pures'$.
Then, letting $\bar x_{\pure}(t) = t^{-1} \int_{0}^{t} x_{\pure}(s) \dd s$ and $\bar u(t) = t^{-1} \int_{0}^{t} \braket{\payv(x(s))}{x(s)} \dd s$, we may assume (by descending to a subsequence of $t_{n}$ if necessary) that $\bar x(t_{n})$ and $\bar u(t_{n})$ converge to some $\bareq$ and $\bar u^{\ast}$ respectively as $n\to\infty$.

We now claim that $\payv_{\pure}(\bareq) = \bar u^{\ast}$ for all $\pure\in\pures'$, implying that $\bareq$ is a restricted equilibrium of $\game$.
Indeed, let $\legendre_{\pure}(t) = \breg(\bvec_{\pure},x(t))$ for all $\pure\in\pures'$.
Then, $\dot \legendre_{\pure} = \braket{\payv(x)}{x} - \payv_{\pure}(x)$ by \cref{lem:Bregman-grad}, so the linearity of $\payv(x)$ in $x$ implies that
\begin{equation}
\label{eq:PermAvgs}
\frac{1}{t} \int_{0}^{t} \braket{\payv(x(s))}{x(s)} \dd s - \payv_{\pure}(\bar x(t)) 
	= \frac{\legendre_{\pure}(t) - \legendre_{\pure}(0)}{t}.
\end{equation}
Given that $x(t_{n})$ remains a minimal positive distance away from $\bd(\strat')$, it follows that $\legendre_{\pure}(t_n)$ is bounded from above for all $\pure\in\pures'$.
Therefore, the \acl{RHS} of \eqref{eq:PermAvgs} vanishes as $t_{n}\to\infty$, implying in turn that $\payv_{\pure}(\bareq) = \bar u^{\ast}$ for all $\pure\in\pures'$, as claimed.

Now, since $\bareq$ is a restricted equilibrium of $\game$, \cref{prop:stationary} implies that it is a boundary rest point of \eqref{eq:HD}, so $\bar u^{\ast} = \braket{\payv(\bareq)}{\bareq} < \braket{\payv(\bareq)}{p}$ by \eqref{eq:perm-cond}.
Moreover, since $\Psi(x) = \braket{\payv(x)}{p - x}$, setting $t=t_{n}$ in \eqref{eq:Lyap-avg} yields
\begin{equation}
t_{n}^{-1}\int_{0}^{t_{n}} \Psi(x(s)) \dd s
	= t_{n}^{-1} \int_{0}^{t_{n}} \braket{\payv(x(s))}{p - x(s)} \dd s
	= \braket{\payv(\bar x(t_{n}))}{p} - \bar u(t_{n})
\end{equation}
so $\lim_{n\to\infty} t_{n}^{-1} \int_{0}^{t_{n}} \Psi(x(s)) \dd s = \braket{\payv(\bareq)}{p} - \bar u^{\ast} > 0$.
This establishes \eqref{eq:Lyap-avg} for large enough $t = t_{n}$, completing our proof.
\end{proof}


\bibliographystyle{apalike}
\bibliography{Bibliography-GD}

\end{document}